\documentclass[reqno,11pt]{amsart}
\usepackage{amssymb}
\usepackage{graphicx}
\usepackage{amsmath}
\usepackage{mathbbol}
\usepackage[margin=1.23in]{geometry}

\usepackage[usenames, dvipsnames]{color}
\usepackage{hyperref}
\usepackage{verbatim}
\usepackage{mathrsfs}
\usepackage{bm}
\usepackage{color}
\usepackage{esint}
\allowdisplaybreaks[4]
\usepackage{dutchcal}

\numberwithin{equation}{section}

\newtheorem{theorem}{Theorem}[section]

\newtheorem{lemma}[theorem]{Lemma}
\newtheorem{prop}[theorem]{Proposition}

\theoremstyle{definition}
\newtheorem{remark}[theorem]{Remark}

\theoremstyle{definition}

\theoremstyle{definition}

\makeatletter
\def\dashint{\operatorname%
	{\,\,\text{\bf-}\kern-.98em\DOTSI\intop\ilimits@\!\!}}
\makeatother

\def\\det{\text{\det}}

\def\.5{\frac{1}{2}}

\newcommand{\RN}[1]{%
	\textup{\uppercase\expandafter{\romannumeral#1}}%
}

\newcommand{\dist}{\text{dist}}

\renewcommand{\epsilon}{\varepsilon}

\newcounter{marnote}

\begin{document}
	
	\title[Optimal higher derivative estimates]{Optimal higher derivative estimates for Stokes equations with closely spaced rigid inclusions}
	
	\author[H.J. Dong]{Hongjie Dong}
	\address[H.J. Dong]{Division of Applied Mathematics, Brown University, 182 George Street, Providence, RI 02912, USA}
	\email{Hongjie\_Dong@brown.edu}
	\thanks{H. Dong was partially supported by the NSF under agreement DMS-2350129.}
	
	\author[H.G. Li]{Haigang Li}
	\address[H.G. Li]{School of Mathematical Sciences, Beijing Normal University, Laboratory of Mathematics and Complex Systems, Ministry of Education, Beijing 100875, China.}
	\email{hgli@bnu.edu.cn}
	\thanks{H. Li was partially Supported by Beijing Natural Science Foundation (No. 1242006), National Natural Science Foundation of China (No. 12471191), and the Fundamental Research Funds for the Central Universities (No. 2233200015).}
	
	\author[H.J. Teng]{Huaijun Teng}
	\address[H.J. Teng]{School of Mathematical Sciences, Beijing Normal University, Laboratory of Mathematics and Complex Systems, Ministry of Education, Beijing 100875, China.}
	\email{hjt2021@mail.bnu.edu.cn}
	
	\author[P.H. Zhang]{PeiHao Zhang}
	\address[P.H. Zhang]{School of Mathematical Sciences, Beijing Normal University, Laboratory of Mathematics and Complex Systems, Ministry of Education, Beijing 100875, China.}
	\email{phzhang@mail.bnu.edu.cn}


	\date{\today} 
	
	

	\begin{abstract}
		In this paper, we study the interaction between two closely spaced rigid inclusions suspended in a Stokes flow. It is well known that the stress significantly amplifies in the narrow region between the inclusions as the distance between them approaches zero. To gain deeper insight into these interactions, we derive high-order derivative estimates for the Stokes equation in the presence of two rigid inclusions in two dimensions. Our approach resonates with the method used to handle the incompressibility constraint in the standard convex integration scheme. Under certain symmetric assumptions on the domain, these estimates are shown to be optimal. As a result, we establish the precise blow-up rates of the Cauchy stress and its higher-order derivatives in the narrow region.
	\end{abstract}
	
	\maketitle

	\section{Introduction}
	
	Let \( D \subset \mathbb{R}^2 \) be a smooth, bounded domain containing two rigid inclusions (particles), denoted by \( D_1 \subset D \) and \( D_2 \subset D \). These inclusions are separated by a distance of \( \varepsilon \) and are positioned away from the outer boundary \( \partial D \). Specifically,
	\begin{equation*}
		\begin{split}
			\overline{D}_{1},\overline{D}_{2}\subset D,\quad
			\varepsilon:=\mbox{dist}(D_{1},D_{2})>0,\quad\mbox{dist}(D_{1}\cup D_{2},\partial D)>\kappa_{0}>0,
		\end{split}
	\end{equation*}
	where $\kappa_{0}$ is a constant independent of $\varepsilon$. For a fixed integer $m\ge0$, we assume that the $C^{m+1,\gamma}$ norms of $\partial{D}_{1}$ and $\partial{D}_{2}$ are bounded by some positive constant, independent of $\varepsilon$. To study the interaction between adjacent inclusions, we consider the following Stokes equations
	\begin{align}\label{maineqs}
		\begin{cases}
			\mu\Delta{\bf u}=\nabla p,\quad\nabla\cdot{\bf u}=0\quad&\hbox{in}\ 
			\Omega:=D\setminus\overline{D_1\cup D_2 },\\
			{\bf u}|_{+}={\bf u}|_{-},&\hbox{on}\ \partial{D}_{i},\,i=1,2,\\
			e({\bf u})=0,&\hbox{in}~~D_{i},\,i=1,2,\\ 
			\int_{\partial{D}_{i}}\boldsymbol{\psi}_{\alpha}\cdot\frac{\partial {\bf u}}{\partial \nu}\Big|_{+}
			\,ds-\int_{\partial{D}_{i}}p\boldsymbol{\psi}
			_{\alpha}\cdot\nu\ ds =0,&i=1,2,\,\alpha=1,2,3,\\    
			{\bf u}=\boldsymbol{\varphi},&\hbox{on}\ \partial{D},
		\end{cases}
	\end{align}
	where $\mu>0$, $e({\bf u})=\frac{1}{2}(\nabla{\bf u}+(\nabla{\bf u})^{T})$ is the strain tensor, $\frac{\partial{\bf u}}
	{\partial\nu}|_{+}:=\mu(\nabla{\bf u}+(\nabla{\bf u})^{T})\nu$,  $\nu$ is 
	the unit outer normal vector of $D_{i},~i=1,2$, and $\boldsymbol{\psi}_1=(1,0)^{T}, \boldsymbol{\psi}_2=(0,1)^{T}, \boldsymbol{\psi}_3=(x_{2},-x_{1})^{T}$ represent the rigid motions in dimension two. Throughout this paper, the subscript \( \pm \) denotes the limit from the outside and inside of the inclusions, respectively. Since \( D \) is bounded, the divergence-free condition \( \nabla \cdot \mathbf{u} = 0 \), together with the Gauss theorem, implies that the prescribed velocity field \( \boldsymbol{\varphi} \) must satisfy the compatibility condition \( \int_{\partial D} \boldsymbol{\varphi} \cdot \nu = 0 \) to ensure the existence and uniqueness of the solution \cite{Lady}. For a comprehensive study of the steady Stokes equations, we refer the interested reader to the monograph by Galdi \cite{GaldiBook}.
	
	The mathematical theory governing the motion of solid particles in a fluid is a classical problem in fluid mechanics, with origins tracing back to the foundational works of Stokes \cite{Stokes}, Kirchhoff \cite{Kirchhoff}, and Jeffery \cite{Jeffery}. Understanding the interaction mechanisms between particles is critical for practical applications involving particulate flows \cite{DDj}. Under low Reynolds number conditions, the fluid motion is governed by the Stokes equations. In the specific case of two-dimensional incompressible Stokes flow with two circular inclusions separated by a small distance \( \varepsilon \), Ammari, Kang, Kim, and Yu \cite{AKKY} achieved significant progress by employing bipolar coordinates. They derived an asymptotic formula for the concentrated stress, fully capturing the singular behavior of the stress and demonstrating a blow-up rate of \( \varepsilon^{-1/2} \). However, their methods have certain limitations. As highlighted by Kang in his ICM presentation \cite{K}, extending these techniques to inclusions of general shapes and higher-dimensional settings remains both an interesting and challenging task. Subsequently, Li and Xu \cite{LX1, LX2} advanced the field by addressing the challenges posed by pressure term and the divergence-free condition through an iterative approach based on energy estimates. They established that the optimal blow-up rate of the gradient is \( \varepsilon^{-1/2} \) in two dimensions and \( (\varepsilon \ln \varepsilon)^{-1} \) in three dimensions for inclusions with strictly convex boundaries near the narrow gap region. Furthermore, they derived estimates for higher-dimensional scenarios and second-order derivatives. It is worth noting that these optimal blow up rates are consistent with those found in the perfect conductivity problem \cite{AKL,AKL3,BLY,KangLY} and the linear elasticity problem with hard inclusions \cite{BLL,BLL2,KY,Li2021}. There is also a wealth of literature on elliptic equations and systems related to the (insulated) conductivity problem and linear elasticity problem. For example, references include Bonnetier and Vogelius \cite{BV}, Li and Vogelius \cite{LV}, Li and Nirenberg \cite{LN}, Bonnetier and Triki \cite{BT1,BT2}, Li and Yang \cite{LiYang}, Weinkove \cite{Ben}, and Dong, Li, and Yang \cite{DLY1}. Equation \eqref{maineqs} can also be viewed as a Stokes equation with discontinuous coefficients, where the viscosity coefficient is infinite in the inclusions. For Stokes equations with general discontinuous coefficients and their applications, we refer the reader to \cite{LIL07, DK18, KW18, CK19, KMW21} and the references therein.
	
	To effectively analyze the singular behavior of solutions, as well as to develop accurate numerical schemes, it is crucial to obtain higher-order derivative estimates—both from an engineering perspective and for the requirements of numerical experiments. Through the asymptotic expansion of the gradient, the higher-order derivatives of the solution to \eqref{maineqs} may also exhibit blow-up in the narrow region between the two inclusions. The optimal blow-up rate, if it exists, remains unknown. Therefore, establishing estimates for the higher-order derivatives of the solution and determining their blow-up rates is essential. In this paper, we provide a definitive answer to this question in two dimensions, precisely quantifying the singular behavior of the high-order derivatives for the Stokes system with closely spaced rigid inclusions.
	
	Our approach involves constructing a series of auxiliary functions and applying an energy iteration technique to demonstrate that these auxiliary functions effectively capture the singular behavior of the high-order derivatives of the solution. In doing so, we refine the energy iteration method introduced in \cite{LX1, LX2}, reducing the problem of estimating higher-order derivatives to analyzing the inhomogeneous terms. The primary challenge lies in the divergence-free condition of the Stokes system. In this regard, our approach aligns with the method used to handle the incompressibility constraint in the standard convex integration scheme (cf. \cite{CHL21} and references therein).
	
	It is worth noting that the energy iteration method was originally developed to address stress concentration issues in high-contrast elastic composite materials, a problem that traces back to the seminal work of Babuška \cite{Bab}, often referred to as the Babuška problem. By applying this method, it has been shown that the optimal blow-up rates of the gradient are \( \varepsilon^{-1/2} \) in two dimensions and \( (\varepsilon \ln \varepsilon)^{-1} \) in three dimensions \cite{BLL, BLL2, Li2021,LX24} as mentioned above. Kang and Yu \cite{KY} further demonstrated that the blow-up rate \( \varepsilon^{-1/2} \) is indeed optimal in two-dimensional cases by utilizing layer potential techniques. 
	More recently, we advanced the energy iteration method to establish that the optimal blow-up rates of the \( m \)-th order derivatives are \( \varepsilon^{-m/2} \) in two dimensions and \( |\ln \varepsilon|^{-1} \varepsilon^{-(m+1)/2} \) in three dimensions, as shown in \cite{DLTZ}. Additionally, we proved that the upper bound for the blow-up rate is \( \varepsilon^{-(m+1)/2} \) in higher dimensions. Based on these estimates for higher-order derivatives, a second-order vectorial finite element method is developed in \cite{LLL} to numerically simulate the stress concentration between close-to-touching inclusions. While there are technical connections between the work presented in this paper and previous studies on the Lamé system \cite{DLTZ}, much of the material here is novel, primarily due to the challenges introduced by the pressure terms and the divergence-free condition in the Stokes system.
	
	Before presenting our main results, we first provide the domain setup and introduce the relevant notation. To simplify matters, we translate and rotate the coordinate system such that there are two points: \( P_1 = (0, {\varepsilon}/{2}) \in \partial D_1 \) and \( P_2 = (0, -{\varepsilon}/{2}) \in \partial D_2 \), satisfying \( \dist(P_1, P_2) = \dist(\partial D_1, \partial D_2) = \varepsilon \). There exists a constant \( R \), independent of \( \varepsilon \), such that the portions of \( \partial D_1 \) and \( \partial D_2 \) near the origin can be expressed, respectively, as graphs:  
	\[
	x_2 = \frac{\varepsilon}{2} + h_1(x_1) \quad \text{and} \quad x_2 = -\frac{\varepsilon}{2} - h_2(x_1) \quad \text{for } |x_1| \leq 2R.
	\]
	Here, \( h_1, h_2 \in C^{m+1, \gamma}(B_{2R}(0)) \) and satisfy the following conditions:
	\begin{equation}        \label{h1h14}
	     \begin{aligned}
		&-\frac{\varepsilon}{2} - h_2(x_1) < \frac{\varepsilon}{2} + h_1(x_1), \quad h_1(x_1) + h_2(x_1) \geq \kappa |x_1|^2 \quad \text{for } |x_1| \leq 2R,\\
		&h_1(0) = h_2(0) = 0, \quad h_1'(0) = h_2'(0) = 0, \\
		&|h_i(x_1)| \leq C|x_1|^2, \quad |h_i'(x_1)| \leq C|x_1|, \quad |h_i^{(k)}(x_1)| \leq C, \quad 2 \leq k \leq m+1,~ i = 1, 2,
	\end{aligned}
	\end{equation}
	where \( \kappa > 0 \) is a constant. 
	
	For \( 0 \leq r \leq 2R \), we define the {\em neck region} between \( D_1 \) and \( D_2 \) as:  
	\[
	\Omega_r := \left\{ (x_1, x_2) \in \Omega : -\frac{\varepsilon}{2} - h_2(x_1) < x_2 < \frac{\varepsilon}{2} + h_1(x_1), \, |x_1| < r \right\}.
	\]
	The top and bottom boundaries of the neck region are denoted by:
	\begin{align*}
	\Gamma^+_r &:= \left\{ (x_1, x_2) \in \Omega : x_2 = \frac{\varepsilon}{2} + h_1(x_1), \, |x_1| < r \right\},\\
	\Gamma^-_r &:= \left\{ (x_1, x_2) \in \Omega : x_2 = -\frac{\varepsilon}{2} - h_2(x_1), \, |x_1| < r \right\}.    
	\end{align*}
	Throughout this paper, we use \( C \) to denote a universal constant, meaning that the value of \( C \) may vary from line to line. However, \( C \) depends only on \( \kappa_0 \), \( \kappa \), \( R \), \(\boldsymbol{\varphi}\), and the upper bounds of the \( C^{m+1} \)-norms of \( \partial D_1 \) and \( \partial D_2 \), but is independent of \( \varepsilon \).
	
	\subsection{Upper Bounds of the Higher Derivatives}
	
	By the standard theory for Stokes systems, we have the boundedness,  for $m\ge 0$, 
	$$
	\|\nabla^{m+1} {\bf u}\|_{L^{\infty}(\Omega\setminus\Omega_{R})}+\|\nabla^{m} (p-p(z))\|_{L^{\infty}(\Omega\setminus\Omega_{R})}\leq\,C
	$$ 
	for some fixed point $z\in\Omega\setminus\Omega_{R}$. 
	
	Our first result gives the higher derivative estimates for problem \eqref{maineqs} in the narrow region $\Omega_{R}$ in two dimensions.
	
	\begin{theorem}\label{main thm1n}
		Assume that \(D_{1}, D_{2}, D, \Omega\) are defined as above. Let \({\bf u} \in H^{1}(D; \mathbb{R}^2) \cap C^{m+1}(\bar{\Omega}; \mathbb{R}^2)\) and \(p \in L^2(D) \cap C^{m}(\bar{\Omega})\) be the solution to \eqref{maineqs}, with \(\boldsymbol{\varphi} \in C^{m+1,\alpha}(\partial D; \mathbb{R}^2)\), for an integer \(m \geq 0\) and \(0 < \alpha < 1\). Then, for sufficiently small \(0 < \varepsilon < 1/2\), the following estimates hold:
		
		\begin{equation}\label{mainest1n}
			\text{(i)} \quad |\nabla {\bf u}(x)| \leq \frac{C\sqrt{\varepsilon}}{\varepsilon + |x_1|^2}+C, \quad |p(x) - p(z_1, 0)| \leq \frac{C\sqrt{\varepsilon}}{(\varepsilon + |x_1|^2)^{3/2}}+C, \quad 
			x = (x_1, x_2) \in \Omega_{R},
		\end{equation}
		for some point \(z = (z_1, 0) \in \Omega_{R}\) with \(|z_1| = R/2\); and
		
		(ii) for \(m \geq 1\),
		\begin{equation}\label{mainest2n}
			|\nabla^{m+1} {\bf u}(x)| + |\nabla^m p(x)| \leq \frac{C\sqrt{\varepsilon}}{(\varepsilon + |x_1|^2)^{\frac{m+3}{2}}}+C, \quad 
			x = (x_1, x_2) \in \Omega_{R}.
		\end{equation}
	\end{theorem}

	\begin{remark}  
		While for \( m = 0, 1 \), the estimates \eqref{mainest1n} and \eqref{mainest2n} are consistent with the previous results obtained in \cite{LX1}, where \( h_1(x_1) = h_2(x_1) = \frac{1}{2}|x_1|^2 \), the general case of \( h_1(x_1) \) and \( h_2(x_1) \) presents significantly greater challenges. The construction of auxiliary functions becomes considerably more intricate, particularly due to the asymmetry between \( h_1 \) and \( h_2 \). In this paper, we have refined these constructions, allowing for a complete characterization of the singular behavior of derivatives of all orders.  
	\end{remark}
	
	If \( D_1 \cup D_2 \) and \( D \) are symmetric with respect to both axes and \( \boldsymbol{\varphi}(-x) = -\boldsymbol{\varphi}(x) \), more precise upper bounds for the derivatives can be derived, reflecting a reduced singularity of half an order. In this context, we further assume that \( h_1(x_1) = h_2(x_1) = \frac{1}{2}h(x_1) \) in \( \Omega_{2R} \).
	
	\begin{theorem}\label{main thm1}
		Assume that \(D_{1}, D_{2}, D, \Omega\) are defined as above. Let \({\bf u} \in H^{1}(D; \mathbb{R}^2) \cap C^{m+1}(\bar{\Omega}; \mathbb{R}^2)\) and \(p \in L^2(D) \cap C^{m}(\bar{\Omega})\) be the solution to \eqref{maineqs}, where \(\boldsymbol{\varphi} \in C^{m+1,\alpha}(\partial D; \mathbb{R}^2)\), for an integer \(m \geq 0\) and \(0 < \alpha < 1\), and \(\boldsymbol{\varphi}(-x) = -\boldsymbol{\varphi}(x)\). Then, for sufficiently small \(0 < \varepsilon < 1/2\), the following estimate holds for any \(m \geq 0\):
		\begin{equation*}
			|\nabla^{m+1}{\bf u}(x)| + |\nabla^m (p(x) - p(z_1, 0))| \leq \frac{C\sqrt{\varepsilon}}{(\varepsilon + |x_1|^2)^{\frac{m+2}{2}}}+C \quad 
			\text{for}~x = (x_1, x_2) \in \Omega_{R},
		\end{equation*}
		where \(z = (z_1, 0) \in \Omega_{R}\) with \(|z_1| = R/2\).
		
		As a consequence, for the Cauchy stress tensor \(\sigma[{\bf u}, p] = 2\mu e({\bf u}) - p\mathbb{I}\), where \(\mathbb{I}\) is the identity matrix, we have
		\begin{equation*}
			\big|\nabla^m \sigma[{\bf u}, p - p(z_1, 0)]\big| \leq \frac{C\sqrt{\varepsilon}}{(\varepsilon + |x_1|^2)^{\frac{m+2}{2}}}+C \quad 
			\text{for}~x = (x_1, x_2) \in \Omega_{R}.
		\end{equation*}
	\end{theorem}
	
	\begin{remark}
		The upper bounds for higher derivatives in Theorem \ref{main thm1} are proved to be sharp in Theorem \ref{thmlowerbound}.
	\end{remark}

	\subsection{Lower Bounds of the Higher Derivatives}
	To establish the optimality of the blow-up rates derived in Theorem \ref{main thm1}, we will establish a lower bound for \( |\nabla^{m+1} {\bf u}(x)| \) that exhibits the same blow-up rate as identified in the theorem. This is done under the assumption that \( D_1 \cup D_2 \) and \( D \) are symmetric with respect to both the \( x_1 \)-axis and the \( x_2 \)-axis. Additionally, we assume that \( h_1(x_1) \) and \( h_2(x_1) \) are quadratic and symmetric with respect to the plane \( \{x_2 = 0\} \). Specifically, we take \( h_1(x_1) = h_2(x_1) = \frac{1}{2} |x_1|^2 \) for \( |x_1| \leq 2R \). We define a linear and continuous functional of ${\boldsymbol\varphi}$:
	\begin{equation*}
		\tilde b_{j}^{*\alpha}[{\boldsymbol\varphi}]:=
		\int_{\partial D_j^0}{\boldsymbol\psi}_\alpha\cdot\sigma[{\bf 
			u}^*,p^*]\nu,\quad\alpha=1,2,3,~j=1,2,
	\end{equation*}
	where $({\bf u}^*,p^*)$ verify
	\begin{align*}
		\begin{cases}
			\mu\Delta{\bf u}^{*}=\nabla p^{*},\quad\nabla\cdot {\bf 
				u}^{*}=0\quad&\hbox{in}\ \Omega^{0},\\
			{\bf 
				u}^{*}=\sum_{\alpha=1}^{3}C_{*}^{\alpha}{\boldsymbol\psi}_{\alpha}&\hbox{on}\
			\partial D_{1}^{0}\cup\partial D_{2}^{0},\\
			\int_{\partial{D}_{1}^{0}}{\boldsymbol\psi}_\alpha\cdot\sigma[{\bf 
				u}^*,p^{*}]\nu+\int_{\partial{D}_{2}^{0}}{\boldsymbol\psi}_\alpha\cdot\sigma[{\bf
				u}^*,p^{*}]\nu=0,&\alpha=1,2,3,\\
			{\bf u}^{*}={\boldsymbol\varphi}&\hbox{on}\ \partial{D},
		\end{cases}
	\end{align*}
	where $D_{1}^{0}:=\{x\in\mathbb R^{2}~\big|~ x+P_{1}\in D_{1}\},\quad~ 
	D_{2}^{0}:=\{x\in\mathbb R^{2}~\big|~ x+P_{2}\in D_{2}\}$, 
	$\Omega^{0}:=D\setminus\overline{D_{1}^{0}\cup D_{2}^{0}}$ and the constants 
	$C_{*}^{\alpha}$, $\alpha=1,2,3$, are uniquely determined by the solution 
	$({\bf u}^*,p^*)$. 
	\begin{theorem}\label{thmlowerbound}
		Let $D_1,D_2\subset D$ be defined as above and let ${\bf u}\in 
		H^{1}(D;\mathbb R^{2})\cap C^{m}(\bar{\Omega};\mathbb R^{2})$ be a solution 
		to \eqref{maineqs}. Suppose that there exists a $\boldsymbol{\varphi}$ such that 
		$\boldsymbol{\varphi}(-x)=-\boldsymbol{\varphi}(x)$ and 
		$b_{1}^{*1}[\boldsymbol{\varphi}]\neq 0$. Then for $m\ge 0$ and 
		sufficiently small $0<\varepsilon<1/2$, there exists a small constant $r>0$ 
		which may depend on $m$, such that
		\begin{align}\label{lower}
			|\partial_{x_1}^{m}\partial_{x_{2}}{\bf 
				u}^{(1)}(r\sqrt{\varepsilon},0)|\ge 
			C|b_{1}^{*1}[\boldsymbol{\varphi}]|\varepsilon^{-\frac{m+1}{2}}, 
		\end{align}
		and consequently,
		$$
		\big|\partial_{x_1}^m\sigma[{\bf u},p-p(z_1,0)]\big|(r\sqrt{\varepsilon},0)\geq 
		C|\tilde b_{1}^{*1}[{\boldsymbol\varphi}]|\varepsilon^{-\frac{m+1}{2}}.
		$$
		
	\end{theorem}
	
	\begin{remark}
		The optimality of the gradient estimates in Theorem \ref{main thm1n} is established through \eqref{lower} with \(m=0\). However, for general functions \(h_{1}(x_{1})\) and \(h_{2}(x_{1})\), the optimality of the high-order derivative estimates in Theorem \ref{main thm1n} remains uncertain. We conjecture that the asymmetry \(h_{1}(x_{1}) \neq h_{2}(x_{1})\) could introduce additional singularities. Investigating and classifying these new singularities would be both theoretically interesting and practically useful for applications.
	\end{remark}

It is worth noting that our approach breaks down in a few places in three dimensions. In particular, the construction of auxiliary functions such as $F$ and $G$ in the proof of Proposition \ref{prop2.1} cannot be easily extended to the three dimensional case. We are going to address this in a future work.

	The remainder of this paper is organized as follows: In Section \ref{sec_pri}, we establish a structured framework for the energy method. Specifically, we introduce an auxiliary function pair \(({\bf v},\bar{p})\) that satisfies the same boundary conditions as \(({\bf u},p)\) on \(\Gamma^{\pm}_{2R}\), with \(\nabla \cdot {\bf v} = 0\) in \(\Omega_{2R}\). This allows us to reduce the local estimates of the difference \( |\nabla^{m+1}({\bf u} - {\bf v})| + |\nabla^m(p - \bar{p})| \) to the control of terms on the right-hand side of the equations, denoted by \({\bf f} := \mu \Delta {\bf v}(x) - \nabla \bar{p}(x)\). By Proposition \ref{prop3.3}, it suffices to construct the auxiliary function pair \(({\bf v}, \bar{p})\) to satisfy the conditions on \({\bf f}\) and \(\nabla^k {\bf f}\) for an explicit boundary value problem.
	
	This process is carried out in Sections \ref{sec_u1}, \ref{sec_u2}, and \ref{sec_u3}, for the Stokes system with specified Dirichlet boundary data \(\boldsymbol{\psi}_{\alpha}\) on \(\partial{D}_{1}\) and \(\partial{D}_{2}\), where \(\alpha = 1, 2, 3\) (see \eqref{u,peq1} below). The main method for establishing estimates for the high-order derivatives of the solution is explored in detail. For each \(\alpha\), we construct a sequence of auxiliary functions \({\bf v}_\alpha^m\) and corresponding \(\bar{p}_m\) in \(\Omega_{2R}\), ensuring that the quantity 
	\[
	\bigg|\sum_{l=1}^{m+1} (\mu \Delta {\bf v}_{\alpha}^{l} - \nabla \bar{p}_{l})(x)\bigg|
	\] 
	satisfies a suitable pointwise upper bound in \(\Omega_{2R}\), as required by Proposition \ref{prop3.3}. We then demonstrate that \(\nabla^{m+1}\sum_{l=1}^{m+1} {\bf v}_\alpha^l\) and \(\nabla^m \sum_{l=1}^{m+1} \bar{p}_l\) effectively capture all the singularities of \(\nabla^{m+1} {\bf u}_{1}^{\alpha}\) and \(\nabla^m p_1^{\alpha}\), up to an \(O(1)\) term.
	
	In Section \ref{sec3}, we prove Theorems \ref{main thm1n} and \ref{main thm1} by utilizing the estimates established in Sections \ref{sec_u1}, \ref{sec_u2}, and \ref{sec_u3}. Finally, we prove Theorem \ref{thmlowerbound}, confirming the optimality of the upper bounds by deriving a lower bound under certain symmetry conditions on the domain and boundary data.
	
	
	\section{Preliminary results}\label{sec_pri}
	
	The objective of this section is to establish estimates for the higher-order derivatives of the solution to the Stokes equation with the right-hand side \({\bf f}\). The results presented here apply in all dimensions \(d \geq 2\). Let \(D\) be a domain in \(\mathbb{R}^d\), with two closely positioned subdomains \(D_1\) and \(D_2\), and denote \(\Omega = D \setminus \overline{D_1 \cup D_2}\). We consider the Stokes system
	\begin{equation}\label{u1}
		\begin{cases}
			-\mu\Delta{\bf u}+\nabla p=0\quad&\mathrm{in}\,\,\Omega,\\
			\nabla\cdot {\bf u}=0\quad&\mathrm{in}\,\,\Omega,\\
			{\bf u}={\bf g}\quad&\mathrm{on}\,\,\partial \Omega,
		\end{cases}
	\end{equation}
	with given smooth boundary data \({\bf g}\) satisfying  \( \int_{\partial \Omega} {\bf g} \cdot \nu = 0 \). The existence and uniqueness of the solution to \eqref{u1} are well known. By standard regularity theory, we have
	\[
	\|{\bf u}\|_{C^{m+1}(\Omega \setminus \Omega_R)} + \|p\|_{C^m(\Omega \setminus \Omega_R)} \leq C.
	\]
	Let \(({ \bf v}, \bar{p})\) be a known function pair satisfying \({\bf v} = {\bf u} = {\bf g}\) on \(\Gamma^+_{2R} \cup \Gamma^-_{2R}\) and \(\nabla \cdot {\bf v} = 0\) in \(\Omega_{2R}\), with
	\[
	\|{\bf v}\|_{C^{m+1}(\Omega \setminus \Omega_R)} + \|\bar{p}\|_{C^m(\Omega \setminus \Omega_R)} \leq C.
	\]
	Define the difference between \(({ \bf u}, p)\) and \(({ \bf v}, \bar{p})\) as
	\[
	{\bf w} := {\bf u} - {\bf v} \quad \text{and} \quad q := p - \bar{p}.
	\]
	Then, \(({ \bf w}, q)\) satisfies the following boundary value problem for the nonhomogeneous Stokes equations in the narrow region:
	\begin{equation}\label{w1}
		\begin{cases}
			-\mu\Delta{\bf w}+\nabla q={\bf f}\quad&\mathrm{in}\,\,\Omega_{2R},\\
			\nabla\cdot {\bf w}=0\quad&\mathrm{in}\,\,\Omega_{2R},\\
			{\bf w}=0\quad&\mathrm{on}\,\,\Gamma^+_{2R}\cup\Gamma^-_{2R},
		\end{cases}
	\end{equation}
	where \({\bf f} := \mu \Delta {\bf v} - \nabla \bar{p}\). Denote the vertical distance at \(x_1\) in the narrow region between \(D_1\) and \(D_2\) by
	\[
	\delta(x_1) := \varepsilon + h_1(x_1) + h_2(x_1), \quad |x_1| \leq 2R.
	\]
    
	We have the following higher derivative estimates for the solution to \eqref{w1}.
	
	\begin{prop}\label{prop3.3}
		Let ${\bf w}\in H^{1}(\Omega_{\frac{3}{2}R})\cap L^{\infty}(\Omega_{2R}
		\setminus\Omega_{R})$ and $p\in L^{\infty}(\Omega_{2R}\setminus\Omega_{R})$ be 
		the solution to \eqref{w1}. For any $m\ge1$, if
		\begin{equation*}
			|\nabla^k{\bf f}(x)|\leq 
			C\delta(x')^{l-k}, \quad\forall\,0\le k\leq m,~l\ge -\frac{3}{2}\quad\mbox{for}
			~x=(x',x_d)\in\Omega_{2R},
		\end{equation*}
		then it holds that
		\begin{equation*}
			\|\nabla{\bf w}\|_{L^{\infty}(\Omega_{\delta(z')/2}(z'))}\leq 
			C\delta(z')^{l+1},\quad z=(z',z_d) \in \Omega_{R},
		\end{equation*}
		and
		\begin{equation*}
			\|\nabla^{m+1}{\bf w}\|_{L^{\infty}(\Omega_{\delta(z')/2}(z'))} +\|
			\nabla^{m}q\|
			_{L^{\infty}(\Omega_{\delta(z')/2}(z'))}\le C\delta(z')^{l+1-m},\quad z =(z',z_d) 
			\in \Omega_{R}.
		\end{equation*}
	\end{prop}
	
	From Proposition \ref{prop3.3}, we reduce the estimates of the higher derivatives of \({\bf w}\) and \(q\) to continuously improving the estimates of \({\bf f}\) and its derivatives. Thus, to obtain accurate estimates of the higher-order derivatives of \({\bf u}\) and \(p\), it suffices to construct appropriate auxiliary functions \({\bf v}\) and \(\bar{p}\) in the region \(\Omega_{2R}\), such that \({\bf f} = \mu \Delta {\bf v} - \nabla \bar{p}\) and its derivatives have suitable upper bounds. We will apply this approach to prove Theorems \ref{main thm1n} and \ref{main thm1}.
	
	To prove Proposition \ref{prop3.3}, we need the following proposition, which is from \cite[Proposition 3.6]{LX1}.
	\begin{prop}\label{lem3.1} (\cite{LX1})
		Let $({\bf w},q)$ be the solution to \eqref{w1}. Then for $z=(z',z_d)
		\in\Omega_R$,  the following estimates hold:
		\begin{equation*}
			\|\nabla{\bf w}\|_{L^{\infty}(\Omega_{\delta(z')/2}(z')) }\le C 
			\Big(\delta(z')^{-\frac{d}{2}}\|\nabla{\bf w}\|_{L^2(\Omega_{\delta(z')}(z'))} +
			\delta(z')\|{\bf f}\|_{L^{\infty}(\Omega_{\delta(z')}(z'))} \Big),
		\end{equation*}
		and for $m\ge1$,
		\begin{align*}
			&\|\nabla^{m+1}{\bf w}\|_{L^{\infty}(\Omega_{\delta(z')/2}(z'))}+\|
			\nabla^mq\|
			_{L^{\infty}(\Omega_{\delta(z')/2}(z'))} \nonumber\\
			\le&\, C{\delta(z')^{-m-1}}\Big(\delta(z')^{1-\frac{d}{2}}\|
			\nabla{\bf 
				w}\|_{L^2(\Omega_{\delta(z')}(z'))}+\sum_{j=0}^{m}\delta(z')^{2+j}\|
			\nabla^{j}
			{\bf f}\|_{L^{\infty}(\Omega_{\delta(z')}(z'))}\Big).
		\end{align*}
	\end{prop}
	
	From Proposition \ref{lem3.1}, the $W^{m+1,\infty}$ estimates of ${\bf 
		w}$ depends on the local energy estimates $\|\nabla {\bf w}\|_{L^{2}(\Omega_{\delta(z')}
		(z'))}$ and the pointwise estimates of $\nabla^{j}{\bf f}$, $j=0,1,\dots,m$, in 
	the narrow region. Indeed, we are able to demonstrate that the local energy estimates $\|
	\nabla {\bf w}\|_{L^{2}(\Omega_{\delta(z')}(z'))}$ are also closely related to the pointwise upper 
	bound of ${\bf f}$ itself in the narrow region. To this end, we first prove that the global 
	energy of ${\bf w}$ is bounded in the domain $\Omega_{\frac{3}{2}R}$, under certain 
	assumptions on ${\bf f}$.

	\begin{lemma}\label{lemmaenergy}
		Let $({\bf w},q)$ be the solution to \eqref{w1}. If 
		\begin{equation}\label{cond_f}
			|{\bf f}(x)|\leq C\delta(x')^{-3/2}\quad\mbox{in}~ \Omega_{2R},
			\quad\mbox{and} \quad \|q\|_{L^\infty(\Omega_{2R}\setminus\Omega_{R})},~\|{\bf w}\|
			_{L^\infty(\Omega_{2R}\setminus\Omega_{R})}\leq C,
		\end{equation}
		then
		\begin{equation*}
			\int_{\Omega_{\frac{3}{2}R}}|\nabla {\bf w}|^{2}\leq C.
		\end{equation*}
	\end{lemma}
	
	\begin{proof}
		Let $\eta$ be a smooth function satisfying $\eta(x')=1$ if $|x'|<\frac{3}{2}R$, $
		\eta(x')=0$ if $|x'|>2R$, $0\leqslant\eta(x')\leqslant1$ if $\frac{3}{2}R\le |x'|\le 
		2R$, and $|\nabla_{x'}\eta(x')|\leq C$. Multiplying the first equation in \eqref{w1} 
		by ${\bf w}\eta^2$ and integrating by parts, we get
		\begin{align}\label{zy1}
			\mu\int_{\Omega_{2R}}\eta^2|\nabla{\bf w}|^2\leq \mu\int_{\Omega_{2R}}|
			\nabla_{x'}\eta||\eta\nabla{\bf w}||{\bf w}|+C\int_{\Omega_{2R}}\big|
			q\nabla\cdot({\bf w}\eta^2)\big|+\Big|\int_{\Omega_{2R}}{\bf f}\cdot{\bf w}
			\eta^2\Big|.
		\end{align}
		By virtue of the Cauchy inequality and \eqref{cond_f}, we have
		\begin{equation}\label{zy2}
			\int_{\Omega_{2R}}|\nabla_{x'}\eta||\eta\nabla{\bf w}||{\bf w}|\leq 
			\frac{1}{2}\int_{\Omega_{2R}}\eta^2|\nabla{\bf w}|^2+C\int_{\Omega_{2R}
				\setminus\Omega_{R}}|{\bf w}|^2\leq \frac{1}{2}\int_{\Omega_{2R}}\eta^2|
			\nabla{\bf w}|^2+C.
		\end{equation}
		For the second term on the right-hand side of \eqref{zy1}, recalling that $
		\nabla\cdot{\bf w}=0$ in $\Omega_{2R}$, and by \eqref{cond_f}, we derive
		\begin{equation}\label{zy3}
			\int_{\Omega_{2R}}\big|q\nabla\cdot({\bf w}\eta^2)\big|\leq 
			\int_{\Omega_{2R}\setminus \Omega_{R}}|q||{\bf w}|\leq C.
		\end{equation}
		It follows from \eqref{cond_f} that
		\begin{align*}
			\|\delta(x') {\bf f}\|_{L^2(\Omega_{\frac{3}{2}R})}\leq C\Big(\int_{\Omega_{3R/
					2}}\frac{1}{\delta(x')}\Big)^{1/2}\leq C.
		\end{align*}
		Thus, for the last term, by using \eqref{cond_f}, H{\"o}lder’s inequality,  
		and Hardy’s inequality, we derive
		\begin{align}\label{zy4}
			\Big|\int_{\Omega_{2R}} {\bf f}\cdot{\bf w}\eta^2\Big|\leq&\, \Big|
			\int_{\Omega_{\frac{3}{2}R}}{\bf w}\cdot {\bf f}\Big|+\Big|\int_{\Omega_{2R}
				\setminus\Omega_{\frac{3}{2}R}}\eta^2{\bf w}\cdot {\bf f}\Big|
			\leq\,\|\delta(x') {\bf f}\|_{L^2(\Omega_{\frac{3}{2}R})}\Big\|\frac{\bf w}
			{\delta(x')}\Big\|_{L^2(\Omega_{\frac{3}{2}R})}+C\nonumber\\
			\leq&\, C\Big(\int_{\Omega_{\frac{3}{2}R}}|\nabla{\bf w}|^2\Big)^{1/2}+C.
		\end{align}
		Substituting \eqref{zy2}--\eqref{zy4} into \eqref{zy1}, we obtain
		\begin{equation*}
			\int_{\Omega_{\frac{3}{2}R}}|\nabla {\bf w}|^2\leq\, C\Big(\int_{\Omega_{\frac{3}{2}R}}|
			\nabla{\bf w}|^2\Big)^{1/2}+C.
		\end{equation*}
		Hence, $\int_{\Omega_{\frac{3}{2}R}}|\nabla {\bf w}|^2\leq C$. We thus complete the 
		proof.
	\end{proof}
	
	By Lemma \ref{lemmaenergy} and utilizing the iteration technique, the following lemma illustrates the relationship between the local energy estimates \(\| \nabla {\bf w}\|_{L^{2}(\Omega_{\delta(z')})(z')}\) and the pointwise estimates of \({\bf f}\) in \(\Omega_{2R}\). The proof follows the spirit of \cite[Lemma 2.4]{LZZ}, and here we present only the key steps for the reader’s convenience.
	\begin{lemma}\label{lem3.2}
		Let ${\bf w}\in H^{1}(\Omega_{\frac{3}{2}R})$ be the solution to \eqref{w1}. If
		\begin{equation}\label{jstj}
			|{\bf f}(x)|\leq C\delta(x')^l, \quad\forall~l\ge -\frac{3}{2}, \quad 
			x\in \Omega_{2R},
		\end{equation}
		then
		\begin{equation*}
			\int_{\Omega_{\delta(z')}(z')}\left|\nabla{\bf w}\right|^2 \le 
			C\delta(z')^{d+2+2l},\quad z \in \Omega_{R}.
		\end{equation*}
	\end{lemma}
	\begin{proof}
		For $|z'|\le R$ and $0<t<s\leq R$, we have the following Caccioppoli-type inequality, which can be found in \cite[Lemma 3.10]{LX1}:
		\begin{equation}\label{iterating1}
			\int_{\Omega_{t}(z')}|\nabla {\bf w}|^{2}\leq
			\Big(\frac{1}{4}+\frac{C\delta(z')^2}{(s-t)^{2}}\Big)\int_{\Omega_{s}(z')}|\nabla {\bf w}|^2+C\big((s-t)^{2}+\delta(z')^2\big)\int_{\Omega_{s}(z')}|{\bf f}|^{2}.
		\end{equation}
		By using \eqref{jstj}, 
		\begin{align}\label{jbnl1}
			\int_{\Omega_s(z')}|{\bf f}|^2\leq Cs^{d-1}\delta(z')^{2l+1},~0<s\le C\sqrt{\delta(z')}.
		\end{align}
		Denote $E(t):=\int_{\Omega_{t}(z')}\left|\nabla{\bf w}\right|^2$ and substituting \eqref{jbnl1} into \eqref{iterating1} yields
		
		\begin{equation}\label{itera1}
			E(t)\le\Big(\frac{1}{4}+\Big(\frac{c_0\delta(z')}{s-t}\Big)^2\Big)E(s)+C\Big((s-t)^2+\delta(z')^2\Big)s^{d-1}\delta(z')^{2l+1},
		\end{equation}
		where $c_{0}$ is a constant and we fix it now. Let $k_{0}=\left[\frac{1}{8c_{0}\sqrt{\delta(z')}}\right]$ and $t_{i}=\delta(z')+2c_{0}i\delta(z'), i=0,1,2,\dots,k_{0}$. Then taking $s=t_{i+1}$ and $t=t_i$ in \eqref{itera1}, we have
		\begin{align*}
			E(t_{i})\leq \frac{1}{2}E(t_{i+1})+C(i+1)^{d-1}\delta(z')^{2l+d+2}.
		\end{align*}
		After $k_{0}$ iterations, and using Lemma \ref{lemmaenergy}, we obtain
		\begin{align*}
			E(t_0)\leq&\,\Big(\frac{1}{2}\Big)^{k_0}E(t_{k_{0}})+C\delta(z')^{d+2+2l}\sum_{i=0}^{k_0-1}\Big(\frac{1}{2}\Big)^i(i+1)^{d-1}\\
			\leq&\,\Big(\frac{1}{2}\Big)^{k_0}\int_{\Omega_{\frac{3}{2}R}}\left|\nabla{\bf w}\right|^2+C\delta(z')^{d+2+2l} \le C\delta(z')^{d+2+2l},
		\end{align*}
		for sufficiently small $\epsilon$. This completes the proof.
	\end{proof}
	
	\begin{proof}[Proof of Proposition \ref{prop3.3}]
		Combining Lemma \ref{lem3.2} and Proposition \ref{lem3.1}, we finish the proof of Proposition \ref{prop3.3}.
	\end{proof}
	
	Denote
	$${\boldsymbol\psi}_{1}=\begin{pmatrix}
		1 \\
		0
	\end{pmatrix},\quad
	{\boldsymbol\psi}_{2}=\begin{pmatrix}
		0\\
		1
	\end{pmatrix},\quad
	{\boldsymbol\psi}_{3}=\begin{pmatrix}
		x_{2}\\
		-x_{1}
	\end{pmatrix}.$$ 
	Let $({\bf u}_{1}^\alpha,p_1^\alpha)$, $\alpha=1,2,3$, be the 
	solution to the Dirichlet problem of the Stokes system
	\begin{equation}\label{u,peq1}
		\begin{cases}
			\mu\Delta{\bf u}_1^\alpha=\nabla p_1^\alpha,\quad\nabla\cdot{\bf 
				u}_1^\alpha=0&\mathrm{in}~\Omega,\\
			{\bf 
				u}_{1}^\alpha=\boldsymbol{\psi}_{\alpha}&\mathrm{on}~\partial{D}_{1},\\
			{\bf u}_{1}^\alpha=0&\mathrm{on}~\partial{D_{2}}\cup\partial{D}.
		\end{cases}
	\end{equation}
	In the next three sections, we consider \eqref{u,peq1} with $\alpha=1,2,3$, and establish the higher derivative estimates of the solutions by applying Proposition \ref{prop3.3}, respectively. 
	
	\section{Estimates for $({\bf u}_{1}^{1},p_{1}^{1})$}\label{sec_u1}
	In this section, we use Proposition \ref{prop3.3} to establish high-order derivative estimates for the solution to \eqref{u,peq1} with prescribed boundary data \(\boldsymbol{\psi}_{1}\) on \(\partial D_{1}\). For any \(m \geq 1\), we construct a sequence of auxiliary functions \({\bf v}_1^l\) and corresponding \(\bar{p}_l\) in \(\Omega_{2R}\), such that the quantity \(\left|\sum_{l=1}^{m+1}(\mu\Delta {\bf v}_1^l - \nabla \bar{p}_l)(x)\right|\) is pointwise bounded in \(\Omega_{2R}\) and satisfies the conditions of Proposition \ref{prop3.3}. We then show that \(\nabla^{m+1} \sum_{l=1}^{m+1} {\bf v}_1^l\) and \(\nabla^m \sum_{l=1}^{m+1} \bar{p}_l\) effectively capture all singularities of \(\nabla^{m+1} {\bf u}_1^1\) and \(\nabla^m p_1^1\), up to \(O(1)\) accuracy.
	
	\subsection{For general $h_{1}(x_{1})$ and $h_{2}(x_{1})$}
	We introduce the following modified Keller-type 
	function $k(x)\in C^{m+1}(\mathbb{R}^2)$:
	\begin{equation}\label{def_k}
		k(x)=\frac{x_2-\frac12(h_1-h_2)(x_{1})}{\delta(x_1)}\quad\mbox{in}~\Omega_{2R},
	\end{equation}
	which satisfies $k(x)=\frac{1}{2}$ on $\Gamma_{2R}^+$ and $k(x)=-\frac{1}{2}$ on $\Gamma_{2R}^-$, and so
	\begin{equation}\label{kbc}
		k(x)+\frac12=1\,\text{on}\,~\Gamma_{2R}^+,\quad k(x)+\frac12=0\,~\text{on}\,
		\Gamma_{2R}^-, \quad \text{and}~~k(x)^2-\frac14=0\,~\text{on}\,\Gamma_{2R}^+\cup\Gamma_{2R}^-.
	\end{equation}
	A direct calculation shows that
	\begin{equation}\label{kdds}
		\partial_{x_{1}}k(x)=-\frac{(h'_1-h'_2)(x_1)}{2\delta(x_1)}-\frac{(h'_1+h'_2)(x_1)}{\delta(x_1)}k(x),\quad \partial_{x_{2}}k(x)=\frac{1}{\delta(x_1)}\quad\mbox{in}~\Omega_{2R}.
	\end{equation}
	
	Note that due to the asymmetry between \(h_1\) and \(h_2\), the difference term \(h_1 - h_2\) introduces significant challenges in the calculations and leads to an increase in the singularities of the higher-order derivatives.
	
	\begin{prop}\label{prop2.1}
		Under the same assumption as in Theorem \ref{main thm1n}, let ${\bf u}_{1}^{1}$ 
		and $p_{1}^{1}$ be the solution to \eqref{u,peq1} with $\alpha=1$. Then for 
		sufficiently small $0<\varepsilon<1/2$, we have
		\begin{equation}\label{mainestv111}
			\mbox{(i) }~ |\nabla{\bf u}_{1}^{1}(x)|\le C(\varepsilon+|x_1|^2)^{-1},\quad |p_1^1(x)-
			p_1^1(z_1,0)|\leq C(\varepsilon+|x_1|^2)^{-\frac32}\quad \text{for}~x\in\Omega_{R}
		\end{equation}
		for some point $z=(z_1,0)\in \Omega_{R}$ with $|z_1|=R/2$; and
		
		(ii) for any $m\ge 1$, 
		\begin{equation}\label{mainestv112}
			|\nabla^{m+1}{\bf u}_{1}^{1}(x)|+|\nabla^{m}{p}_{1}^{1}(x)| \le 
			C(\varepsilon+|x_1|^2)^{-\frac{m+3}{2}}\quad \text{for}~x\in\Omega_{R}.
		\end{equation}
	\end{prop}
	
	\begin{proof}[Proof of Proposition \ref{prop2.1}]
		Without loss of generality, we may assume that $x_1\ge 0$.
		We prove Proposition \ref{prop2.1} by induction. To do so, we construct a series of auxiliary function pairs \(({\bf v}_{1}^{l}(x), \bar{p}_{l}(x))\) to minimize \(\Big|\sum_{l=1}^{m+1}(\mu\Delta{\bf v}_{1}^{l} - \nabla\bar{p}_{l})(x)\Big|\). Then, we show that \(\nabla^{m+1}\sum_{l=1}^{m+1}{\bf v}_1^l\) and \(\nabla^m\sum_{l=1}^{m+1}{\bar{p}}_{l}\) effectively capture all singular terms in \(\nabla^{m+1}{\bf u}_{1}^{1}\) and \(\nabla^{m}{p}_{1}^1\), respectively.
		
		{\bf Step I. Gradient estimates.} To estimate $|\nabla{\bf u}_{1}^{1}|$, we first construct an auxiliary function 
		$${\bf v}_{1}^{1}(x)=\Big(({\bf v}_{1}^{1})^{(1)}(x),({\bf v}_{1}^{1})^{(2)}(x)\Big)^{T},$$
		which satisfies the same boundary condition as ${\bf u}_{1}^{1}$, that is, ${\bf v}_{1}^{1}(x)={\bf u}_{1}^{1}
		(x)={\boldsymbol{\psi}}_{1}$ on $\Gamma^{+}_{2R}$, ${\bf v}_{1}^{1}(x)={\bf 
			u}_{1}^{1}(x)=0$ on $\Gamma^{-}_{2R}$. We define the difference between the pairs $({\bf u}_{1}^{1}, p_{1})$ and $({\bf v}_{1}^{1},\bar{p}_{1})$ as
		\begin{equation*}
			{\bf w}:={\bf u}_{1}^{1}-{\bf v}_{1}^{1}\quad\mbox{and}\quad 
			q:=p_{1}^{1}-\bar{p}_{1}.
		\end{equation*}
		Then $({\bf w},q)$ satisfies the  boundary value problem in the narrow region
		\begin{equation}\label{w11}
			\begin{cases}
				-\mu\Delta{\bf w}+\nabla q={\bf f}_{1}^{1}\quad&\mathrm{in}\ \Omega_{2R},\\
				\nabla\cdot {\bf w}=0\quad&\mathrm{in}\ \Omega_{2R},\\
				{\bf w}=0\quad&\mathrm{on}\ \Gamma^+_{2R}\cup\Gamma^-_{2R},
			\end{cases}
		\end{equation}
		where
		\begin{equation}\label{f11_def}
			{\bf f}_{1}^{1}:=\mu\Delta{\bf v}_{1}^{1}-\nabla\bar{p}_{1}.
		\end{equation}
		
		{\bf Step I.1. Construction of ${\bf v}_{1}^{1}$ and its estimates.} By using the property of $k(x)$ described in \eqref{kbc}, we find an auxiliary function ${\bf v}_{1}^{1}$ in the following form:
		\begin{equation}\label{v11def}
			{\bf v}_{1}^{1}(x)=
			\begin{pmatrix}
				k(x)+\frac{1}{2}\\0
			\end{pmatrix}+
			\begin{pmatrix}
				F(x)\\
				G(x)
			\end{pmatrix}\Big(k(x)^2-\frac{1}{4}\Big)\quad\text{in}~\Omega_{2R},
		\end{equation}
		where $F(x)$ and $G(x)$ are determined in the following way. 
		
		To begin with, through a direct calculation, we have
		\begin{equation}\label{bds1}
			\begin{split}
				&\partial_{x_{1}}({\bf v}_{1}^{1})^{(1)}=\partial_{x_{1}} F(x)\Big(k(x)^2-\frac{1}{4}\Big)+2F(x)k(x)\partial_{x_{1}} k(x)+\partial_{x_{1}} k(x),\\
				&\partial_{x_{2}}({\bf v}_{1}^{1})^{(2)}=\partial_{x_{2}} G(x)\Big(k(x)^2-
				\frac{1}{4}\Big)+\frac{2k(x)}{\delta(x_1)}G(x).
			\end{split}
		\end{equation}
		In order to have $\partial_{x_{1}}({\bf v}_{1}^{1})^{(1)}+\partial_{x_{2}}({\bf v}_{1}^{1})^{(2)}=0$ in $\Omega_{2R}$, we set 
		\begin{equation}\label{ffz1}
			\frac14(\partial_{x_{1}}F+\partial_{x_{2}}G)(x)=-\frac{(h'_1-h'_2)(x_1)}{2\delta(x_1)},
		\end{equation}
		and then \eqref{bds1} gives 
		$$-\frac{2(h'_1-h'_2)(x_1)}{\delta(x_1)}\Big(k(x)^2-\frac{1}{4}\Big)+2F(x)k(x)\partial_{x_{1}} k(x)+\partial_{x_{1}} k(x)+\frac{2k(x)}{\delta(x_1)}G(x)=0.$$
		By using \eqref{kdds}, we have
		\begin{equation}\label{ffz2}
			G(x)=(h'_1-h'_2)(x_{1})k(x)+\frac12(h'_1+h'_2)(x_1)-F(x)
			\delta(x_1)\partial_{x_1}k(x).
		\end{equation}
		Substituting this into \eqref{ffz1} and using \eqref{kdds} again, we see that $F(x)$ satisfies the following first-order partial differential equation:
		\begin{equation*}
			\partial_{x_{1}}F(x)+\frac{(h'_1+h'_2)(x_1)}{\delta(x_1)}F(x)-
			\delta(x_1)
			\partial_{x_1}k(x)\partial_{x_{2}}F(x)=-\frac{3(h'_1-h'_2)(x_1)}
			{\delta(x_1)}.
		\end{equation*}
		In particular, if we assume $F(x)=F(x_{1})$, then
		\begin{equation*}
			F'(x_{1})+\frac{(h'_1+h'_2)(x_1)}{\delta(x_1)}F(x_{1})=-\frac{3(h'_1-h'_2)(x_1)}
			{\delta(x_1)}.
		\end{equation*}
		Then,
		\begin{equation}\label{deff}
			F(x_{1})=-\frac{3(h_1-h_2)(x_{1})}{\delta(x_1)}
		\end{equation}
		is a particular solution. Substituting this $F$ into \eqref{ffz2} yields
		\begin{equation}\label{defg}
			G(x)=(h'_1-h'_2)(x_{1})k(x)+\frac12(h'_1+h'_2)(x_1)+3(h_1-h_2)(x_{1})\partial_{x_1}k(x).
		\end{equation}
		Therefore, $F$ and $G$ are determined, which satisfies $\nabla\cdot{\bf v}_{1}^{1}(x)=0$ in $\Omega_{2R}$. 
		
		Using \eqref{h1h14}, \eqref{v11def}, \eqref{deff}, and \eqref{defg}, a direct calculation shows that in $\Omega_{2R}$,
		\begin{equation}\label{v11esti}
			\begin{split}
				&|\partial_{x_1}({\bf v}_{1}^{1})^{(1)}|\le 
				C\delta(x_1)^{-1/2},\quad|
				\partial_{x_2}({\bf v}_{1}^{1})^{(1)}|\leq C\delta(x_1)^{-1}, \\
				&|\partial_{x_1}({\bf v}_{1}^{1})^{(2)}|\le C, \quad\quad\quad\,\,\,
				\quad|\partial_{x_2}({\bf v}_{1}^{1})^{(2)}|\le C\delta(x_1)^{-1/2}.
			\end{split}
		\end{equation}
		For high-order derivatives, for $k\ge1,$ in 
		$\Omega_{2R},$
		\begin{align}\label{v11highesti}
			\begin{split}
				&|\partial_{x_1}^{k}\partial_{x_2}^{s}({\bf v}_{1}^{1})^{(1)}|\le 
				C\delta(x_1)^{-\frac{k}{2}-s}\quad\,\,\,\text{for}~s\le2,
				\quad\partial_{x_2}^{s}({\bf v}_{1}^{1})^{(1)}=0\quad\text{for}~s\ge3,\\
				&|\partial_{x_1}^{k}\partial_{x_2}^{s}({\bf v}_{1}^{1})^{(2)}|\le 
				C\delta(x_1)^{\frac{1}{2}-\frac{k}{2}-s}\quad\text{for}~s\le3,
				\quad\partial_{x_2}^{s}({\bf v}_{1}^{1})^{(2)}=0\quad\text{for}~s\ge4.
			\end{split}	
		\end{align}
		From \eqref{v11highesti}, we have 
		\begin{equation}\label{Delta v}
			|\mu\Delta{\bf v}_{1}^{1}|\le C|\Delta({\bf v}_{1}^{1})^{(1)}|+C|\Delta({\bf v}_{1}^{1})^{(2)}|\le C\delta(x_1)^{-2}.
		\end{equation}
		We observe that there are two main terms in \(\Delta{\bf v}_{1}^{1}\): \(\partial_{x_2x_2}({\bf v}_{1}^{1})^{(1)}\) and \(\partial_{x_2x_2}({\bf v}_{1}^{1})^{(2)}\), which are of orders \(\delta(x_1)^{-2}\) and \(\delta(x_1)^{-3/2}\), respectively. Next, we seek a suitable \(\bar{p}_{1}\) such that \(|\mu\Delta{\bf v}_{1}^{1} - \nabla\bar{p}_{1}|\) becomes as small as possible.
		
		{\bf Step I.2. Construction of $\bar{p}_{1}$ and its estimates.} 
		We choose
		\begin{equation}\label{barp1}
			\bar{p}_{1}= 6\mu\int_{x_1}^{R}\frac{(h_1-h_2)(y)}{\delta(y)^3}\ dy   
			+\mu\partial_{x_2}({\bf v}_{1}^{1})^{(2)} \quad\text{in}~\Omega_{2R}.
		\end{equation}
		This choice allows us to eliminate the dominant terms. Specifically, by \eqref{v11def} and \eqref{barp1},
		\begin{equation}\label{cancel1}
			\mu\partial_{x_2x_2}({\bf v}_{1}^{1})^{(2)}
			-\partial_{x_2}\bar{p}_{1}=0,\quad \mu\partial_{x_2x_2}({\bf v}_{1}^{1})^{(1)}-\partial_{x_1}\bar{p}_{1}=-\mu\partial_{x_1x_2}({\bf v}_{1}^{1})^{(2)}.
		\end{equation}
		
		From \eqref{barp1}, a straightforward calculation shows that
		\begin{equation}\label{barpesti}
			|\bar{p}_{1}|\le C\delta(x_1)^{-3/2},\quad|\partial_{x_1}\bar{p}
			_{1}|\le C\delta(x_1)^{-2},\quad|\partial_{x_2}\bar{p}_{1}|\le 
			C\delta(x_1)^{-3/2}\quad\text{in}~\Omega_{2R}.
		\end{equation}
		Using \eqref{v11esti} and \eqref{barpesti}, we deduce 
		\begin{equation}\label{gradmv1p1x}
			|\nabla{\bf v}_{1}^{1}|\le C\delta(x_1)^{-1},\quad |\nabla\bar{p}_{1}|\le C\delta(x_1)^{-2}\quad\text{in}~\Omega_{2R}.
		\end{equation}

		{\bf Step I.3. Estimates of ${\bf f}_{1}^{1}$.} By virtue of \eqref{f11_def} and \eqref{cancel1}, we have
		\begin{equation*}
			({\bf f}_{1}^{1})^{(1)}=\mu \partial_{x_1}\big(\partial_{x_{1}}({\bf v}_{1}^{1})^{(1)}-\partial_{x_{2}}({\bf v}_{1}^{1})^{(2)}\big), \quad 	({\bf f}_{1}^{1})^{(2)}=\mu\partial_{x_1x_1}({\bf v}_{1}^{1})^{(2)}\quad \text{in}~\Omega_{2R}.
		\end{equation*}
		From \eqref{v11def}, \eqref{deff}, \eqref{defg}, we notice that $({\bf v}_1^1)^{(1)}$ is a quadratic polynomial in $x_2$, while $({\bf v}_1^1)^{(2)}$ is a cubic polynomial in $x_2$. Therefore, we can represent ${\bf f}_{1}^{1}$ as polynomials in $x_{2}$:
		\begin{align}\label{f11form}
			\begin{split}
				&({\bf f}_{1}^{1})^{(1)}=S_{21}(x_1)x_{2}^{2}+S_{11}(x_1)x_{2}
				+S_{01}(x_1),\\
				&({\bf f}_{1}^{1})^{(2)}=G_{31}(x_1)x_{2}^{3}+G_{21}(x_1)x_{2}^{2}
				+G_{11}(x_1)x_{2}+G_{01}(x_1).
			\end{split}
		\end{align}
		By a direct calculation and using  \eqref{v11highesti}, for any $s\ge 0$, $i=0,1,2$ and $j=0,1,2,3,$ 
		\begin{equation}\label{SGesti}
			|S_{i1}^{(s)}|\le C\delta(x_1)^{-i-\frac{s+2}{2}},\quad|G_{j1}^{(s)}|\le 
			C\delta(x_1)^{-j-\frac{s+1}{2}}.
		\end{equation}
		Then, it is immediate from \eqref{f11form} and  \eqref{SGesti} that
		\begin{equation*}
			|({\bf f}_{1}^{1})^{(1)}|\le C\delta(x_1)^{-1}, \quad|({\bf f}_{1}
			^{1})^{(2)}|\le C\delta(x_1)^{-1/2}.
		\end{equation*}
		Hence, 
		\begin{equation*}
			|{\bf f}_{1}^{1}|\le C (|({\bf f}_{1}^{1})^{(1)}|+|({\bf 
				f}_{1}^{1})^{(2)}|) \le 
			C \delta(x_1)^{-1}\quad \text{in}~\Omega_{2R}.
		\end{equation*}
		This upper bound is much better than the right-hand side, $C\delta(x_1)^{-2}$, of \eqref{Delta v}.
		
		We now employ Proposition \ref{prop3.3} to obtain 
		\begin{align*}
			\|\nabla{\bf u}_{1}^{1}-\nabla{\bf v}_{1}^{1}\|
			_{L^{\infty}(\Omega_{\delta(x_1)/2}(x_1))}\le C \quad \mbox{for}~ x\in\Omega_{R}.
		\end{align*}
		By \eqref{v11esti},
		$$|\nabla{\bf u}_{1}^{1}(x)|\le C|\nabla {\bf v}_{1}^{1}(x)|+C \le 
		C\delta(x_1)^{-1}\quad \text{in}~\Omega_{R}.$$
		Thus, Proposition \ref{prop2.1} holds true for $m=0$.
		
		In the sequel, we need the following estimates for higher-order derivatives, $2\le k\le m+1,$ 
		\begin{equation}\label{barphighest}
			\begin{split}
				&|\partial_{x_1}^{k}\bar{p}_{1}|\le C\delta(x_1)^{-\frac{k+3}{2}},\quad\quad\quad|\partial_{x_1}^{k-1}\partial_{x_2}{\bar{p}}_{1}|\le 
				C\delta(x_1)^{-\frac{k+2}{2}}, \\ 
				&|\partial_{x_1}^{k-2}\partial_{x_2}^{2}{\bar{p}}_{1}|\le C\delta(x_1)^{-\frac{k+3}{2}},
				\quad\partial_{x_2}^{s}{\bar{p}_{1}}=0\quad\text{for}~s\ge3
			\end{split}\quad\text{in}~\Omega_{2R}.
		\end{equation}
		By \eqref{v11highesti} and \eqref{barphighest}, we have, for $m\ge 1$, 
		\begin{equation}\label{gradmv1p1}
			|\nabla^{m+1}{\bf v}_{1}^{1}|\le C \delta(x_1)^{-\frac{m+3}{2}},\quad|\nabla^m\bar{p}_{1}|\le 
			C\delta(x_1)^{-\frac{m+3}{2}}\quad\text{in}~\Omega_{2R}.
		\end{equation}
		
		{\bf Step II.  High-order Derivatives Estimates.} Next, we inductively derive the estimates for of $\nabla^{m+1}{\bf u}_{1}^{1}$ and $\nabla^{m}p_{1}^{1}$, based on Proposition \ref{prop3.3}.  
		
		To achieve this, we need to construct a series of auxiliary functions $({\bf v}_{1}^{l},\bar{p}_{l})$ to replace $({\bf v}_{1}^{1},\bar{p}_{1})$ in \eqref{w11} with $(\sum_{l=1}^{m+1}{\bf v}_{1}^{l},\sum_{l=1}^{m+1}\bar{p}_{l})$, and make the right-hand side 
		$$
		{\bf f}_{1}^{m+1}:=\Big(\mu\Delta\sum_{l=1}^{m+1}{\bf v}_{1}^{l}-\nabla\sum_{l=1}^{m+1}\bar{p}_{l}\Big)(x)
		$$ 
		as small as possible. 
		
		Define
		\begin{equation*}
			{\bf f}_{1}^{j}:={\bf f}_{1}^{j-1}(x)+\mu\Delta{\bf v}_{1}^{j}(x)-\nabla\bar{p}_{j}(x)=\sum_{l=1}^{j}(\mu\Delta{\bf v}_{1}^{l}-\nabla\bar{p}_{l})(x)
			\quad\text{for}~ 2\le j\le m+1.
		\end{equation*}
		We first prove by induction that for $j\ge1$, ${\bf f}_{1}^{j}$ can be represented as polynomials in $x_2$ as follows:
		\begin{equation}\label{f1jform}
			({\bf f}_{1}^{j})^{(1)}=\sum_{i=0}^{2j}S_{ij}(x_1)x_{2}^{i},\quad ({\bf f}_{1}^{j})^{(2)}=\sum_{i=0}^{2j+1}G_{ij}(x_1)x_{2}^{i}\quad\text{in}~\Omega_{2R},
		\end{equation}
		where
		\begin{equation}\label{f1jesti}
			|S_{ij}^{(k)}|\le C\delta(x_1)^{j-i-\frac{k}{2}-2},\quad|G_{ij}^{(k)}|\le 
			C\delta(x_1)^{j-i-\frac{k}{2}-\frac{3}{2}},\quad~\forall~k\ge0.
		\end{equation}
		
		Indeed, by virtue of \eqref{f11form} and \eqref{SGesti}, we have \eqref{f1jform} and 
		\eqref{f1jesti} for $j=1.$ Assume that \eqref{f1jform} and 
		\eqref{f1jesti} 
		hold for $j=l-1$ with $l\ge2$, that is,
		\begin{equation}\label{f1l-1form}
			({\bf f}_{1}^{l-1})^{(1)}=\sum_{i=0}^{2l-2}S_{i(l-1)}(x_1)x_{2}^{i},\quad
			({\bf f}_{1}^{l-1})^{(2)}=\sum_{i=0}^{2l-1}G_{i(l-1)}(x_1)x_{2}^{i}
			,	
		\end{equation}
		where, for any $k\ge0$,
		\begin{equation}\label{SGl-1est}
			|S_{i(l-1)}^{(k)}|\le C\delta(x_1)^{l-i-\frac{k}{2}-3},\quad|G_{i(l-1)}
			^{(k)}|\le C\delta(x_1)^{l-i-\frac{k}{2}-\frac{5}{2}}.
		\end{equation}
		By \eqref{f1l-1form} and \eqref{SGl-1est}, we have
		\begin{equation*}
			|({\bf f}_{1}^{l-1})^{(1)}|\leq C\delta(x_1)^{l-3},\quad |({\bf f}_{1}^{l-1})^{(1)}|\leq C\delta(x_1)^{l-\frac52}\quad \text{in}~\Omega_{2R}.
		\end{equation*}
		Note that the leading term in ${\bf f}_1^{l-1}$ is $({\bf f}_{1}^{l-1})^{(1)}$, which is of order $\delta(x_1)^{l-3}$. 
		We will construct ${\bf v}_{1}^{l}$ and $\bar{p}_l$ to cancel out this leading term such that ${\bf f}_{1}^{l}={\bf f}_{1}^{l-1}(x)+\mu\Delta{\bf v}_{1}^{l}(x)-\nabla\bar{p}_{l}(x)$ satisfies \eqref{f1jform} and \eqref{f1jesti}. 
		
		{\bf Step II.1. Construction of ${\bf v}_{1}^{l}$ and its estimates.} We construct a divergence-free auxiliary function 
		$${\bf v}_{1}^{l}=\Big(({\bf v}_{1}^{l})^{(1)}(x),({\bf 
			v}_{1}^{l})^{(2)}(x)\Big)^{T}$$ satisfying ${\bf v}_{1}^{l}(x)=0$ on 
		$\Gamma^{+}_{2R}\cup\Gamma^{-}_{2R}$ in the following form
		\begin{align}\label{v1ldef}
			\begin{split}
				({\bf v}_{1}^{l})^{(1)}(x)=&\,\Big(\sum_{i=0}^{2l-2}F_{1l}^{i}(x_1)x_{2}
				^{i}+\tilde{F}_{1l}(x_1) \Big)\Big(k(x)^2-\frac{1}{4}\Big), \\
				({\bf v}_{1}^{l})^{(2)}(x)=&\, \Big(\sum_{i=0}^{2l-1}F_{2l}^{i}(x_1)x_{2}
				^{i}+\tilde{F}_{2l}(x)\Big)\Big(k(x)^2-\frac{1}{4}\Big)
			\end{split}
			\quad\quad \text{in}~\Omega_{2R}.
		\end{align}
		
		We first construct functions $F_{jl}^i(x_1)$ and then $\tilde{F}_{jl}(x_1)$, $j=1,2$. Here the functions $F_{1l}^i(x_1)$, $i=0,1,\dots,2l-2$, are used to eliminate $({\bf f}_{1}^{l-1})^{(1)}$ through the following differential equation
		\begin{equation}\label{cancelf1l1}
			\mu\partial_{x_2x_2}\Big(\sum_{i=0}^{2l-2}F_{1l}^{i}(x_1)x_2^{i}
			\Big(k(x)^2-\frac14\Big)\Big)=-({\bf f}_{1}^{l-1})^{(1)}.
		\end{equation}
		Recalling the definition of $({\bf f}_{1}^{l-1})^{(1)}$, \eqref{f1l-1form}, it can be deduced that for $0\le i\le 2l-2,$
		\begin{equation}\label{dy1}
			F_{1l}^{i}(x_1)= -\frac{\delta(x_1)^2S_{i(l-1)}(x_1)}{\mu (i+1)(i+2)}+
			(h_1-h_2)(x_{1})F_{1l}^{i+1}(x_1)+\frac14(\varepsilon+2h_1(x_{1}))(\varepsilon+2h_2(x_{1}))F_{1l}^{i+2}(x_1),
		\end{equation}
		by comparing the coefficients of the polynomials of $x_2$ on both sides of \eqref{cancelf1l1}. Here we use the convention that $F_{1l}^{i}(x_1)\equiv0$ if $i\notin\{0,\dots,2l-2\}$ and $F_{2l}^{i}=0$ if $i\notin\{0,\dots,2l-1\}$. We note that once $F_{1l}^i$, $i=0,1,\dots,2l-2$ is determined, in general, there does not exist $F_{2l}^i$, $i=0,1,\dots,2l-1$ such that $\nabla\cdot{\bf v}_1^l=0$. We choose 
		\begin{align}\label{dy2}
			&F_{2l}^{i}(x_1)=(h_1-h_2)(x_{1})F_{2l}^{i+1}(x_1)+\frac14(\varepsilon+2h_1(x_{1}))(\varepsilon+2h_2(x_{1}))F_{2l}^{i+2}(x_1) \nonumber\\
			&\quad -\frac{\delta(x_1)^2}{i+2}\Big(\delta(x_1)^{-2}\Big(F_{1l}^{i-1}(x_1)-(h_1-h_2)(x_{1})
			F_{1l}^{i}(x_1)-\frac14(\varepsilon+2h_1)(\varepsilon+2h_2)F_{1l}^{i+1}(x_1)\Big)\Big)',
		\end{align}
		for $i=0,\dots,2l-1$, by comparing the coefficients of the polynomial of $x_2$, such that
		\begin{align}\label{zj1}
			\nabla\cdot{\bf v}_1^l=\nabla \cdot\begin{pmatrix}
				\Big(\sum_{i=0}^{2l-2}F_{1l}^{i}(x_1)x_{2}
				^{i}\Big)\Big(k(x)^2-\frac{1}{4}\Big)\\\\
				\Big(\sum_{i=0}^{2l-1}F_{2l}^{i}(x_1)x_{2}
				^{i}\Big)\Big(k(x)^2-\frac{1}{4}\Big)
			\end{pmatrix}:=R(x_1),
		\end{align}
		where
		\begin{align*}
			R(x_1)=&\,-\Big(\frac{(\varepsilon+2h_1(x_{1}))(\varepsilon+2h_2(x_{1}))}{4\delta(x_1)^2}F_{1l}^{0}(x_1)\Big)'\\
			&\,-
			\frac{(\varepsilon+2h_1(x_{1}))(\varepsilon+2h_2(x_{1}))}{4\delta(x_1)^2}F_{2l}^{1}(x_1)-\frac{h_1(x_1)-h_2(x_1)}{\delta(x_1)^2}F_{2l}^{0}(x_1)
		\end{align*}
		is the 0-th order term of $x_{2}$. 
		
		Next, in order to further eliminate the non-zero term $R(x_1)$ in \eqref{zj1}, we choose $\tilde{F}_{1l}(x_1)$ and $\tilde{F}_{2l}(x)$ such that
		\begin{equation}\label{yxws}
			\nabla \cdot\begin{pmatrix}
				\tilde{F}_{1l}(x_1)\Big(k(x)^2-\frac{1}{4}\Big) \\
				\tilde{F}_{2l}(x)\Big(k(x)^2-\frac{1}{4}\Big)
			\end{pmatrix}=-R(x_1).
		\end{equation}
		The process of finding $\tilde{F}_{jl}(x)$ is similar to the construction of ${\bf v}_{1}^{1}$ in Step I.1. By a direct calculation, \eqref{yxws} becomes
		\begin{equation}\label{fzzy1}
			\big(\tilde{F}_{1l}'(x_1)+\partial_{x_{2}}\tilde{F}_{2l}(x)\big)\Big(k(x)^2-\frac{1}{4}\Big)+2\tilde{F}_{1l}(x_1)k(x)\partial_{x_{1}}k(x)+\frac{2k(x)}{\delta(x_1)}\tilde{F}_{2l}(x)=-R(x_1).
		\end{equation}
		To determine $\tilde{F}_{1l}(x_1)$ and $\tilde{F}_{2l}(x)$, we set 
		\begin{equation}\label{fzzy2}
			\tilde{F}_{1l}'(x_1)+\partial_{x_{2}}\tilde{F}_{2l}(x)=4R(x_1).
		\end{equation}
		It follows from \eqref{fzzy1} and \eqref{fzzy2} that
		\begin{equation}\label{fzzy3}
			\tilde{F}_{2l}(x)=-2\delta(x_1)R(x_1)k(x)-\delta(x_1)\partial_{x_1}k(x)\tilde{F}_{1l}(x_1).
		\end{equation}
		By \eqref{kdds}, \eqref{fzzy2}, and \eqref{fzzy3}, we have
		\begin{equation*}
			\tilde{F}_{1l}'(x_1)+\frac{(h'_1+h'_2)(x_1)}{\delta(x_1)}\tilde{F}_{1l}(x_1)=6R(x_1).
		\end{equation*}
		Similarly as before, here we take special solutions
		\begin{equation}\label{dy3}
			\tilde{F}_{1l}(x_1)=\frac{6\int_{0}^{x_1}\delta(y)R(y)\ dy}{\delta(x_1)},\quad
			\tilde{F}_{2l}(x)=-2\delta(x_1)R(x_1)k(x)-\delta(x_1)\partial_{x_1}k(x)
			\tilde{F}_{1l}(x_1)
		\end{equation}
		by using \eqref{fzzy3}. Thus, the divergence-free condition $\nabla\cdot{\bf v}_1^l(x)=0$ holds in $\Omega_{2R}$. 
		By \eqref{def_k}, \eqref{kdds}, and \eqref{dy3}, we have
		\begin{align}\label{?dy3}
			\tilde{F}_{2l}(x)=&\,\Big(\frac{(h_1-h_2)'(x_1)}{\delta(x_1)}-2R(x_1)\Big)x_2\nonumber\\
			&+\Big((h_1-h_2)'(x_{1})-\frac{(h_1+h_2)'(h_1-h_2)(x_{1})}{2\delta(x_1)}\Big)\tilde{F}_{1l}(x_1)+(h_1-h_2)(x_{1})R(x_1)\nonumber\\
			:=&\,P(x_1)x_2+Q(x_1).
		\end{align}
		
		Moreover, by using \eqref{h1h14}, \eqref{kdds}, \eqref{SGl-1est}, \eqref{dy1}, \eqref{dy2}, \eqref{dy3}, and \eqref{?dy3}, a direct calculation gives, for $k\ge 0$ and $0\le i\le2l-2$, in $\Omega_{2R}$,
		\begin{align}\label{v1lF1l2lest}
			\begin{split}
				&|F_{1l}^{i(k)}|\le C\delta(x_1)^{l-i-\frac{k}{2}-1},\qquad|\tilde{F}_{1l}
				^{(k)}|\le C\delta(x_1)^{l-\frac{k}{2}-1},\\
				&|F_{2l}^{i(k)}|\le C\delta(x_1)^{l-i-\frac{k}{2}-\frac{1}{2}},
				\quad\quad|F_{2l}^{(2l-1)(k)}|\le C\delta(x_1)^{-l-\frac{k}{2}+\frac12},
			\end{split}       
		\end{align}
		and 
		\begin{equation}\label{PQest}
			|P^{(k)}|\le C\delta(x_1)^{l-\frac{k}{2}-\frac32},\quad|Q^{(k)}|
			\le C \delta(x_1)^{l-\frac{k}{2}-\frac12},
		\end{equation}
		where $P(x_1)$ and $Q(x_1)$ are defined in \eqref{?dy3}. By \eqref{v1lF1l2lest}, we find that $\frac{2\tilde{F}_{1l}(x_1)}{\delta(x_1)^{2}}$, which is of order $\delta(x_1)^{l-3}$, becomes a new leading term in ${\bf f}_{1}^{l-1}+\mu\Delta{\bf v}_{1}^{l}$. 
		
		{\bf Step II.2. Construction of $\bar{p}_{l}$ and its estimates.} We divide the construction of $\bar{p}_{l}$ into two steps as follows: 
		$$\bar{p}_{l}:=\hat{p}_{l}+\tilde{p}_{l}.$$ 
		
		To eliminate the new leading term $\frac{2\tilde{F}_{1l}(x_1)}{\delta(x_1)^{2}}$,  
		we set
		\begin{equation}\label{hatp1l}
			\hat{p}_{l}=2\mu\int_{0}^{x_1}\frac{\tilde{F}_{1l}(y)}{\delta(y)^2}\ dy 
			\quad\text{in}~\Omega_{2R}.
		\end{equation}
		Then,
		\begin{equation}\label{hatpcancel1}
			\partial_{x_1}\hat{p}_{l}=\frac{2\mu}{\delta(x_1)^2}\tilde{F}_{1l}(x_1)=\mu\partial_{x_2x_2}\Big(\tilde{F}_{1l}(x_1)
			\Big(k(x)^2-\frac{1}{4}\Big)\Big),\quad \partial_{x_2}\hat{p}_{l}=0.
		\end{equation}
		From \eqref{v1ldef} and \eqref{dy3}, we know that $({\bf v}_{1}^{l})^{(1)}(x)$ is a polynomial of order $2l$ in $x_2$, while $({\bf v}_{1}^{l})^{(2)}(x)$ is a polynomial of order $2l+1$ in $x_2$. Combining with \eqref{f1l-1form}, \eqref{cancelf1l1}, and \eqref{hatpcancel1}, we have
		\begin{align}\label{f1l-1-hatp}
			&{\bf f}_{1}^{l-1}+\mu\Delta{\bf v}_{1}^{l}-\nabla\hat{p}_{l}\nonumber\\=&\begin{pmatrix}
				\mu\partial_{x_1x_1}({\bf v}_{1}^{l})^{(1)}\\\\
				({\bf f}_{1}^{l-1})^{(2)}+\mu\Delta({\bf v}_{1}^{l})^{(2)}
			\end{pmatrix}
			=\begin{pmatrix}
				\sum_{i=0}^{2l}H_{il}(x_1)x_2^{i}\\\\
				\sum_{i=0}^{2l+1}G_{il}(x_1)x_2^{i}+\sum_{i=0}^{2l-1}\tilde{S}_{il}(x_1)x_2^{i}
			\end{pmatrix},
		\end{align}
		where 
		$$ \sum_{i=0}^{2l+1}G_{il}(x_1)x_2^{i}=\mu\partial_{x_1x_1}({\bf v}_{1}^{l})^{(2)},\quad\mbox{and}~\sum_{i=0}^{2l-1}\tilde{S}_{il}(x_1)x_2^{i}= ({\bf f}_{1}^{l-1})^{(2)}+\mu\partial_{x_2x_2}({\bf v}_{1}^{l})^{(2)}.$$
		By \eqref{h1h14}, \eqref{v1ldef}, \eqref{v1lF1l2lest}, and \eqref{PQest}, we obtain for $k\ge0,$
		\begin{equation}\label{HGtilSest}
			|H_{il}^{(k)}|\le C \delta(x_1)^{l-i-\frac{k}{2}-2},\,\,\,|G_{il}^{(k)}|
			\le\delta(x_1)^{l-i-\frac{k}{2}-\frac{3}{2}},\,\,\,|\tilde{S}_{il}^{(k)}|\le 
			C\delta(x_1)^{l-i-\frac{k}{2}-\frac{5}{2}}.
		\end{equation}
		By \eqref{f1l-1-hatp} and \eqref{HGtilSest}, 
		\begin{align*}
			&|\partial_{x_1x_1}({\bf v}_{1}^{l})^{(1)}|\leq C\delta(x_1)^{l-2},\,\quad  |\partial_{x_1x_1}({\bf v}_{1}^{l})^{(2)}|\leq C\delta(x_1)^{l-\frac32},\\
			& |({\bf f}_{1}^{l-1})^{(2)}+\mu\partial_{x_2x_2}({\bf v}_{1}^{l})^{(2)}|\leq C\delta(x_1)^{l-\frac52}.
		\end{align*}
		Notice that the leading term in $({\bf f}_{1}^{l-1})^{(2)}+\mu(\Delta{\bf v}_{1}^{l})^{(2)}-\partial_{x_{2}}\hat{p}_{l}$ is $\sum_{i=0}^{2l-1}\tilde{S}_{il}(x_1)x_2^{i}$. 
		
		Thus, to further reduce the upper bound of the inhomogeneous term, we choose 
		\begin{equation}\label{tilpl}
			\tilde{p}_{l}=\sum_{i=0}^{2l-1}\frac{\tilde{S}_{il}(x_1)}{i+1}x_2^{i+1}
			\quad\text{in}~\Omega_{2R},
		\end{equation}
		such that
		$	\partial_{x_{2}}\tilde{p}_{l}=\sum_{i=0}^{2l-1}\tilde{S}_{il}(x_1)x_2^{i}.$ Now we define an auxiliary pressure function $\bar{p}_{l}:=\hat{p}_{l}+\tilde{p}_{l}.$ 
		
		{\bf Step II.3. Estimates of ${\bf f}_{1}^{l}$.} 
		From \eqref{f1l-1-hatp} and \eqref{tilpl}, we have
		\begin{equation*}
			\begin{split}
				({\bf f}_{1}^{l})^{(1)}=&\,\sum_{i=1}^{2l}\Big(H_{il}-\frac1i\tilde{S}
				_{(i-1)l}'\Big)(x_1)x_2^{i} +H_{0l}(x_1):=\sum_{i=0}^{2l}
				S_{il}(x_1)x_2^{i},\\
				({\bf f}_{1}^{l})^{(2)}=&\,\sum_{i=0}^{2l+1}G_{il}(x_1)x_2^{i},
			\end{split}	
		\end{equation*} 
		where $$S_{il}(x_1)=H_{il}(x_1)-\frac1i \tilde{S}_{(i-1)l}'(x_1),\,\,1\le 
		i\le2l,\quad S_{0l}(x_1)=H_{0l}(x_1).$$
		
		In view of \eqref{HGtilSest}, we have for $k\ge0,$
		$$|S_{il}^{(k)}|\le C\delta(x_1)^{l-i-\frac{k}{2}-2},\quad |G_{il}^{(k)}|\le 
		C\delta(x_1)^{l-i-\frac{k}{2}-\frac{3}{2}}.$$
		Thus, \eqref{f1jform} and \eqref{f1jesti} hold for any $j\ge 1.$ By using 
		\eqref{v1ldef}, \eqref{v1lF1l2lest}--\eqref{hatp1l}, and \eqref{tilpl}, a direct calculation yields for $l\ge 2$ and $k\ge 0$,
		\begin{align}\label{highvpfl-1esti}
			\begin{split}
				&|\partial_{x_1}^{k}\partial_{x_2}^{s}({\bf v}_{1}^{l})^{(1)}|
				\le\delta(x_1)^{l-s-\frac{k+2}{2}}\quad\text{for}~ s\le 2l,\\
				&\partial_{x_2}^{s}({\bf v}_{1}^{l})^{(1)}=0\quad\text{for}~s\ge 2l+1, \\
				&|\partial_{x_1}^{k}\partial_{x_2}^{s}({\bf v}_{1}^{l})^{(2)}|
				\le\delta(x_1)^{l-s-\frac{k+1}{2}}\quad\text{for}~ s\le 2l+1,\\
				&\partial_{x_2}^{s}({\bf v}_{1}^{l})^{(2)}=0\quad\text{for}~s\ge 
				2l+2, \\
				&|\partial_{x_1}^{k}\partial_{x_2}^{s}\tilde{p}_{l}|\le C\delta(x_1)^{l-s-
					\frac{k+3}{2}}\quad\text{for}~0\leq s\le2l,\\
				&\partial_{x_2}^{s}\tilde{p}_{l}=0\quad\text{for}~s\ge 2l+1,\,\,|\partial_{x_1}
				^{k}\hat{p}_{l}|\le C\delta(x_1)^{l-\frac{k}{2}-\frac{5}{2}}.
			\end{split}	
		\end{align}
		From \eqref{f1jform} and \eqref{f1jesti}, we have
		$|({\bf f}_{1}^{m+1})^{(1)}| \le C\delta(x_1)^{m-1},\,\,|({\bf f}_{1}
		^{m+1})^{(2)}|\le C \delta(x_1)^{m-\frac{1}{2}}$, and then
		$$|{\bf f}_{1}^{m+1}|\le |({\bf f}_{1}^{m+1})^{(1)}|+|({\bf f}_{1}^{m+1})^{(2)}|\leq C \delta(x_1)^{m-1}. $$	
		Similarly, for $1\le s\le m,$
		$$|\nabla^{s}{\bf f}_{1}^{m+1}|\le C \delta(x_1)^{m-s-1}.$$
		
		{\bf Step II.4. Estimates of $\nabla^{m+1}{\bf u}_{1}^{1}$ and $\nabla^{m}p_{1}^{1}$.} 
		We denote
		\begin{equation*}
			{\bf v}^{m+1}(x):=\sum_{l=1}^{m+1}{\bf v}_{1}^{l}(x),\quad \bar{p}^{m+1}(x):=\sum_{l=1}^{m+1}\bar{p}_{l}(x).
		\end{equation*}
		It is easy to verify that ${\bf v}^{m+1}(x)={\bf u}_1^1(x)$ on $\Gamma_{2R}^{\pm}$, $\nabla\cdot{\bf v}^{m+1}(x)=0$ in $\Omega_{2R}$, and
		\begin{align*}
			|{\bf f}^{m+1}(x)|=&\,|\mu\Delta{\bf v}^{m+1}(x)-\nabla \bar{p}^{m+1}(x)|\leq C \delta(x_1)^{m-1},\\
			|\nabla^s{\bf f}^{m+1}(x)|\leq&\, C \delta(x_1)^{m-s-1},~ 1\le s\le m.
		\end{align*}
		Thus, by applying Proposition \ref{prop3.3}, it holds that
		\begin{equation}\label{u11-vkest}
			\|\nabla^{m+1}({\bf u}_{1}^{1} -{\bf v}^{m+1})\|_{L^{\infty}(\Omega_{\delta(x_1)/2}(x_1))} +\|\nabla^{m} (p_{1}^{1}-\bar{p}^{m+1})\|_{L^{\infty}(\Omega_{\delta(x_1)/2}(x_1))} \le C.
		\end{equation}
		
		By \eqref{highvpfl-1esti}, we obtain for $l\ge 2$,
		\begin{equation}\label{v1lest}
			\begin{split}
				&|\nabla^{m+1}{\bf v}_{1}^{l}|\le C\delta(x_1)^{-\frac{m+3}{2}}\quad\text{for}
				~l\le \frac{m+1}{2},\\
				&|\nabla^{m+1}{\bf v}_{1}^{l}|\le C\delta(x_1)^{l-m-2}
				\quad\text{for}~\frac{m+1}{2}\le l\le m+1,\\
				&|\nabla^{m}\bar{p}_{l}|\le C \delta(x_1)^{-\frac{m+3}{2}}\quad\text{for}~ 
				l\le \frac{m-1}{2},\quad 	|\nabla\bar{p}_{l}|\leq C\delta(x_1)^{l-3},\\ 
				&|\nabla^{m}\bar{p}_{l}|\le C\delta(x_1)^{l-m-\frac{3}
					{2}}\quad\text{for}~1\le\frac{m}{2}\le l\le m+1,\quad |\bar{p}_{l}|\leq C\delta(x_1)^{l-5/2}.
			\end{split}
		\end{equation}
		It follows from \eqref{gradmv1p1x}, \eqref{gradmv1p1}, and \eqref{v1lest} that
		\begin{equation}\label{zgj}
			|\nabla{\bf v}^{m+1}(x)|\leq C\delta(x_1)^{-1}
		\end{equation}
		and for $m\ge 1$,
		\begin{equation}\label{zgj1}
			|\nabla^{m+1}{\bf v}^{m+1}(x)|+|\nabla^{m}\bar{p}^{m+1}(x)|\leq C\delta(x_1)^{-\frac{m+3}{2}}.
		\end{equation}
		By \eqref{u11-vkest}, \eqref{zgj}, and \eqref{zgj1}, we have for $x\in\Omega_{R},$
		\begin{equation*}
			|\nabla{\bf u}_{1}^{1}(x)|\leq C\delta(x_1)^{-1}\quad\text{in}~\Omega_{R},
		\end{equation*}
		and
		\begin{equation*}
			|\nabla^{m+1}{\bf u}_{1}^{1}(x)|+|\nabla^{m}{p}_{1}^{1}(x)|\le C \delta(x_1)^{-\frac{m+3}{2}} \quad\text{in}~\Omega_{R}.
		\end{equation*}
		
		For the estimates of $p_1^1(x)$, by the mean value theorem, \eqref{barpesti}, \eqref{u11-vkest}, and \eqref{v1lest}, we have
		\begin{align*}
			|p_{1}^{1}(x)-p_{1}^{1}(z_1,0)|\leq&\, 	|p_{1}^{1}-\bar{p}_1-\bar{p}_2-(p_{1}^{1}-\bar{p}_1-\bar{p}_2)(z_1,0)|+|\bar{p}_1+\bar{p}_2|+C\\
			\leq&\, C\|\nabla (p_{1}^{1}-\bar{p}_1-\bar{p}_2)\|_{L^\infty(\Omega_{R})}+|\bar{p}_1+\bar{p}_2|+C\\
			\leq&\, C \delta(x_1)^{-3/2}
		\end{align*}
		for a fixed point $(z_1,0)\in\Omega_{R}$ with $|z_1|=R/2$.
		This completes the proof of Proposition \ref{prop2.1}.
	\end{proof}
	
	\begin{remark}
		We would like to emphasize that the asymmetry of the inclusions leads to an increase in the singularity of \(\nabla^2{\bf u}_1^1\). Specifically, from \eqref{v11def} and \eqref{deff}, it can be deduced that the leading term of \(\nabla^2({\bf v}_1^1)^{(1)}\) is \(2\delta(x_1)^{-2}F(x)\), which is of order \(\delta(x_1)^{-2}\) when \(h_1(x_1) \neq h_2(x_1)\). However, if \(h_1(x_1) = h_2(x_1)\), the leading terms of \(\nabla^2({\bf v}_1^1)^{(1)}\) are \(\partial_{x_1x_2}k(x)\) and \(2\delta(x_1)^{-2}G(x)\), both of which are of order \(\delta(x_1)^{-3/2}\). 
	\end{remark}

	\subsection{The case when $h_{1}(x_{1})=h_{2}(x_{1})=\frac{1}{2}h(x_{1})$}
	When $h_1(x_1)= h_2(x_1)$, the estimates for the higher-order derivatives and $p_{1}^{1}(x)$ in Proposition \ref{prop2.1} can be improved by an alternative method. In this case, we use the Green function approach to construct the auxiliary ${\bf v}_{1}^{l} $ in the induction process, similar to \cite{DLTZ} for the Lam\'e system. 
	
	\begin{prop}\label{prop2.1g}
		Under the same assumption as in Theorem \ref{main thm1}, let ${\bf u}_{1}^{1}$ 
		and $p_{1}^{1}$ be the solution to \eqref{u,peq1} with $\alpha=1$ and $h_1(x_1)=h_2(x_1)=\frac{1}{2}h(x)$. Then for sufficiently small $0<\varepsilon<1/2$, we have
		$$\mbox{(i)}~|\nabla{\bf u}_{1}^{1}(x)|\le C(\varepsilon+|x_1|^2)^{-1},\quad |p_1^1(x)-p_1^1(z_1,0)|\leq C(\varepsilon+|x_1|^2)^{-1/2},\quad x\in\Omega_{R},$$
		for some  point $z=(z_1,0)\in \Omega_{R}$ with $|z_1|=R/2$; and
		
		(ii) for any $m\ge 1$, 
		\begin{equation*}
			|\nabla^{m+1}{\bf u}_{1}^{1}(x)|+|\nabla^{m}{p}_{1}^{1}(x)| \le 
			C(\varepsilon+|x_1|^2)^{-\frac{m+2}{2}},\quad x\in\Omega_{R}.
		\end{equation*}
	\end{prop}
	\begin{proof}
		We follow the notation in Proposition \ref{prop2.1}. If $h_1(x_1)=h_2(x_1)=\frac{1}{2}h(x_{1})$, then $\delta(x_1)=\varepsilon+h(x_1)$, and the Keller-type function $k(x)=\frac{x_2}{\delta(x_1)}$ in $\Omega_{2R}$. We begin by choosing ${\bf v}_{1}^{1}=(({\bf v}_{1}^{1})^{(1)},({\bf v}_{1}^{1})^{(2)})$ with
		\begin{equation}\label{v111defg}
			({\bf v}_{1}^{1})^{(1)}(x)=\frac{x_2}{\delta(x_1)}+\frac{1}{2} 
			\quad\text{for}~
			x\in\Omega_{2R},
		\end{equation}
		and 
		\begin{equation}\label{v112defg}
			({\bf 
				v}_{1}^{1})^{(2)}(x)=-\int_{-\frac{1}{2}\delta(x_1)}^{x_2}\partial_{x_1}
			({\bf v}_{1}^{1})^{(1)}(x_1,y)\ dy \quad \text{for}~x\in\Omega_{2R}.
		\end{equation} 
		
		It is easy to check that
		\begin{equation}\label{divfree}
			\nabla\cdot{\bf v}_{1}^{1}=\partial_{x_1}({\bf 
				v}_{1}^{1})^{(1)}+\partial_{x_2}
			({\bf v}_{1}^{1})^{(2)}=0.
		\end{equation}
		At this moment, noting that 
		\begin{equation}\label{divfree2}
			\partial_{x_1}({\bf v}_{1}^{1})^{(1)}(x)=-\frac{h'(x_1)}
			{\delta(x_1)^2}x_2
		\end{equation} is an odd function with respect to $x_2$, so we have $({\bf 
			v}_{1}^{1})^{(2)}|_{\Gamma^{\pm}_{2R}}=0.$ Therefore, there is no need to solve function $F(x)$ and $G(x)$ as in Step I.1 in the proof of Proposition \ref{prop2.1}
		
		A direct calculation yields,  for $0\le k\le m+1$ and $s\ge 2$,
		\begin{equation}\label{v11highestig}
			|\partial_{x_1}^{k}({\bf v}_{1}^{1})^{(1)}|\le C\delta(x_1)^{-k/2},
			\quad|\partial_{x_1}^{k-1}\partial_{x_2}({\bf v}_{1}^{1})^{(1)}|\le 
			C\delta(x_1)^{-(k+1)/2},\quad\partial_{x_2}^{s}({\bf v}_{1}^{1})^{(1)}=0.
		\end{equation}
		Thus, by using \eqref{v112defg}--\eqref{v11highestig}, 
		\begin{equation*}
			|\partial_{x_1}({\bf v}_{1}^{1})^{(2)}|\le C, \quad |\partial_{x_2}({\bf v}_{1}^{1})^{(2)}|=|\partial_{x_1}({\bf v}_{1}^{1})^{(1)}|\le C\delta(x_1)^{-1/2}.
		\end{equation*}
		Similarly, for $2\le k\le m+2,$
		\begin{equation}\label{v112highestig}
			|\partial_{x_1}^{k}({\bf v}_{1}^{1})^{(2)}|\le 
			C\delta(x_1)^{-\frac{k-1}{2}}\quad \text{in}~\Omega_{2R}.
		\end{equation}
		
		Then we only need to choose
		\begin{equation}\label{barp1g}
			\bar{p}_{1}=\frac{\mu h'(x_1)}{\delta(x_1)^2}x_2\quad\text{in}
			~\Omega_{2R},
		\end{equation}
		By \eqref{v112defg}, \eqref{divfree2}, and \eqref{barp1g}, it is easy to 
		verify that
		\begin{equation}\label{barp1eqsg}
			\mu\partial_{x_2x_2}({\bf 
				v}_{1}^{1})^{(2)}-\partial_{x_2}{\bar{p}_{1}}=\partial_{x_2}(\mu\partial_{x_2}({\bf 
				v}_{1}^{1})^{(2)}-{\bar{p}_{1}})=\partial_{x_2}(-\mu\partial_{x_1}({\bf v}_{1}^{1})^{(1)}-{\bar{p}_{1}})=0
		\end{equation}
		and
		\begin{align*}
			|\partial_{x_1}\bar{p}_{1}|\le C\delta(x_1)^{-1}.
		\end{align*}
		Then, we have
		\begin{equation*}
			({\bf f}_{1}^{1})^{(1)}=\mu \partial_{x_1x_1}({\bf v}_{1}^{1})^{(1)}-
			\partial_{x_1}\bar{p}_{1},\quad({\bf f}_{1}^{1})^{(2)}=\mu 
			\partial_{x_1x_1}({\bf 
				v}_{1}^{1})^{(2)},
		\end{equation*}
		and by using \eqref{v11highestig} and  \eqref{v112highestig},
		\begin{equation*}
			|{\bf f}_{1}^{1}|\le C (|({\bf f}_{1}^{1})^{(1)}|+|({\bf 
				f}_{1}^{1})^{(2)}|) \le 
			C \delta(x_1)^{-1}\quad \text{in}~\Omega_{2R}.
		\end{equation*}
		
		For $k$-th order derivatives, it follows from \eqref{v112defg}, \eqref{divfree}, and 
		\eqref{v11highestig} that in $\Omega_{2R}$,
		\begin{equation}\label{v112highesti1g}
			|\partial_{x_1}^{k}\partial_{x_2}^{s}({\bf 
				v}_{1}^{1})^{(2)}|=|\partial_{x_1}
			^{k+1}\partial_{x_2}^{s-1}({\bf v}_{1}^{1})^{(1)}|\le C\delta(x_1)^{-
				\frac{k+2s-1}{2}}\,\text{for}~1\le s\le2,
		\end{equation}
		and for $s\ge 3$, 
		\begin{equation*}
			\partial_{x_2}^{s}({\bf v}_{1}^{1})^{(2)}=-\partial_{x_2}^{s-1}(\partial_{x_1}
			({\bf v}_{1}^{1})^{(1)})=0\quad\text{in}~\Omega_{2R};
		\end{equation*}
		and for $1\le k\le m+1,$ 
		\begin{align}\label{barphighestg}
			|\partial_{x_1}^{k}\bar{p}_{1}|\le C\delta(x_1)^{-\frac{k+1}{2}},\quad\,|
			\partial_{x_1}^{k-1}\partial_{x_2}{\bar{p}}_{1}|\le 
			C\delta(x_1)^{-\frac{k+2}{2}},
			\,\quad\,\partial_{x_2}^{s}{\bar{p}_{1}}=0\quad\text{for}~s\ge2.
		\end{align}
		
		Next, we inductively choose auxiliary functions  ${\bf v}_{1}^{l}(x)$ and $\bar{p}_{l}(x)$ to further reduce the upper bound of ${\bf f}_1^l$ in $\Omega_{2R}$. Different from the proof of Proposition \ref{prop2.1}, the following expression of  auxiliary functions by Green function is much simpler. Specifically, instead of \eqref{v1ldef}, we choose ${\bf v}_{1}^{l}=(({\bf v}_{1}^{l})^{(1)},({\bf v}_{1}^{l})^{(2)})^{T}$ with 
		\begin{equation}\label{v1kdef}
			\begin{split}
				({\bf v}_{1}^{l})^{(1)}(x)=&\,-\frac{1}{\mu}\int_{-\frac{1}{2}\delta(x_1)}^{
					\frac{1}{2}\delta(x_1)}G(x_2,y)({\bf f}_1^{l-1})^{(1)}(x_1,y) \ 
				dy, \\
				({\bf v}_{1}^{l})^{(2)}(x)=&\,-\int_{-\frac{1}{2}\delta(x_1)}^{x_2}
				\partial_{x_1}({\bf v}_{1}^{l})^{(1)}(x_1,y)\ dy
			\end{split}\quad \text{in}~\Omega_{2R},
		\end{equation}
		satisfying ${\bf v}_{1}^{l}(x)=0$ on $\Gamma^{+}_{2R}\cup\Gamma^{-}_{2R}$, where
		\begin{align*}
			G(x_2,y)=\frac{1}{\delta(x_1)}\left\{ \begin{array}{l}
				(y+\frac12h(x_1)+\frac{\varepsilon}{2})(x_2-\frac12h(x_1)-\frac{\varepsilon}{2}),
				\quad-\frac{\varepsilon}{2}-\frac12h(x_1)\le y\le x_2,\\\\
				(x_2+\frac12h(x_1)+\frac{\varepsilon}{2})(y-\frac12h(x_1)-\frac{\varepsilon}{2}),\quad
				x_2\le y\le \frac{\varepsilon}{2}+\frac12h(x_1).
			\end{array} \right.
		\end{align*} 
		Then it is easy to see that
		\begin{equation}\label{v1keq}
			\begin{split}
				\mu\partial_{.x_2x_2}({\bf v}_{1}^{l})^{(1)}=&\,-({\bf f}_{1}^{l-1})^{(1)},\\
				\nabla\cdot{\bf v}_{1}^{l}=&\,\partial_{x_1}({\bf v}_{1}^{l})^{(1)}+\partial_{x_2}
				({\bf v}_{1}^{l})^{(2)}=0.
			\end{split}
		\end{equation}
		
		Furthermore, we take 
		\begin{equation}\label{barpkdef}
			\bar{p}_{l}=\mu\int_{0}^{x_2} \Big(\partial_{x_1x_1}({\bf v}_{1}^{l-1})^{(2)}+
			\partial_{x_2x_2}({\bf v}_{1}^{l})^{(2)} \Big) (x_1,y) \ 
			dy\quad\text{in}~\Omega_{2R},
		\end{equation}
		such that
		\begin{equation}\label{barpkv1kv1k-1}
			\partial_{x_2}\bar{p}_{l}=\mu\partial_{x_1x_1}({\bf 
				v}_{1}^{l-1})^{(2)}+ \mu 
			\partial_{x_2x_2}({\bf v}_{1}^{l})^{(2)}.
		\end{equation}
		By \eqref{v1keq} and \eqref{barpkv1kv1k-1}, we have
		\begin{equation*}
			\begin{split}
				&({\bf f}_{1}^{l})^{(1)}=({\bf f}_{1}^{l-1})^{(1)} +\mu \Delta({\bf v}_{1}
				^{l})^{(1)}-\partial_{x_1}\bar{p}_{l}=\mu\partial_{x_1x_1}({\bf 
					v}_{1}^{l})^{(1)}-\partial_{x_1}\bar{p}_{l}, \\
				&({\bf f}_{1}^{l})^{(2)}=({\bf f}_{1}^{l-1})^{(2)}+\mu \Delta({\bf v}_{1}
				^{l})^{(2)}-\partial_{x_2}\bar{p}_l
				=\mu\partial_{x_1x_1}({\bf v}_{1}^{l})^{(2)}.
			\end{split}
		\end{equation*}
		
		We can inductively prove the following estimates, for $j\ge1$ and $j=1,2,
		\dots,m+1$,
		\begin{align}\label{vpfesti}
			&|({\bf v}_{1}^{j})^{(1)}|\le C\delta(x_1)^{j-1},\quad|({\bf v}_{1}^{j})^{(2)}|
			\le C\delta(x_1)^{(2j-1)/2},\quad|\bar{p}_{j}|\le 
			C\delta(x_1)^{(2j-3)/2}\quad \text{in}~\Omega_{2R},
		\end{align}
		and for $0\le k\le m+1,$ 
		\begin{align}\label{highvpfesti}
			\begin{split}
				&|\partial_{x_1}^{k}\partial_{x_2}^{s}({\bf v}_{1}^{j})^{(1)}|
				\le C\delta(x_1)^{(2j-2s-k-2)/2}\quad\text{for}~0\le s\le 2j-1,\\
				&\partial_{x_2}^{s}({\bf v}_{1}^{j})^{(1)}=0\quad\text{for}~s\ge 2j,\\
				&|\partial_{x_1}^{k}\partial_{x_2}^{s}({\bf v}_{1}^{j})^{(2)}|
				\le C\delta(x_1)^{(2j-2s-k-1)/2}\quad\text{for}~0\le s\le 2j,\\
				&\partial_{x_2}^{s}({\bf v}_{1}^{j})^{(2)}=0\quad\text{for}~s\ge 
				2j+1,\\
				&|\partial_{x_1}^{k}\partial_{x_2}^{s}\bar{p}_{j}|\le 
				C\delta(x_1)^{(2j-2s-k-3)/2}\quad\text{for}~s\le 2j-1,\\
				&\partial_{x_2}^{s}\bar{p}_{j}=0\quad\text{for}~s\ge 2j,~ j\ge2.
			\end{split}	
		\end{align} 
		Consequently,
		\begin{align*}
			\begin{split}
				&|({\bf f}_{1}^{m+1})^{(1)}|\le C \Big(|\partial_{x_1x_1}({\bf v}_{1}
				^{m+1})^{(1)}|+|\partial_{x_1}\bar{p}_{m+1}|\Big)\le C\delta(x_1)^{m-1},\\
				&|({\bf f}_{1}^{m+1})^{(2)}|\le C |\partial_{x_1x_1}({\bf v}_{1}^{m+1})^{(2)}|
				\le C \delta(x_1)^{m-\frac{1}{2}}.
			\end{split}
		\end{align*}
		Thus,
		$$|{\bf f}_{1}^{m+1}|\le C \delta(x_1)^{m-1} \quad\mbox{and}~	|\nabla^{s}{\bf f}_{1}^{m+1}|\le C \delta(x_1)^{m-s-1}\quad\mbox{for}~1\le s\le m.$$
		
		Denote
		\begin{equation*}
			{\bf v}^{m+1}(x)=\sum_{l=1}^{m+1}{\bf v}_{1}^{l}(x),\quad \bar{p}^{m+1}(x)=\sum_{l=1}^{m+1}\bar{p}_{l}(x).
		\end{equation*}
		It is easy to verify that ${\bf v}^{m+1}(x)={\bf u}_1^1(x)$ on $\Gamma_{2R}^{\pm}$, $\nabla\cdot{\bf v}^{m+1}(x)=0$ in $\Omega_{2R}$, and for $1\le s\le m$,
		\begin{equation*}
			|{\bf f}^{m+1}(x)|=|\mu\Delta{\bf v}^{m+1}(x)-\nabla \bar{p}^{m+1}(x)|\leq C \delta(x_1)^{m-1},\quad |\nabla^s{\bf f}^{m+1}(x)|\leq C \delta(x_1)^{m-s-1}.
		\end{equation*}
		Thus, by virtue of Proposition \ref{prop3.3}, it holds that
		\begin{equation}\label{u11-vkestg}
			\|\nabla^{m+1}({\bf u}_{1}^{1} -{\bf v}^{m+1})\|
			_{L^{\infty}(\Omega_{\delta(x_1)/2}(x_1))} +\|\nabla^{m} (p_{1}^{1}-
			\bar{p}^{m+1})\|_{L^{\infty}(\Omega_{\delta(x_1)/2}(x_1))} 
			\le C.
		\end{equation}
		
		By \eqref{v1kdef}, \eqref{barpkdef}, \eqref{vpfesti}, and \eqref{highvpfesti}, \eqref{v1lest} becomes 
		\begin{align}\label{v1lestg}
			\begin{split}
				&|\nabla^{m+1}{\bf v}_{1}^{l}|\le C\delta(x_1)^{-\frac{m+2}{2}}\quad\text{for}~l\le 
				\frac{m+1}{2},\\
				&|\nabla^{m+1}{\bf v}_{1}^{l}|\le C\delta(x_1)^{l-m-2}
				\quad\text{for}~\frac{m+2}{2}\le l\le m+1,\\
				&|\nabla^{m}\bar{p}_{l}|\le C \delta(x_1)^{-\frac{m+2}{2}}\quad\text{for}~ 
				l\le \frac{m}{2},\\
				&|\nabla^{m}\bar{p}_{l}|\le C\delta(x_1)^{\frac{2l-2m-3}
					{2}}\quad\text{for}~\frac{m+1}{2}\le l\le m+1.
			\end{split}
		\end{align}
		It follows from \eqref{v1lestg} that
		\begin{equation}\label{zgjdc}
			|\nabla{\bf v}^{m+1}(x)|\leq C\delta(x_1)^{-1}
		\end{equation}
		and for $m\ge 1$,
		\begin{equation}\label{zgj1dc}
			|\nabla^{m+1}{\bf v}^{m+1}(x)|+|\nabla^{m}\bar{p}^{m+1}(x)|\leq C\delta(x_1)^{-\frac{m+2}{2}}.
		\end{equation}
		
		From \eqref{u11-vkestg}, \eqref{zgjdc}, and \eqref{zgj1dc}, we have
		\begin{equation*}
			|\nabla{\bf u}_{1}^{1}(x)|\leq C\delta(x_1)^{-1}\quad\text{in}~\Omega_{R},
		\end{equation*}
		and
		\begin{equation*}
			|\nabla^{m+1}{\bf u}_{1}^{1}(x)|+	|\nabla^{m}{p}_{1}^{1}(x)|\le C \delta(x_1)^{-\frac{m+2}{2}} \quad\text{in}~\Omega_{R}.
		\end{equation*}
		
		For any point $(z_1,0)\in\Omega_{R}$ with $|z_1|=R/2$, because $\bar{p}_{1}$ and $\bar{p}_{2}$ are odd in $x_2$, we have
		\begin{align}\label{p11minus} 
			p_{1}^{1}(x_1,0)-p_{1}^{1}(z_1,0)=p_{1}^{1}(x_1,0)-(\bar{p}_{1}+\bar{p}_{2})(x_1,0)-p_{1}^{1}(z_1,0)+(\bar{p}_{1}+\bar{p}_{2})(z_1,0).
		\end{align}
		It follows from \eqref{u11-vkestg} and \eqref{p11minus} that
		\begin{equation}\label{p11esti}
			|p_{1}^{1}(x_1,0)-p_{1}^{1}(z_1,0)| \le C .
		\end{equation}
		By the mean value theorem, using \eqref{p11esti} and the estimates of $|\nabla p_{1}|$, we obtain
		\begin{align*}
			|p_{1}^{1}(x)-p_{1}^{1}(z_1,0)|=&\,|p_{1}^{1}(x)-p_{1}^{1}(x_{1},0)|+|p_{1}^{1}
			(x_{1},0) -p_{1}^{1}(z_1,0)|\nonumber\\
			\leq&\, \delta(x_1)|\nabla p_{1}^{1}|+|p_{1}^{1}(x_{1},0) -p_{1}^{1}(z_1,0)|\le C 
			\delta(x_1)^{-1/2}.
		\end{align*}
		This completes the proof of Proposition \ref{prop2.1g}.
	\end{proof}
	
	\section{Estimates for $({\bf u}_{1}^{2},p_{1}^{2})$}\label{sec_u2}
	
	In this section, we derive the higher-order derivatives of \({\bf u}_1^2\) and \(p_1^2\) using the method outlined in Section \ref{sec_u1}. However, since \(\boldsymbol{\psi}_2 = (0, 1)^T\) introduces a higher singularity compared to \(\boldsymbol{\psi}_1 = (1, 0)^T\), new difficulties arise in constructing the first auxiliary function pair \(({\bf v}_2^1, \bar{p}_1)\). To address this, we split the process into two parts: \((\tilde{\bf v}_2^1, \tilde{p}_1)\) and \((\hat{\bf v}_2^1, \hat{p}_1)\). In comparison to \eqref{v11def}, in order to ensure \(\nabla \cdot \tilde{\bf v}_2^1 = 0\), the correction terms \(F(x)\) and \(G(x)\) become larger, resulting in \( |\mu\Delta\tilde{\bf v}_2^1 - \nabla\tilde{p}_1| \leq C\delta(x_1)^{-3/2} \), which produces a larger bound. Therefore, to fully capture the singularities of \(\nabla{\bf u}_1^2\), we introduce supplementary terms \(\hat{\bf v}_2^1\) and \(\hat{p}_1\) such that \({\bf v}_2^1 := \tilde{\bf v}_2^1 + \hat{\bf v}_2^1\) and \(\bar{p}_1 := \tilde{p}_1 + \hat{p}_1\), reducing the bound to \( |\mu\Delta{\bf v}_2^1 - \nabla\bar{p}_1| \leq C \delta(x_1)^{-1} \).
	
	To estimate higher-order derivatives, we continue to use the induction method as before. Since the leading terms of \( \sum_{i=1}^{l-1} (\mu \Delta {\bf v}_2^i - \nabla \bar{p}_i)(x) \) appear in the second component, we first construct \(\tilde{p}_l\) to remove this singular term, then build \({\bf v}_2^l\) and the corresponding \(\hat{p}_l\) to handle the leading terms in the first component. This ensures that \( \sum_{i=1}^{l-1} (\mu \Delta {\bf v}_2^i - \nabla \bar{p}_i)(x) + \mu \Delta {\bf v}_2^l - \nabla (\tilde{p}_l + \hat{p}_l)(x) \) has the desired properties, as outlined in Step II below.

	\begin{prop}\label{prop2.2}
		Under the same assumption as in Theorem \ref{main thm1n}, let ${\bf u}_{1}^{2}$ and $p_{1}^{2}$ be the solution to \eqref{u,peq1} with $\alpha=2$. Then for sufficiently small $0<\varepsilon<1/2$, we have
		$$\mbox{(i) }\quad~|\nabla{\bf u}_{1}^{2}(x)|\le C(\varepsilon+|x_1|^2)^{-\frac32},\quad |p_1^2(x)-p_1^2(z_1,0)|\leq C(\varepsilon+|x_1|^2)^{-2},\quad x\in\Omega_{R},\quad$$ 
		for some point $z=(z_1,0)\in \Omega_{R}$, $|z_1|=R/2$; and 
		
		(ii) for any $m\ge 1$, 
		\begin{equation*}
			|\nabla^{m+1}{\bf u}_{1}^{2}(x)|+|\nabla^{m}{p}_{1}^{2}(x)| \le 
			C(\varepsilon+|x_1|^2)^{-\frac{m+4}{2}},\quad x\in\Omega_{R}.
		\end{equation*}
	\end{prop}
	
	\begin{proof}[Proof of Proposition \ref{prop2.2}] 
As in the proof of Proposition \ref{prop2.1}, we may assume that $x_1\ge 0$.
        
		{\bf Step I. Gradient estimate.} 
		Similar to Section \ref{sec_u1}, in order to establish the gradient estimate of ${\bf u}_1^2$, we need to construct a divergence-free auxiliary function ${\bf v}_{2}^{1}$ which satisfies the boundary condition ${\bf v}_{2}^{1}(x)={\bf u}_{1}^{2}(x)=\boldsymbol{\psi}_{2}$ on $\Gamma^{+}_{2R}$, ${\bf v}_{2}^{1}(x)={\bf u}_{1}^{2}(x)=0$ on $\Gamma^{-}_{2R}$, 
		and a corresponding $\bar{p}_{1}$, such that
		$$|\mu\Delta{\bf v}_{2}^{1}-\nabla\bar{p}_{1}|\leq C\delta(x_1)^{-1}.$$
		However, $\boldsymbol{\psi}_{2}=(0,1)^{T}$ introduce in a higher singularity than $\boldsymbol{\psi}_{2}=(1,0)^{T}$. We divide the process into two steps to construct
		$${\bf v}_{2}^{1}:=\tilde{\bf v}_{2}^{1}+\hat{\bf v}_{2}^{1}\quad\mbox{and}\quad \bar{p}_{1}:=\tilde{p}_{1}+\hat{p}_{1}.$$
		
		{\bf Step I.1. Construction of $\tilde{\bf v}_{2}^{1}$ and $\tilde{p}_{1}$.} 
		We begin by constructing an auxiliary function $\tilde{\bf v}_{2}^{1}$ that satisfies the boundary conditions $\tilde{\bf v}_{2}^{1}(x)={\bf u}_{1}^{2}(x)=\boldsymbol{\psi}_{2}$ on $\Gamma^{+}_{2R}$ and $\tilde{\bf v}_{2}^{1}(x)={\bf u}_{1}^{2}(x)=0$ on $\Gamma^{-}_{2R}$. The construction process of $\tilde{\bf v}_{2}^{1}(x)$ is similar to that of ${\bf v}_{1}^{1}(x)$ in Proposition \ref{prop2.1}, with some minor modifications. Specifically, we assume that $\tilde{\bf v}_{2}^{1}(x)$ takes the following form:
		\begin{equation}\label{tildev21def}
			\tilde{\bf v}_{2}^{1}(x)=
			\begin{pmatrix}
				0\\k(x)+\frac{1}{2}
			\end{pmatrix}+ 
			\begin{pmatrix}
				F(x)\\
				G(x)
			\end{pmatrix}\Big(k(x)^2-\frac{1}{4}\Big)\quad\text{in}~\Omega_{2R},
		\end{equation}
		where $k(x)=\frac{x_2-\frac12(h_1-h_2)(x_{1})}{\delta(x_1)}$ is defined in \eqref{def_k}. 
		
		According to \eqref{kbc}, we can see that $\tilde{\bf v}_{2}^{1}(x)={\bf u}_{1}^{2}(x)$ on $\Gamma^{+}_{2R}\cup\Gamma^{-}_{2R}$. The additional term on the right-hand side of equation \eqref{tildev21def} is used to adjust the divergence of the first term, allowing $\nabla\cdot\tilde{\bf v}_{2}^{1}(x)=0$. By a direct calculation, we have
		\begin{equation}\label{u12fz1}
			\begin{split}
				&\partial_{x_{1}}(\tilde{\bf v}_{2}^{1})^{(1)}=\partial_{x_{1}} F(x)\Big(k(x)^2-\frac{1}{4}\Big)+2F(x)k(x)\partial_{x_{1}} k(x),\\
				&\partial_{x_{2}}(\tilde{\bf v}_{2}^{1})^{(2)}=\frac{1}{\delta(x_1)}+\partial_{x_{2}} G(x)\Big(k(x)^2-
				\frac{1}{4}\Big)+\frac{2k(x)}{\delta(x_1)}G(x).
			\end{split}
		\end{equation}
		To determine $F(x)$ and $G(x)$, instead of \eqref{ffz1}, we set 
		\begin{equation}\label{u12fz2}
			\frac14(\partial_{x_{1}}F+\partial_{x_{2}}G)(x)=\frac{1}{\delta(x_1)}.
		\end{equation}
		By using \eqref{u12fz1}, \eqref{u12fz2}, and the divergence-free condition $\partial_{x_{1}}({\bf v}_{2}^{1})^{(1)}+\partial_{x_{2}}({\bf v}_{2}^{1})^{(2)}=0$, we obtain
		\begin{equation}\label{u12fz3}
			G(x)=-2k(x)-F(x)\delta(x_1)\partial_{x_{1}}k(x).
		\end{equation}
		
		Based on \eqref{u12fz2} and \eqref{u12fz3}, we conclude that $F(x)$ satisfies the following first-order differential equation:
		\begin{equation*}
			\partial_{x_{1}}F(x)+\frac{(h'_1+h'_2)(x_1)}{\delta(x_1)}F(x)-
			\delta(x_1)
			\partial_{x_1}k(x)\partial_{x_{2}}F(x)=\frac{6}
			{\delta(x_1)}.
		\end{equation*}
		Instead of \eqref{deff}, we set
		\begin{equation}\label{deffv21}
			F(x)=F(x_{1})=\frac{6x_1}{\delta(x_1)},
		\end{equation}
		and by substituting it into \eqref{u12fz3}, we obtain
		\begin{equation}\label{defgv21}
			G(x)=-2k(x)-6x_1\partial_{x_{1}}k(x).
		\end{equation}
		Thus, $F(x)$ and $G(x)$ in \eqref{tildev21def} are fixed. 
		Therefore, we have
		\begin{equation*}
			\tilde{\bf v}_{2}^{1}(x)=
			\begin{pmatrix}
				0\\k(x)+\frac{1}{2}
			\end{pmatrix}+ 
			\begin{pmatrix}
				\frac{6x_1}{\delta(x_1)}\\
				-2k(x)-6x_1\partial_{x_{1}}k(x)
			\end{pmatrix}\Big(k(x)^2-\frac{1}{4}\Big)\quad\text{in}~\Omega_{2R},
		\end{equation*}
		which satisfies the divergence-free condition $\nabla\cdot\tilde{\bf v}_{2}^{1}(x)=0$ in $\Omega_{2R}$. 
		
		A direct calculation yields, 
		\begin{equation*}
			\begin{split}
				&|\partial_{x_{1}}(\tilde{{\bf v}}_2^1)^{(1)}|\leq C\delta(x_1)^{-1},\quad |\partial_{x_{2}}(\tilde{{\bf v}}_2^1)^{(1)}|\leq C\delta(x_1)^{-3/2},\\
				&|\partial_{x_{1}}(\tilde{{\bf v}}_2^1)^{(2)}|\leq C\delta(x_1)^{-1/2},\,|\partial_{x_{2}}(\tilde{{\bf v}}_2^1)^{(2)}|\leq C\delta(x_1)^{-1};
			\end{split}
		\end{equation*}
		and for $k\ge 1$, $0\le s\le i+1$, and $i=1,2$, 
		\begin{equation}\label{xjyy}
			|\partial_{x_{1}}^k\partial_{x_{2}}^s(\tilde{{\bf v}}_2^1)^{(i)}|\leq 
			C\delta(x_1)^{\frac{i}{2}-\frac{k}{2}-s-1}, \quad\partial_{x_{2}}^{i+2}
			(\tilde{{\bf v}}
			_2^1)^{(i)}=0.
		\end{equation}
		Here we note that $\partial_{x_2x_2}(\tilde{\bf v}_{2}^{1})^{(1)}$ is of order $\delta(x_1)^{-5/2}$ and $\partial_{x_2x_2} (\tilde{\bf v}_{2}^{1})^{(2)}$ is of order $\delta(x_1)^{-2}$, which are two leading terms of $\Delta\tilde{\bf v}_{2}^{1}$. In order to eliminate these two leading terms, we introduce         
		\begin{equation}\label{tildep21}
			\tilde{p}_{1}=-12\mu\int_{x_1}^{R}\frac{y}{\delta(y)^{3}}\ dy+ \mu\partial_{x_{2}}(\tilde{\bf v}_{2}^{1})^{(2)}
			\quad\text{in}~\Omega_{2R},
		\end{equation}
		such that
		\begin{equation}\label{fz1}
			|\partial_{x_{1}}\tilde{p}_{1}-\mu\partial_{x_2x_2}(\tilde{\bf v}_{2}^{1})^{(1)}|=|\mu\partial_{x_1x_2}(\tilde{\bf v}_{2}^{1})^{(2)}|\leq C\delta(x_1)^{-3/2},\quad\partial_{x_{2}}\tilde{p}_{1}=\mu\partial_{x_2x_2}(\tilde{\bf v}_{2}^{1})^{(2)}. 
		\end{equation}
		
		From \eqref{tildev21def}, \eqref{deffv21}, and \eqref{defgv21}, we see that $(\tilde{\bf v}_2^1)^{(1)}$ is a quadratic polynomial in $x_2$, while $(\tilde{\bf v}_2^1)^{(2)}$ is a cubic polynomial in $x_2$. Then, by \eqref{tildev21def} and \eqref{fz1}, we write
		\begin{equation}\label{tildev21-gradp1}
			\begin{aligned}
				\mu\Delta\tilde{\bf v}_{2}^{1}-\nabla\tilde{p}_{1}=&\,	\mu\begin{pmatrix}
					\partial_{x_1}(\partial_{x_{1}}(\tilde{\bf v}_2^1)^{(1)}-\partial_{x_{2}}(\tilde{\bf v}_2^1)^{(2)})\\\\
					\partial_{x_1x_1}(\tilde{\bf v}_{2}^{1})^{(2)}
				\end{pmatrix}\\
				:=&\,\begin{pmatrix}
					A(x_1)x_{2}^{2}+B(x_1)x_2+C(x_1) \\\\
					\hat{A}(x_1)x_{2}^{3}+\hat{B}(x_1)x_{2}^2+\hat{C}(x_1)x_{2}+\hat{D}(x_1)
				\end{pmatrix},	    
			\end{aligned}	
		\end{equation}
		where for $0\le s\le m$, in $\Omega_{2R}$,
		\begin{equation}\label{ABCDest}
			\begin{split}
				&|A^{(s)}|\le C\delta(x_1)^{-\frac{7+s}{2}},\,\,|B^{(s)}|\le C\delta(x_1)^{-\frac{5+s}
					{2}},\,\,|C^{(s)}|\le C\delta(x_1)^{-\frac{3+s}{2}},\\
				&|\hat{A}^{(s)}|\le C\delta(x_1)^{-\frac{8+s}{2}},\,\,|\hat{B}^{(s)}|\le 
				C\delta(x_1)^{-\frac{6+s}{2}},\\
				&|\hat{C}^{(s)}|\le C\delta(x_1)^{-\frac{4+s}{2}},\,\,|\hat{D}^{(s)}|\le 
				C\delta(x_1)^{-\frac{2+s}
					{2}}.
			\end{split}
		\end{equation}
		By using \eqref{tildev21-gradp1} and \eqref{ABCDest}, we have
		\begin{equation*}
			|(\mu\Delta\tilde{\bf v}_{2}^{1}-\nabla\tilde{p}_{1})^{(1)}|\leq C\delta(x_1)^{-3/2},\quad |(\mu\Delta\tilde{\bf v}_{2}^{1}-\nabla\tilde{p}_{1})^{(2)}|\leq C\delta(x_1)^{-1}.
		\end{equation*}
		Thus, we have $|\mu\Delta\tilde{\bf v}_{2}^{1}-\nabla\tilde{p}_{1}|\le C\delta(x_1)^{-3/2}$. However, this bound is too large to apply Proposition \ref{prop3.3}.  Therefore, in order to fully capture the singularity in $\nabla{\bf u}_1^2$, we need to introduce additional divergence-free correctors to reduce the upper bound, $\delta(x_1)^{-3/2}$.
		
		{\bf Step I.2. Construction of $\hat{\bf v}_{2}^{1}$ and $\hat{p}_{1}$.} We choose the correctors in the following form in $\Omega_{2R}$:
		\begin{equation}\label{hatv21def}
			\hat{{\bf v}}_{2}^{1}(x)=
			\begin{pmatrix}
				F_{11}^{2}(x_1)x_2^{2}+F_{11}^{1}(x_1)x_2+F_{11}^{0}(x_1)+\tilde{F}
				_{11}(x_1) \\\\
				F_{21}^{3}(x_1)x_2^{3}+F_{21}^{2}(x_1)x_2^2+F_{21}^{1}(x)x_2+F_{21}
				^{0}(x_1)+\tilde{F}_{21}(x)
			\end{pmatrix}\Big(k(x)^2-\frac{1}{4}\Big).
		\end{equation}
		From \eqref{tildev21-gradp1}, we can see that $(\mu\Delta\tilde{\bf v}_{2}^{1}-\nabla\tilde{p}_{1})^{(1)}$ is a quadratic polynomial in $x_2$. Hence, we choose $F_{11}^{i}(x_1)$, $i=0,1,2$, such that
		\begin{equation}\label{part22cancel}
			\mu	\partial_{x_2x_2}\Big(\sum_{i=0}^{2}F_{11}^{i}(x_1)x_2^i\Big(k(x)^2-\frac14\Big)\Big)=-(\mu\Delta\tilde{\bf v}_{2}^{1}-\nabla\tilde{p}_{1})^{(1)}.
		\end{equation}
		By comparing the coefficients of each term of $x_2$, we obtain
		\begin{equation*}
			\begin{split}
				&F_{11}^{2}(x_1)=-\frac{\delta(x_1)^2}{12\mu}A(x_1),\,\, F_{11}^{1}(x_1)=-
				\frac{\delta(x_1)^{2}}{6\mu}B(x_1)+(h_{1}-h_{2})F_{11}^{2}(x_1), \\ 
				&F_{11}^{0}(x_1)=-\frac{\delta(x_1)^2}{2\mu}C(x_1)+(h_{1}-h_{2})F_{11}^{1}(x_1)+
				\frac14(\varepsilon+2h_1(x_{1}))(\varepsilon+2h_2(x_{1}))F_{11}^{2}(x_1).
			\end{split}
		\end{equation*}
		Thus, we successfully remove the leading terms in $\mu\Delta\tilde{\bf v}_{2}^{1}-\nabla\tilde{p}_{1}$. 
		
		To ensure $\nabla\cdot\hat{{\bf v}}_{2}^{1}(x)=0$, we further take
		\begin{align*}
			&F_{21}^{3}(x_1)=-\frac{\delta(x_1)^{2}}{5}\Big(\frac{F_{11}^{2}(x_1)}
			{\delta(x_1)^{2}}\Big)',\\ &F_{21}^{2}(x_1)=\frac{\delta(x_1)^2}{4}
			\Big(\frac{(h_1-h_2)(x_{1})F_{11}^{2}(x_1)}{\delta(x_1)^2}-\frac{F_{11}^{1}(x_1)}
			{\delta(x_1)^2}\Big)'+(h_1-h_2)(x_{1})F_{21}^{3}(x_1),\\
			&F_{21}^{1}(x_1)=\frac{\delta(x_1)^2}{3}\Big(\frac{(h_1-h_2)(x_{1})F_{11}^{1}(x_1)}
			{\delta(x_1)^2}+\frac{(\varepsilon+2h_1(x_{1}))(\varepsilon+2h_2(x_{1}))}{4\delta(x_1)^2}
			F_{11}^{2}(x_1)-\frac{F_{11}^{0}(x_1)}{\delta(x_1)^2}\Big)' \\
			&\qquad\qquad\qquad+(h_1-h_2)(x_{1})F_{21}^{2}(x_1)+
			\frac14(\varepsilon+2h_1(x_{1}))(\varepsilon+2h_2(x_{1}))F_{21}^{3}(x_1),\\
			&F_{21}^{0}(x_1)=\frac{\delta(x_1)^2}{2}\Big(\frac{(h_1-h_2)(x_{1})F_{11}^{0}
				(x_1)}{\delta(x_1)^2}+\frac{(\varepsilon+2h_1(x_{1}))(\varepsilon+2h_2(x_{1}))}
			{4\delta(x_1)^2}F_{11}^{1}(x_1)\Big)'\\
			&\qquad\qquad\qquad+(h_1-h_2)(x_{1})F_{21}^{1}(x_1)+\frac14(\varepsilon+2h_1(x_{1}))
			(\varepsilon+2h_2(x_{1}))F_{21}^{2}(x_1),
		\end{align*}
		such that
		\begin{equation*}
			\nabla\cdot
			\begin{pmatrix}
				\sum_{i=0}^{2}F_{11}^{i}(x_1)x_2^i\Big(k(x)^2-\frac{1}{4}\Big) \\
				\sum_{i=0}^{3}F_{21}^{i}(x_1)x_2^i\Big(k(x)^2-\frac{1}{4}\Big)
			\end{pmatrix}:=R(x_1).
		\end{equation*}
		Here
		\begin{align*}R(x_1)=&\,-\frac{(h_1-h_2)(x_{1})F_{21}^{0}(x_1)}{\delta(x_1)^2}-
			\Big(\frac{(\varepsilon+2h_1(x_{1}))(\varepsilon+2h_2(x_{1}))}{4\delta(x_1)^2}F_{11}^0(x_1)
			\Big)'\\
			&\,-\frac{(\varepsilon+2h_1(x_{1}))(\varepsilon+2h_2(x_{1}))}{4\delta(x_1)^2}F_{21}^{1}
			(x_1)\end{align*}
		is the $0$-th order term in the polynomial of $x_2$. By the same way as in Proposition \ref{prop2.1}, \eqref{yxws}--\eqref{dy3}, we set
		\begin{equation*}
			\begin{split}
				&\tilde{F}_{11}(x_1)=\frac{6\int_{0}^{x_1}\delta(y)R(y)\ dy}
				{\delta(x_1)},\\
				&\tilde{F}_{21}(x)=-2\delta(x_1)R(x_1)k(x)-\delta(x_1)\partial_{x_1}
				k(x)\tilde{F}_{11}(x_1):=P(x_1)x_2+Q(x_1),
			\end{split}
		\end{equation*} 
		such that
		\begin{equation*}
			\nabla\cdot
			\begin{pmatrix}
				\tilde{F}_{11}(x_1)\Big(k(x)^2-\frac{1}{4}\Big) \\
				\tilde{F}_{21}(x)\Big(k(x)^2-\frac{1}{4}\Big)
			\end{pmatrix}=-R(x_1).
		\end{equation*}
		Thus, we have $\nabla\cdot	\hat{{\bf v}}_{2}^{1}(x)=0$. 
		
		By \eqref{h1h14} and \eqref{ABCDest}, a calculation gives, for $k\ge0,\,i=0,1,2,
		$ and $j=0,1,2,3$, 
		\begin{equation}\label{fz2}
			\begin{split}
				&|F_{11}^{i(k)}|\leq C\delta(x_1)^{\frac12-i-\frac{k}{2}},\,\,\,\, |
				F_{21}^{j(k)}|\leq 
				C\delta(x_1)^{1-j-\frac{k}{2}},\,\, |\tilde{F}_{11}^{(k)}|\leq C 
				\delta(x_1)^{\frac{1-k}{2}}, \\ &|P(x_1)^{(k)}|\le C \delta(x_1)^{-\frac{k}{2}},
				\,\,\,\,|Q(x_1)^{(k)}|\le C\delta(x_1)^{\frac{2-k}{2}}
			\end{split}\quad \text{in}~\Omega_{2R}.
		\end{equation}
		Now we define ${\bf v}_{2}^{1}(x)=\tilde{{\bf v} }_{2}^{1}(x)+\hat{{\bf v}}_{2}^{1}(x),$ which satisfies $\nabla\cdot{\bf v}_{2}^{1}(x)=\nabla\cdot\tilde{{\bf v} }_{2}^{1}(x)+\nabla\cdot\hat{{\bf v}}_{2}^{1}(x)=0.$ Moreover, it follows from \eqref{hatv21def} and \eqref{fz2} that $2\mu\delta(x_1)^{-2}\tilde{F}_{11}(x_1)$ is the leading term in $\mu\Delta{\bf v}_{2}^{1}(x)-\nabla\tilde{p}_1$, which is of order $\delta(x_1)^{-\frac32}$. 
		
		To further cancel this leading term, we set
		\begin{equation}\label{hatp21}
			\hat{p}_{1}=2\mu\int_{0}^{x_1}\frac{\tilde{F}_{11}(y)}{\delta(y)^2}\ 
			dy\quad\text{in}~\Omega_{2R},
		\end{equation}
		such that
		\begin{equation}\label{F113hatp21}
			\partial_{x_1}
			\hat{p}_1=\frac{2\mu}{\delta(x_1)^2}\tilde{F}_{11}(x_1)=\mu\partial_{x_2x_2}\Big(k(x)^2-\frac{1}{4}\Big)\tilde{F}_{11}(x_1),\quad\partial_{x_2}
			\hat{p}_1=0.
		\end{equation}
		Now we define $\bar{p}_{1}=\tilde{p}_{1}+\hat{p}_{1}$.
		
		Denote 
		$${\bf f}_{2}^{1}:=\mu\Delta{\bf v}_{2}^{1}-\nabla\bar{p}_{1}\quad\text{in}~\Omega_{2R}.$$ 
		By means of \eqref{tildev21-gradp1}, \eqref{hatv21def}, \eqref{part22cancel}, and 
		\eqref{F113hatp21}, we have
		\begin{equation*}
			{\bf f}_{2}^{1}=\begin{pmatrix}
				\mu\partial_{x_1x_1}(\hat{\bf v}_{2}^{1})^{(1)}(x)\\\\
				\mu \partial_{x_1x_1}
				(\hat{\bf v}_{2}^{1})^{(2)}(x)+\mu\Delta(\tilde{\bf v}_{2}^{1})^{(2)}(x)-\partial_{x_2}\tilde{p}_{1}(x)+\mu\partial_{x_2x_2}(\hat{\bf v}_{2}^{1})^{(2)}(x)
			\end{pmatrix}.
		\end{equation*}
		From \eqref{hatv21def}, it follows that $(\hat{\bf v}_2^1)^{(1)}$ is a polynomial of order $4$ in $x_2$, while $(\hat{\bf v}_2^1)^{(2)}$ is a polynomial of order $5$ in $x_2$.
		
		In view of \eqref{tildev21-gradp1}, we can rewrite 
		\begin{equation}\label{f211,f212}
			({\bf f}_{2}^{1})^{(1)}=\sum_{i=0}^{4}H_{i1}(x_1)x_{2}^{i},
			\quad({\bf f}_{2}^{1})^{(2)}=\sum_{i=0}^{3}S_{i1}(x_1)x_{2}^{i}+
			\sum_{i=0}^{5}G_{i1}(x_1)x_{2}^{i}\quad \text{in}~\Omega_{2R},
		\end{equation}
		where 
		\begin{equation*}
			\sum_{i=0}^{3}S_{il}(x_1)x_2^{i}:= \mu\Delta(\tilde{\bf v}
			_{2}^{1})^{(2)}-\partial_{x_2}\tilde{p}_{1}+\mu\partial_{x_2x_2}(\hat{\bf v}
			_{2}^{1})^{(2)},\quad
			\sum_{i=0}^{5}G_{il}(x_1)x_{2}^{i}:=\mu \partial_{x_1x_1}
			(\hat{\bf v}_{2}^{1})^{(2)}.
		\end{equation*}
		By using \eqref{ABCDest} and \eqref{fz2}, for 
		$s\ge0,$ $1\le i\le 4,$ $0\le j\le3$, and $0\le k\le5,$
		\begin{equation}\label{GH1esti}
			|H_{i1}^{(s)}|\le C\delta(x_1)^{-\frac12-i-\frac{s}{2}},\quad
			|S_{j1}^{(s)}|\le C \delta(x_1)^{-1-j-\frac{s}{2}},\quad |G_{k1}^{(s)}|
			\le C \delta(x_1)^{-k-\frac{s}{2}}.
		\end{equation}
		Therefore, by \eqref{f211,f212} and \eqref{GH1esti}, we have
		\begin{equation*}
			|({\bf f}_{2}^{1})^{(1)}|\le C\delta(x_1)^{-1/2}, \quad|({\bf f}_{2}
			^{1})^{(2)}|\le C\delta(x_1)^{-1}\quad \text{in}~\Omega_{2R}.
		\end{equation*}
		Thus,
		\begin{equation}\label{f21esti}
			|{\bf f}_{2}^{1}|\le C (|({\bf f}_{2}^{1})^{(1)}|+|({\bf f}_{2}^{1})^{(2)}|) \le 
			C \delta(x_1)^{-1}\quad \text{in}~\Omega_{2R}.
		\end{equation}
		
		By applying Proposition \ref{prop3.3} and using \eqref{f21esti}, we have for any $x\in \Omega_{R}$,
		$$\| \nabla {\bf u}_{1}^{2}-\nabla{\bf v}_{1}^{2}\|_{L^{\infty}
			(\Omega_{\delta(x_1)/2}(x_1))}\le C. $$
		By virtue of \eqref{tildev21def}, \eqref{tildep21}, and \eqref{hatv21def}--\eqref{hatp21}, 
		we have
		\begin{equation}\label{v21high}
			\begin{split}
				&|\nabla^l {\bf v}_2^1|\le C\delta(x_1)^{-\frac{2l+1}{2}}\quad\text{for}~l\le 2,
				\quad |\nabla^l {\bf v}_2^1|\le C\delta(x_1)^{-\frac{l+3}{2}}\quad \text{for}~ 
				l\ge 3,\\
				& |\bar{p}_{1}|\le C\delta(x_1)^{-2},\quad|\nabla^l \bar{p}_1|\le C\delta(x_1)^{-
					\frac{l+4}{2}}\quad \text{for}~ l\ge 1.
			\end{split}
		\end{equation}
		Thus, Proposition \ref{prop2.2} holds for $m=0$, with
		$$|\nabla {\bf u}_{1}^{2}|\le C|\nabla {\bf v}_{2}^{1}(x)| +C\le 
		C\delta(x_1)^{-3/2}\quad\text{in}~\Omega_{R}.$$
		
		{\bf Step II. higher-order derivatives estimates.} Denote
		\begin{equation}\label{deffl}
			{\bf f}_{2}^{j}:={\bf f}_{2}^{j-1}+\mu\Delta{\bf v}_{2}^{j}(x)-\nabla\bar{p}_{j}(x)=\sum_{l=1}^{j}(\mu\Delta{\bf v}_{2}^{l}-\nabla\bar{p}_{l}
			)(x)\quad\text{for}~ 2\le j\le m+1.
		\end{equation}
		We will inductively prove that for $j\ge 1$, ${\bf f}_{2}^{j}$ can be represented as polynomials in $x_2$ as follows:
		\begin{equation}\label{gn1}
			({\bf f}_{2}^{j})^{(1)}=\sum_{i=0}^{2j+2}H_{ij}(x_1)x_2^{i}, \quad
			({\bf f}_{2}^{j})^{(2)}=\sum_{i=0}^{2j+1}S_{ij}(x_1)x_2^{i}+\sum_{i=0}
			^{2j+3}G_{ij}(x_1)x_{2}^{i}\quad\text{in}~\Omega_{2R},
		\end{equation}
		where for any $k\ge 0$,
		\begin{equation}\label{gn2}
			|H_{ij}^{(k)}(x_1)|\le C\delta(x_1)^{j-i-\frac{k}{2}-\frac{3}{2}},\,|S_{ij}
			^{(k)}|\le C\delta(x_1)^{j-i-\frac{k}{2}-2},\,|G_{ij}^{(k)}|\le C 
			\delta(x_1)^{j-i-\frac{k}{2}-1}.
		\end{equation}
		
		Indeed, by \eqref{f211,f212} and \eqref{GH1esti},  we have \eqref{gn1} and 
		\eqref{gn2} hold for $j=1$. Assuming that \eqref{gn1} and \eqref{gn2} hold for 
		$j=l-1$ with $l\ge 2$, that is, in $\Omega_{2R}$,
		\begin{equation}\label{fzgn2}
			({\bf f}_{2}^{(l-1)})^{(1)}=\sum_{i=0}^{2l}H_{i(l-1)}(x_1)x_2^{i}, 
			\quad({\bf f}_{2}^{(l-1)})^{(2)}=\sum_{i=0}^{2l-1}S_{i(l-1)}(x_1)x_2^{i}+
			\sum_{i=0}^{2l+1}G_{i(l-1)}(x_1)x_{2}^{i},
		\end{equation}
		where for any $k\ge 0$,
		\begin{equation}\label{fzgn3}
			|H_{i(l-1)}^{(k)}|\le C \delta(x_1)^{l-i-\frac{k}{2}-\frac{5}{2}},\,|
			S_{i(l-1)}^{(k)}|\le C\delta(x_1)^{l-i-\frac{k}{2}-3},\,|G_{i(l-1)}
			^{(k)}|\le C 
			\delta(x_1)^{l-i-\frac{k}{2}-2}.
		\end{equation}
		Next, we will construct ${\bf v}_{2}^{l}$ and the corresponding $\bar{p}_{l}:=\tilde{p}_{l}+\hat{p}_{l}$, such that ${\bf f}_2^l$ satisfies \eqref{gn1} and \eqref{gn2}.
		
		{\bf Step II.1. Construction of ${\bf v}_{2}^{l}$ and $\bar{p}_{l}$.}
		It follows from \eqref{fzgn2} and \eqref{fzgn3} that
		$$	|({\bf f}_{2}^{(l-1)})^{(1)}|\leq C\delta(x_1)^{l-\frac52}\quad \text{in}~\Omega_{2R}$$
		and
		\begin{equation*}
			\Big|\sum_{i=0}^{2l-1}S_{i(l-1)}(x_1)x_2^{i}\Big|\leq C\delta(x_1)^{l-3}, 
			\quad \Big|\sum_{i=0}^{2l+1}G_{i(l-1)}(x_1)x_{2}^{i}\Big|\leq 
			C\delta(x_1)^{l-2}
			\quad \text{in}~\Omega_{2R}.
		\end{equation*}
		The leading term in ${\bf f}_{2}^{(l-1)}$ is $\sum_{i=0}^{2l-1}S_{i(l-1)}(x_1)x_2^{i}$. To cancel this leading term and reduce the upper bound of non-homogeneous terms, we first take
		\begin{equation*}
			\tilde{p}_{l}=\sum_{i=0}^{2l-1}\frac{S_{i(l-1)}(x_1)}{i+1}x_{2}^{i+1}
			\quad\text{in}~\Omega_{2R},
		\end{equation*}
		such that
		\begin{equation}\label{estl2}
		\begin{aligned}
		    {\bf f}_2^{l-1}-\nabla \tilde{p}_l&=
			\begin{pmatrix}
				\sum_{i=1}^{2l}\left(H_{i(l-1)}(x_1)-\frac{1}{i}S'_{(i-1)(l-1)}
				\right)x_2^{i}+H_{0(l-1)}\\\\
				\sum_{i=0}^{2l+1}G_{i(l-1)}(x_1)x_{2}^{i}
			\end{pmatrix}\\
            &:=\begin{pmatrix}
				\sum_{i=0}^{2l}A_{i(l-1)}(x_1)x_2^{i}\\\\
				\sum_{i=0}^{2l+1}B_{i(l-1)}(x_1)x_{2}^{i}
			\end{pmatrix}.
		\end{aligned}	
		\end{equation}
		By \eqref{fzgn3}, we have for any $k\ge0$,
		\begin{equation}\label{fzxs1}
			|A_{i(l-1)}^{(k)}|\le  C \delta(x_1)^{l-i-\frac{k}{2}-\frac{5}{2}},\quad 
			|B_{i(l-1)}^{(k)}|\le C \delta(x_1)^{l-i-\frac{k}{2}-2}\quad\text{in}
			~\Omega_{2R}.
		\end{equation}
		It follows from \eqref{estl2} and \eqref{fzxs1} that
		\begin{equation*}
			|({\bf f}_2^{l-1})^{(1)}-\partial_{x_{1}} \tilde{p}_l|\leq C\delta(x_1)^{l-\frac52},\quad 	|({\bf f}_2^{l-1})^{(2)}-\partial_{x_{2}} \tilde{p}_l|\leq C\delta(x_1)^{l-2}.
		\end{equation*}
		
		To further reduce the upper bound, $\delta(x_1)^{l-\frac52}$, we need to choose a divergence-free auxiliary function ${\bf v}_{2}^{l}(x)$ that satisfies ${\bf v}_{2}^{l}(x)=0$ on $\Gamma^+_{2R}\cup\Gamma^-_{2R}$, specifically,
		\begin{equation}\label{v2lform}
			\begin{split}
				({\bf v}_{2}^{l})^{(1)}=\Big(\sum_{i=0}^{2l}F_{1l}^{i}(x_1)x_{2}^{i}
				+\tilde{F}_{1l}(x_1)\Big)\Big(k(x)^2-\frac{1}{4}\Big), \\
				({\bf v}_{2}^{l})^{(2)}=\Big(\sum_{i=0}^{2l+1}F_{2l}^{i}(x_1)x_{2}
				^{i}+\tilde{F}_{2l}(x)\Big)\Big(k(x)^2-\frac{1}{4}\Big)
			\end{split}\quad\text{in}~\Omega_{2R}.
		\end{equation} 
		Then we take $F_{1l}^i(x_1)$, $i=0,1,\dots,2l$, so that
		\begin{equation}\label{fzgn1}
			\mu\partial_{x_2x_2}\Big(\sum_{i=0}^{2l}F_{1l}^{i}(x_1)x_{2}^{i}
			\Big(k(x)^2-\frac{1}{4}\Big)\Big)+ ({\bf f}_{2}^{l-1})^{(1)}-
			\partial_{x_1}
			\tilde{p}_{l}=0.
		\end{equation}
		By comparing the coefficients of different order of $x_2$, we have
		\begin{equation*}
		\begin{aligned}
		     F_{1l}^{i}(x_1)=&\, -\frac{\delta(x_1)^2A_{i(l-1)}(x_1)}{\mu(i+1)(i+2)}
			+(h_1-h_2)(x_{1})F_{1l}^{i+1}(x_1)\\
            &\, +\frac14(\varepsilon+2h_1(x_{1}))(\varepsilon+2h_2(x_{1}))F_{1l}
			^{i+2}(x_1).
		\end{aligned}
		\end{equation*}
		Here we use the convention that $F_{1l}^{i}=0$ if $i\notin\{0,\ldots,2l\}$ 
		and $F_{2l}^{i}=0$ if $i\notin\{0,\ldots,2l+1\}.$  
		
		To have $\nabla\cdot{\bf 
			v}_{2}^{l}=0,$ as in Proposition \ref{prop2.1}, we next set
		\begin{align}\label{F1l2lform}
			F_{2l}^{i}(x_1)=&\frac{\delta(x_1)^2}{i+2}\Big(\frac{(h_1-
				h_2)F_{1l}^{i}(x_1)}{\delta(x_1)^2}-\frac{F_{1l}^{i-1}(x_1)}
			{\delta(x_1)^2}+\frac{(\varepsilon+2h_1(x_{1}))(\varepsilon+2h_2(x_{1}))}
			{4\delta(x_1)^2}F_{1l}^{i+1}(x_1)\Big)'\nonumber\\
			&+(h_1-h_2)(x_{1})F_{2l}
			^{i+1}(x_1)+\frac14(\varepsilon+2h_1(x_{1}))(\varepsilon+2h_2(x_{1}))F_{2l}^{i+2}
			(x_1),
		\end{align}
		such that
		\begin{equation*}
			\nabla \cdot\begin{pmatrix}
				\Big(\sum_{i=0}^{2l}F_{1l}^{i}(x_1)x_{2}
				^{i}\Big)\Big(k(x)^2-\frac{1}{4}\Big)\\
				\Big(\sum_{i=0}^{2l+1}F_{2l}^{i}(x_1)x_{2}
				^{i}\Big)\Big(k(x)^2-\frac{1}{4}\Big)
			\end{pmatrix}:=R(x_1)\quad\text{in}~\Omega_{2R},
		\end{equation*}
		where
		\begin{align*}
			R(x_1)=&\,-\frac{h_1-h_2}{\delta(x_1)^2}F_{2l}^{0}(x_1)-\frac{(\varepsilon+2h_1(x_{1}))
				(\varepsilon+2h_2(x_{1}))}{4\delta(x_1)^2}F_{2l}^{1}(x_1)\\
			&\quad -
			\Big(\frac{(\varepsilon+2h_1(x_{1}))(\varepsilon+2h_2(x_{1}))}{4\delta(x_1)^2}F_{1l}^{0}
			\Big)'.   
		\end{align*}          
		As in Proposition \ref{prop2.1}, \eqref{yxws}--
		\eqref{dy3}, we construct $\tilde{F}_{1l}(x_1)$ and $\tilde{F}_{2l}(x)$ as follows:
		\begin{equation}\label{tilF1l2ldef}
			\begin{split}
				&\tilde{F}_{1l}(x_1)=\frac{6\int_{0}^{x_1}\delta(y)R(y)\ dy}{\delta(y)},\\
				&\tilde{F}_{2l}(x)=-2\delta(x_1)R(x_1)k(x)-\delta(x_1)\partial_{x_1}
				k(x)\tilde{F}_{1l}(x_1):=P(x_1)x_2+Q(x_1),
			\end{split}
		\end{equation}
		such that
		\begin{equation*}
			\nabla \cdot\begin{pmatrix}
				\tilde{F}_{1l}(x_1)\Big(k(x)^2-\frac{1}{4}\Big) \\
				\tilde{F}_{2l}(x)\Big(k(x)^2-\frac{1}{4}\Big)
			\end{pmatrix}=-R(x_1)\quad\text{in}~\Omega_{2R}.
		\end{equation*}
		Thus, $\nabla\cdot{\bf v}_{2}^{l}=0$. 
		
		Moreover, by \eqref{fzxs1}, 
		\eqref{F1l2lform}, \eqref{tilF1l2ldef}, and a direct calculation, we derive 
		for 
		$k\ge0$, $0\le i \le 
		2l$, and $0\le j\le 2l+1$, in $\Omega_{2R}$,
		\begin{equation}\label{estc1}
			|F_{1l}^{i(k)}|\le C\delta(x_1)^{l-i-\frac{k}{2}-\frac12}, \quad |F_{2l}
			^{j(k)}|\le C \delta(x_1)^{l-j-\frac{k}{2}},
		\end{equation}
		and 
		\begin{equation}\label{estc2}
			|\tilde{F}_{1l}^{(k)}(x_1)|\le C\delta(x_1)^{l-\frac{k+1}{2}},\,\,|
			P^{(k)}(x_1)|\le C\delta(x_1)^{l-\frac{k+2}{2}},\,\,|Q^{(k)}(x_1)|\le 
			C\delta(x_1)^{l-\frac{k}{2}}.
		\end{equation}
		Next, as in Step I.2, to eliminate the new leading term $2\mu\delta(x_1)^{-2}\tilde{F}_{1l}(x_1)$, we further take 
		\begin{equation*}
			\hat{p}_{l}=2\mu\int_{0}^{x_1}\frac{\tilde{F}_{1l}(y)}{\delta(y)^2}\ 
			dy\quad\text{in}~\Omega_{2R},
		\end{equation*}
		such that
		\begin{equation}\label{f1}
			\partial_{x_1}\hat{p}_{2}=\frac{2\mu}{\delta(x_1)^2}\tilde{F}_{1l}(x_1), \quad \partial_{x_2}\hat{p}_{l}
			=0.
		\end{equation}
		Now we obtain an auxiliary pressure function $\bar{p}_{l}:=\tilde{p}_{l}+\hat{p}_{l}$.
		
		{\bf Step II.2. Estimates of ${\bf f}_2^l$.} By \eqref{deffl}, \eqref{fzgn1}, and \eqref{f1}, we obtain
		\begin{equation*}
			{\bf f}_{2}^{l}=\begin{pmatrix}
				\mu\partial_{x_1x_1}({\bf v}_{2}^{l})^{(1)}\\\\
				({\bf f}_{2}^{l-1})^{(2)}-\partial_{x_2}\tilde{p}_{l}+\mu\partial_{x_2x_2}({\bf v}_{2}^{l})^{(2)}+\mu \partial_{x_1x_1}({\bf v}_{2}^{l})^{(1)}
			\end{pmatrix}\quad\text{in}~\Omega_{2R}.
		\end{equation*}
		It follows from \eqref{hatv21def} that $({\bf v}_2^l)^{(1)}$ is a polynomial of order $2l+2$ in $x_2$, while $(\hat{\bf v}_2^1)^{(2)}$ is a polynomial of order $2l+3$ in $x_2$. Then, by \eqref{estl2}, we can rewrite 
		\begin{equation*}
			({\bf f}_{2}^{l})^{(1)}:=\sum_{i=0}^{2l+2}
			H_{il}(x_1)x_2^{i}, \quad
			({\bf f}_{2}^{l})^{(2)}:=\sum_{i=0}^{2l+1}S_{il}
			(x_1)x_2^{i}+\sum_{i=0}^{2l+3}G_{il}(x_1)x_{2}^{i},
		\end{equation*} 
		where 
		$$\sum_{i=0}^{2l+1}S_{il}(x_1)x_2^{i}:=({\bf f}_{2}
		^{l-1})^{(2)}-\partial_{x_2}\tilde{p}_{l}+\mu\partial_{x_2x_2}({\bf v}_{2}
		^{l})^{(2)},\quad  \sum_{i=0}^{2l+3}G_{il}(x_1)x_{2}^{i}:=\mu 
		\partial_{x_1x_1}({\bf v}_{2}^{l})^{(1)}.$$
		Then, by \eqref{fzxs1}, \eqref{estc1}, and 
		\eqref{estc2}, we have for $k\ge 0$, in $\Omega_{2R}$,
		\begin{align*}
			|H_{il}^{k}|\le C \delta(x_1)^{l-i-\frac{k}{2}-\frac{3}{2}},\,\,|S_{il}
			^{(k)}|\le C\delta(x_1)^{l-i-\frac{k}{2}-2},\,\,|G_{il}^{(k)}|\le C 
			\delta(x_1)^{l-i-\frac{k}{2}-1}.
		\end{align*}
		
		Thus, \eqref{gn1} and \eqref{gn2} hold for any $j\ge 1$. By \eqref{v2lform}, 
		\eqref{estc1}, and \eqref{estc2}, we have for $k\ge0$, in $\Omega_{2R}$,
		\begin{equation}\label{v2lest}
			\begin{split}
				&|\partial_{x_1}^{k}\partial_{x_2}^{s}({\bf v}_{2}^{l})^{(1)}|\le C 
				\delta(x_1)^{l-s-\frac{k+1}{2}}\quad\text{for}~s\le2l+2,\quad\partial_{x_2}^{s}
				({\bf v}_{2}^{l})^{(1)}=0\quad\text{for}~s\ge2l+3,                                                                                                                                                                            \\
				&|\partial_{x_1}^{k}\partial_{x_2}^{s}({\bf v}_{2}^{l})^{(2)}|\le C 
				\delta(x_1)^{l-s-\frac{k}{2}}\quad\text{for}~s\le2l+3,\quad\quad\partial_{x_2}^{s}
				({\bf v}_{2}^{l})^{(2)}=0\quad\text{for}~s\ge2l+4,\\
				&|\partial_{x_1}^{k}\partial_{x_2}^{s}\tilde{p}_{l}|\le C\delta(x_1)^{l-s-
					\frac{k+4}{2}}\, \text{for}~s\le2l,\, \partial_{x_2}^{s}\tilde{p}_{l}
				=0\,\,\,\text{for}~s>2l, \,|\partial_{x_1}^{k}\hat{p}_{l}|\le C\delta(x_1)^{l-\frac{k+4}{2}}.
			\end{split}
		\end{equation}
		From \eqref{gn1} and \eqref{gn2}, we obtain
		\begin{equation*}
			|({\bf f}_{2}^{m+1})^{(1)}|\le C \delta(x_1)^{m-\frac{1}{2}},\quad|({\bf f}_{2}
			^{m+1})^{(2)}|\le C\delta(x_1)^{m-1}.
		\end{equation*}
		Consequently,
		\begin{equation*}
			|{\bf f}_{2}^{m+1}|\le C \delta(x_1)^{m-1}\quad\text{in}~\Omega_{2R}.
		\end{equation*}
		Similarly, for $1\le s\le m+1$,
		\begin{equation*}
			|\nabla^{s}{\bf f}_{2}^{m+1}|\le C \delta(x_1)^{m-1-s}\quad\text{in}~\Omega_{2R}.
		\end{equation*}
		
		{\bf Step II.3. Estimates of $\nabla^{m+1}{\bf u}_{1}^{2}$ and $\nabla^{m}p_{1}^{2}$.} 
		Denote
		\begin{equation*}
			{\bf v}^{m+1}(x)=\sum_{l=1}^{m+1}{\bf v}_{2}^{l}(x),\quad \bar{p}^{m+1}(x)=\sum_{l=1}^{m+1}\bar{p}_{l}(x)\quad\text{in}~\Omega_{2R}.
		\end{equation*}
		It is easy to verify that ${\bf v}(x)={\bf u}_1^2(x)$ on $\Gamma_{2R}^{\pm}$, $\nabla\cdot{\bf v}(x)=0$ in $\Omega_{2R}$, and
		\begin{equation*}
			|{\bf f}^{m+1}(x)|=|\mu\Delta{\bf v}(x)-\nabla \bar{p}(x)|\leq C \delta(x_1)^{m-1},\quad |\nabla^s{\bf f}^{m+1}(x)|\leq C \delta(x_1)^{m-s-1},\quad 1\le s\le m.
		\end{equation*}
		Thus, by virtue of Proposition \ref{prop3.3}, we have
		\begin{equation}\label{u12minusest}
			\|\nabla^{m+1}({\bf u}_{1}^{2} -{\bf v}^{m+1})\|_{L^{\infty}(\Omega_{\delta(x_1)/2}(x_1))} +\|\nabla^{m} (p_{1}^{2}-\bar{p}^{m+1})\|_{L^{\infty}(\Omega_{\delta(x_1)/2}(x_1))} \le C.
		\end{equation}
		It follows from \eqref{v2lest} that for $l\ge 2$,
		\begin{equation}\label{gradv2lest}
			\begin{split}
				&|\nabla^{m+1}{\bf v}_{2}^{l}|\le C\delta(x_1)^{-\frac{m+4}{2}}\quad\text{for}~l\le 
				\frac{m-2}{2},\\&|\nabla^{m+1}{\bf v}_{2}^{l}|\le C\delta(x_1)^{l-m-\frac{3}
					{2}}\quad\text{for}~\frac{m-1}{2}\le l\le m+1,\\
				&|\nabla^{m}\bar{p}_{l}|\le C \delta(x_1)^{-\frac{m+4}{2}}\quad\text{for}~ 
				l\le \frac{m-1}{2},\\&|\nabla^{m}\bar{p}_{l}|\le C\delta(x_1)^{l-m-2}
				\quad\text{for}~ \frac{m}{2}\le l\le m+1
			\end{split}\quad \text{in}~\Omega_{2R}.
		\end{equation}
		Thus, by \eqref{v21high} and \eqref{gradv2lest}, we have for $m\ge 1$,
		\begin{equation}\label{zgj3}
			|\nabla^{m+1} {\bf v}^{m+1}|+|\nabla^m \bar{p}^{m+1}|\leq C\delta(x_1)^{-\frac{m+4}{2}}\quad\text{in}~\Omega_{R}.
		\end{equation}
		
		It follows from \eqref{v21high} and \eqref{v2lest} that
		\begin{equation*}
			|\nabla {\bf v}^{m+1}|\leq C\delta(x_1)^{-2}\quad\text{in}~\Omega_{R}.
		\end{equation*}
		This, in combination with \eqref{u12minusest} and \eqref{zgj3}, yields 
		\begin{equation*}
			|\nabla {\bf u}_1^2|\leq C\delta(x_1)^{-2}	\quad\text{in}~\Omega_{R}
		\end{equation*}
		and for $m\ge1$,
		\begin{equation*}
			|\nabla^{m+1} {\bf u}_{1}^{2}| +|\nabla^{m} {p}_{1}^{2}|\le C \delta(x_1)^{-\frac{m+4}{2}}
			\quad\text{in}~\Omega_{R}.
		\end{equation*}   
		For the estimates of $p_{1}^{2}$, by \eqref{v21high} and \eqref{v2lest}, we  have
		\begin{equation*}
			|\bar{p}_1(x)+\bar{p}_2(x)|\leq C\delta(x_1)^{-2}.
		\end{equation*}
		Thus, for a fixed point $z=(z_1,0)\in\Omega_{R}$ with $|
		z_1|=R$, by the mean value theorem and \eqref{u12minusest}, we have
		\begin{equation*}
			|p_{1}^{2}(x)-p_{1}^{2}(z_1,0)|\le C\|\nabla(p_1^2-\bar{p}_1-\bar{p}_2)\|_{L^\infty}+C|\bar{p}_1(x)+\bar{p}_2(x)|+C\leq 
			C\delta(x_1)^{-2}.
		\end{equation*}
		This finishes the proof of Proposition \ref{prop2.2}.
	\end{proof}
	
	\section{Estimates for $({\bf u}_{1}^{3},p_{1}^{3})$}\label{sec_u3}
	
	In this section, we focus on estimating the higher-order derivatives of \({\bf u}_1^3\) and \(p_1^3\). Since the boundary data \(\boldsymbol{\psi}_3 = (x_2, -x_1)^T\) is relatively small, the leading terms of \(\mu\Delta {\bf v}_3^1 - \nabla \bar{p}_1\) are included in \((\mu\Delta {\bf v}_3^1 - \nabla \bar{p}_1)^{(1)}\) in the first step, while the remaining part is a relatively mild term, as shown in \eqref{f3132polyna} and \eqref{u3fz1} below. However, we must first provide the construction process of \(({\bf v}_3^2, \bar{p}_2)\) to initiate the induction process for estimating the higher-order derivatives of \({\bf u}_1^3\) and \(p_1^3\). This distinguishes this section from the previous ones.
	
	\begin{prop}\label{prop2.3}
		Under the same assumption as in Theorem \ref{main thm1}, let ${\bf u}_{1}^{3}$ and $p_{1}^{3}$ be the solution to \eqref{u,peq1} with $\alpha=3$. Then for sufficiently small $0<\varepsilon<1/2$, we have 
		$$\mbox{ (i) }\quad\quad\quad~|\nabla{\bf u}_{1}^{3}(x)|\le C\delta(x_1)^{-1},\quad |p_1^3(x)-p_1^3(z_1,0)|\leq C\delta(x_1)^{-3/2},\quad x\in\Omega_{R},\quad\quad\quad$$ 
		for some point $z=(z_1,0)\in \Omega_{R}$ with $|z_1|=R/2$; and 
		
		(ii) for any $m\ge 1$, 
		\begin{equation*}
			|\nabla^{m+1}{\bf u}_{1}^{3}(x)|+|\nabla^{m}p_{1}^{3}(x)|\le C \delta(x_1)^{-\frac{m+3}{2}}\quad \text{for}~x\in\Omega_{R}.
		\end{equation*}
	\end{prop}
	
	\begin{proof}[Proof of Proposition \ref{prop2.3}] The proof also follows that of Propositions \ref{prop2.1}. As before, we may assume that $x_1\ge 0$. Here we only list the key steps and modifications.

		{\bf Step I. Gradient Estimates.}
		Set
		\begin{equation}\label{barv31}
			{\bf v}_{3}^{1}(x)={\boldsymbol{\psi}}_{3}\Big(k(x)+\frac{1}{2}\Big) +
			\begin{pmatrix}
				F(x)\\
				G(x)
			\end{pmatrix}\Big(k(x)^{2}-\frac{1}{4}\Big)\quad\text{in}~\Omega_{2R},
		\end{equation}
		which satisfies ${\bf v}_{3}^{1}(x)={\bf u}_1^3(x)$ on $\Gamma_{2R}^+\cup\Gamma_{2R}^-$. To ensure that $\nabla\cdot{\bf v}_{3}^{1}(x)=0$ in $\Omega_{2R}$, we assume 
		\begin{equation}\label{v31fz1}
			\frac14(\partial_{x_{1}}F+\partial_{x_{2}}G)(x)=x_2\partial_{x_{1}}k(x)-\frac{x_1}{\delta(x_1)}.
		\end{equation}
		By \eqref{barv31}, \eqref{v31fz1}, and  $\nabla\cdot{\bf v}_{3}^{1}(x)=0$ in $\Omega_{2R}$, we have
		\begin{equation}\label{v31fz2}
			G(x)=2\big(x_1-\delta(x_1)x_2\partial_{x_{1}}k(x)\big)k(x)-F(x)\delta(x_1)\partial_{x_{1}}k(x).
		\end{equation}
		Substituting \eqref{v31fz2} into \eqref{v31fz1}, we see that $F(x)$ satisfies
		\begin{align*}\label{v31fz3}
			&\partial_{x_{1}}F(x)+\frac{(h'_1+h'_2)(x_1)}{\delta(x_1)}F(x)-
			\delta(x_1)
			\partial_{x_1}k(x)\partial_{x_{2}}F(x)\nonumber\\
			=&\,-\frac{6x_1}
			{\delta(x_1)}+10x_2\partial_{x_{1}}k(x)-(h_1-h_2)(x_{1})\partial_{x_{1}}k(x)+(h'_1-h'_2)(x_1)k(x)\\
			&\,+\frac{(h_1-h_2)(h'_1-h'_2)(x_{1})}{2\delta(x_1)}.
		\end{align*}
		Here we take a special solution
		\begin{equation}\label{deffv31}
			F(x)=1-\frac{(h_1+h_2)(x_1)+3x_{1}^{2}}{\delta(x_1)}-\frac{3(h_1-h_2)(x_{1})
				x_2}{2\delta(x_1)} -5k(x)x_2,
		\end{equation}
		and by \eqref{v31fz2}, 
		\begin{equation}\label{defgv31}
			G(x)=3\delta(x_1)k(x)\partial_{x_1}k(x)x_2+\frac{3(h_1-h_2)(x_{1})}{2}
			\partial_{x_1}k(x)x_2+2x_{1}k(x)-(\varepsilon-3x_1^2)\partial_{x_1}k(x),
		\end{equation}
		such that $\nabla\cdot{\bf v}_{3}^{1}=0$ in $\Omega_{2R}$. 
		
		By \eqref{barv31}, \eqref{deffv31}, and \eqref{defgv31}, the leading terms in $\Delta({\bf v}_3^1)^{(1)}$ and $\Delta({\bf v}_3^1)^{(2)}$ are 
		$$2\delta(x_1)^{-2}\Big(1-\frac{(h_1+h_2)(x_1)+3x_{1}^{2}}{\delta(x_1)}\Big),$$ being of order $\delta(x_1)^{-2}$ and $$\partial_{x_2x_2}\Big(\big(2x_{1}k(x)-(\varepsilon-3x_1^2)\partial_{x_1}k(x)\big)\Big(k(x)^{2}-\frac{1}{4}\Big)\Big),$$ being of order $\delta(x_1)^{-3/2}$, respectively. To cancel these leading terms, we choose
		\begin{equation*}
			{\bar p}_{1}=\frac{2\mu x_1}{\delta(x_1)^2}-2\mu\int_{x_1}^{R}
			\frac{2y(h'_1+h'_2)(y)-(h_1+h_2)(y)-3y^2}{\delta(y)^3}\ dy+r(x)
			\quad\text{in}~\Omega_{2R},
		\end{equation*}
		where
		$$r(x)=\mu\partial_{x_2}\Big(\big(2x_{1}k(x)-(\varepsilon-3x_1^2)\partial_{x_1}k(x)\big)\Big(k(x)^{2}-\frac{1}{4}\Big)\Big),$$
		such that 
		\begin{equation}\label{zyfz1/}
			\begin{split}
				&\Big|\partial_{x_1}{\bar p}_{1}-2\mu\delta(x_1)^{-2}\Big(1-\frac{(h_1+h_2)(x_1)+3x_{1}^{2}}{\delta(x_1)}\Big)\Big|=|\partial_{x_{1}} r(x)|\ \leq C\delta(x_1)^{-1},\\
				&\partial_{x_2}{\bar p}_{1}=\mu\partial_{x_2x_2}\Big(\big(2x_{1}k(x)-(\varepsilon-3x_1^2)\partial_{x_1}k(x)\big)\Big(k(x)^{2}-\frac{1}{4}\Big)\Big).
			\end{split}
		\end{equation}
		
		Denote
		$${\bf f}_{3}^{1}:=\mu\Delta{\bf v}_{3}^{1}-\nabla\bar{p}_{1}\quad\text{in}~\Omega_{2R}.$$ 
		By \eqref{barv31}, \eqref{deffv31}, and \eqref{defgv31}, we can regard $({\bf v}_3^1)^{(1)}$ as a cubic polynomial in $x_2$ and $({\bf v}_3^1)^{(2)}$ as a polynomial of order $4$ in $x_2$. In view of \eqref{barv31} and \eqref{zyfz1/}, ${\bf f}_{3}^{1}(x)$ has the following form:
		\begin{equation}\label{f3132polyna}
			\quad({\bf f}_{3}^{1})^{(1)}=\sum_{i=0}^{2}S_{i1}(x_1)x_{2}^{i}+
			\sum_{j=1}^{4}G_{j1}(x_1)x_{2}^{j}, \quad({\bf f}_{3}^{1})^{(2)}=\sum_{k=0}^{3}\tilde{S}_{k1}(x_1)x_{2}
			^{k}+\sum_{i=1}^{5}\tilde{G}_{l1}(x_1)x_{2}^{l},
		\end{equation}
		where, for $0\le i\le 2,$ $1\le j\le 4$, $0\le k\le 3$, $1\le l\le5$, and $s\ge 
		0$,
		\begin{equation*}
			\begin{split}
				&|(S_{i1})^{(s)}|\le C\delta(x_1)^{-i-\frac{s}{2}-1},\quad\,\, |
				(G_{j1})^{(s)}|\le C\delta(x_1)^{-j-\frac{s}{2}},\\
				&|(\tilde{S}_{k1})^{(s)}|\le C\delta(x_1)^{-k-\frac{s}{2}-\frac12},
				\quad |(\tilde{G}_{l1})^{(s)}|\le C\delta(x_1)^{-l-\frac{s}{2}+
					\frac12}.
			\end{split}
		\end{equation*}
		Then
		\begin{equation}\label{u3fz1}
			\begin{split}
				&\Big|\sum_{i=0}^{2}S_{i1}(x_1)x_{2}^{i}\Big|\leq C\delta(x_1)^{-1},
				\quad \Big|	\sum_{j=1}^{4}G_{j1}(x_1)x_{2}^{j}\Big|\leq C,\\
				&\Big|\sum_{k=0}^{3}\tilde{S}_{k1}(x_1)x_{2}^{k}\Big|\leq 
				C\delta(x_1)^{-\frac12},\quad\Big|\sum_{l=1}^{5}\tilde{G}_{l1}(x_1)x_{2}
				^{l}\Big|\leq C\delta(x_1)^{\frac12},
			\end{split}
		\end{equation}
		which implies
		$$|({\bf f}_{3}^{1})^{(1)}|\le C \delta(x_1)^{-1},\quad |({\bf f}_{3}^{1})^{(2)}|\le C \delta(x_1)^{-1/2}\quad \text{in}~\Omega_{2R}.$$
		Therefore,
		\begin{equation}\label{f31esti}
			|{\bf f}_{3}^{1}|\le C\delta(x_1)^{-1}\quad \text{in}~\Omega_{2R}.
		\end{equation}
		
		By Proposition \ref{prop3.3} and \eqref{f31esti}, we have
		$$\|\nabla {\bf u}_{1}^{3}-\nabla{\bf v}_{3}^{1}\|_{L^{\infty}
			(\Omega_{\delta(x_1)/2}(x_1))}\le C. $$
		Moreover, a direct calculation gives that in $\Omega_{2R}$,
		\begin{equation}\label{v31gradesti}
			\begin{split}
				&|\nabla{\bf v}_{3}^{1}|\le C \delta(x_1)^{-1},\quad|\nabla^{l}{\bf v}
				_{3}^{1}|\le C\delta(x_1)^{-\frac{l+2}{2}}\quad \text{for}~l\ge2, \\ 
				&|\bar{p}_1|\le C\delta(x_1)^{-3/2},\,\quad\quad|\nabla^{l}\bar{p}_{1}|\le 
				C\delta(x_1)^{-\frac{3+l}{2}}\quad\, \text{for}~l\ge1.
			\end{split}
		\end{equation}
		Hence, it holds that
		$$|\nabla{\bf u}_{1}^{3}|\le |\nabla{\bf v}_{3}^{1}|+C \le C \delta(x_1)^{-1}.$$ 
		
		{\bf Step II. Second-order derivatives estimates.}
		
		Denote $${\bf f}_{3}^{l}:={\bf f}_{3}^{l-1}(x)+\mu \Delta{\bf v}_{3}^{l}(x)-
		\nabla\bar{p}_{l}(x),\quad 2\le l\le m+1.$$ 
		Similar to Step II.1 in the proof of Proposition \ref{prop2.1}, we first take
		\begin{equation}\label{v32form}
			{\bf v}_{3}^{2}:=
			\begin{pmatrix}
				F_{12}^{2}(x_1)x_2^{2}+F_{12}^{1}(x_1)x_{2}+F_{12}^{0}(x_1)+
				\tilde{F}_{12}(x_1)\\\\ 
				F_{22}^{3}(x_1)x_{2}^{3}+F_{22}^{2}(x_1)x_{2}^{2}+F_{22}^{1}
				(x_1)x_{2}+F_{22}^{0}(x_1)+\tilde{F}_{22}(x) 
			\end{pmatrix}\Big(k(x)^2-\frac{1}{4}\Big).
		\end{equation} 
		It follows from \eqref{u3fz1} that the leading term of $({\bf f}
		_3^1)^{(1)}$ is $\sum_{i=0}^{2}S_{i1}(x_1)x_{2}^{i}$, being of order $
		\delta(x_1)^{-1}$. 
		
		To cancel this leading term, instead of \eqref{cancelf1l1}, we choose $F_{12}^i$, 
		$i=0,1,2$, 
		such that
		\begin{equation}\label{fz3.31}
	\mu\partial_{x_2x_2}\Big(\sum_{i=0}^{2}F_{12}^ix_2^i\Big(k(x)^2-\frac14\Big)
			\Big)=-\sum_{i=0}^{2}S_{i1}(x_1)x_{2}^{i}.
		\end{equation}
		By comparing the coefficients of each term of $x_2$, we have
		\begin{equation*}
			\begin{split}
				&F_{12}^{2}(x_1)=-\frac{\delta(x_1)^{2}}{12\mu}S_{21}(x_1),\quad F_{12}^{1}
				(x_1)=-\frac{\delta(x_1)^{2}}{6\mu}S_{11}(x_1)+(h_1-h_2)(x_{1})F_{12}^{2}(x_1),\\  
				&F_{12}^{0}(x_1)=-\frac{\delta(x_1)^2}{2\mu}S_{01}(x_1)+ (h_1-h_2)(x_{1})F_{12}
				^{2}(x_1)+\frac14(\varepsilon+2h_1(x_{1}))(\varepsilon+2h_2(x_{1}))F_{12}^{2}(x_1).
			\end{split}
		\end{equation*}
		
		To ensure $\nabla\cdot{\bf v}_3^2(x)=0$, we set
		\begin{equation*}
			\begin{split}
				F_{22}^{3}(x_1)=&\,-\frac{\delta(x_1)^{2}}{5}\Big(\frac{F_{12}^{1}(x_1)}
				{\delta(x_1)^{2}}\Big)',\\
				F_{22}^{2}(x_1)=&\,-\frac{\delta(x_1)^2}{4}
				\Big(\frac{F_{12}^{1}(x_1)-(h_1-h_2)(x_{1})F_{12}^{2}(x_1)}{\delta(x_1)^2}
				\Big)'+ (h_1-h_2)(x_{1})F_{22}^{3}(x_1),\\
				F_{22}^{1}(x_1)=&\,-\frac{\delta(x_1)^2}{3}\Big(\frac{F_{12}^{0}(x_1)}
				{\delta(x_1)^2}-\frac{(h_1-h_2)(x_{1})F_{12}^{1}(x_1)}
				{\delta(x_1)^2}- \frac{(\varepsilon+2h_1)(\varepsilon+2h_2)}
				{4\delta(x_1)^2}F_{12}^{2}(x_1)\Big)'\\
				&\,+\frac14(\varepsilon+2h_1(x_{1}))
				(\varepsilon+2h_2(x_{1}))F_{22}^{3}(x_1)+(h_1-h_2)(x_{1})F_{22}^{2}(x_1),\\
				F_{22}^{0}(x_1)=&\,\frac{\delta(x_1)^2}{2}\Big(\frac{(h_1-h_2)(x_{1})F_{12}
					^{0}(x_1)}{\delta(x_1)^2}+\frac{(\varepsilon+2h_1)
					(\varepsilon+2h_2)}{4\delta(x_1)^2}F_{12}^{1}(x_1)\Big)'\\
				&+\frac14(\varepsilon+2h_1(x_{1}))(\varepsilon+2h_2(x_{1}))F_{22}^{2}(x_1)+(h_1-
				h_2)F_{22}^{1}(x_1),
			\end{split}
		\end{equation*}
		such that
		\begin{equation*}
			\nabla \cdot\begin{pmatrix}
				\Big(\sum_{i=0}^{2}F_{11}^{i}(x_1)x_{2}
				^{i}\Big)\Big(k(x)^2-\frac{1}{4}\Big)\\\\
				\Big(\sum_{i=0}^{3}F_{21}^{i}(x_1)x_{2}
				^{i}\Big)\Big(k(x)^2-\frac{1}{4}\Big)
			\end{pmatrix}:=R(x_1)\quad\text{in}~\Omega_{2R},
		\end{equation*}
		where
		\begin{align*}
			R(x_1)=&\, -\Big(\frac{(\varepsilon+2h_1(x_{1}))(\varepsilon+2h_2(x_{1}))}{4\delta(x_1)^2}
			F_{12}^{0}(x_1)\Big)'\\
			&-\frac{(\varepsilon+2h_1(x_{1}))(\varepsilon+2h_2(x_{1}))}
			{4\delta(x_1)^2}F_{22}^{1}(x_1)-\frac{h_1-h_2}{\delta(x_1)^2}F_{22}^{0}(x_1).
		\end{align*}
		
		Furthermore, to have $\nabla\cdot{\bf v}_{3}^{2}=0,$ we choose $\tilde{F}_{12}(x_1)$ and $\tilde{F}_{22}(x)$ such that
		\begin{equation*}
			\nabla \cdot\begin{pmatrix}
				\tilde{F}_{12}(x_1)\Big(k(x)^2-\frac{1}{4}\Big) \\
				\tilde{F}_{22}(x)\Big(k(x)^2-\frac{1}{4}\Big)
			\end{pmatrix}=-R(x_1)\quad\text{in}~\Omega_{2R}.
		\end{equation*}
		As before, we set
		\begin{equation*}
			\begin{split}
				&\tilde{F}_{12}(x_1)=\frac{6\int_{0}^{x_1}\delta(y)R(y)\ dy}{\delta(x_1)},\\
				&\tilde{F}_{22}(x)=-2\delta(x_1)R(x_1)k(x)-\delta(x_1)\partial_{x_1}k(x)
				\tilde{F}_{12}(x_1):=P(x_1)x_2+Q(x_1).
			\end{split}	
		\end{equation*}
		By \eqref{f3132polyna} and \eqref{v32form}, for any $k\ge0$, 
		$i=0,1,2$, and $j=0,1,2,3$, in $\Omega_{2R}$,
		\begin{equation}\label{partF12}
			\begin{split}
				&|(F_{12}^{i})^{(k)}|\le C \delta(x_1)^{1-i-\frac{k}{2}},\,\, |(F_{22}
				^{j})^{(k)}|\le C \delta(x_1)^{\frac32-j-\frac{k}{2}}, \,\, |\tilde{F}_{12}
				^{(k)}|\le C\delta(x_1)^{1-\frac{k}{2}},\\
				&|P(x_1)^{(k)}|\le C\delta(x_1)^{\frac{1}{2}-\frac{k}{2}},\,\,\,\,\, |Q(x_1)^{(k)}|
				\le C\delta(x_1)^{\frac32-\frac{k}{2}}.
			\end{split}
		\end{equation}
		It follows from \eqref{v32form}, \eqref{fz3.31}, and \eqref{partF12} that the new leading term in ${\bf f}_3^1+\mu\Delta ({\bf v}_3^2)^{(1)}$ is $2\mu\delta(x_1)^{-2}\tilde{F}_{12}(x_1)$, being of order $\delta(x_1)^{-1}$. To cancel this leading term, we choose 
		\begin{equation*}
			\tilde{p}_{2}=2\mu\int_{0}^{x_1}\frac{\tilde{F}_{12}(y)}{\delta(y)^{2}}\ dy
			\quad\text{in}~\Omega_{2R},
		\end{equation*}
		such that
		\begin{equation}\label{FStilp2}
			\partial_{x_1}\tilde{p}_{2}=\frac{2\mu}{\delta(x_1)^2}\tilde{F}_{12}(x_1)=\mu\partial_{x_2x_2}\Big(\tilde{F}_{12}(x_1)\Big(k(x)^2-\frac{1}{4}\Big)\Big),\quad\partial_{x_2}\tilde{p}_{2}=0.
		\end{equation}
		
		As in Step II.2 of the proof of Proposition \ref{prop2.1}, by \eqref{v32form}, \eqref{fz3.31}, \eqref{partF12}, and \eqref{FStilp2}, we have
		\begin{align}\label{f31tilde}
			\begin{split}
				&({\bf f}_{3}^{1})^{(1)}+\mu \Delta({\bf v}_{3}^{2})^{(1)}-
				\partial_{x_1}\tilde{p}_{2}=\mu \partial_{x_1x_1}({\bf v}_{3}
				^{2})^{(1)}=\sum_{i=0}^{4}\hat{G}_{i1}(x_1)x_2^{i},\\
				&({\bf f}_{3}^{1})^{(2)}+\mu \Delta({\bf v}_{3}^{2})^{(2)}-
				\partial_{x_2}\tilde{p}_{2}=\sum_{i=0}^{3}\bar{S}_{i1}(x_1)x_2^{i}
				+\sum_{i=0}^{5}\bar{G}_{i1}(x_1)x_2^{i},
			\end{split}
		\end{align}
		where \begin{equation*}
			\sum_{i=0}^{3}\bar{S}_{i1}(x_1)x_2^{i}:=({\bf f}_{2}
			^{1})^{(2)}-\partial_{x_2}\tilde{p}_{2}+\mu\partial_{x_2x_2}({\bf v}_{2}
			^{2})^{(2)},\quad  \sum_{i=0}^{5}\bar{G}_{i1}(x_1)x_2^{i}:=\mu 
			\partial_{x_1x_1}({\bf v}_{2}^{2})^{(1)}.
		\end{equation*}
		For $s\ge 0,$ $i=0,1,2,3,4$, $j=0,1,2,3$, and $k=0,1,2,3,4,5$,
		\begin{equation*}
			|\hat{G}_{i1}^{(s)}|\le C\delta(x_1)^{-i-\frac{s}{2}},\quad |
			\bar{S}_{j1}^{(s)}|\le C \delta(x_1)^{-\frac12-j-\frac{s}{2}},\quad |
			\bar{G}
			_{k1}^{(s)}|\le C\delta(x_1)^{\frac12-k-\frac{s}{2}}.
		\end{equation*}
		
		Thus, the leading term in $({\bf f}_{3}^{1})^{(i)}+\mu \Delta({\bf v}_{3}^{2})^{(i)}-
		\partial_{x_i}\tilde{p}_{2}$, $i=1,2$, is $\sum_{i=0}^{3}\bar{S}_{i1}(x_1)x_2^{i}$, which is of order $\delta(x_1)^{-\frac12}$. Then we take 
		\begin{equation}\label{hatp2form}
			\hat{p}_{2}=\sum_{i=0}^{3}\frac{1}{i+1}\bar{S}_{i1}(x_1)x_{2}^{i+1}
			\quad\text{in}~\Omega_{2R}.
		\end{equation}
		Set $\bar{p}_{2}=\tilde{p}_{2}+\hat{p}_{2}$. By using \eqref{f31tilde} and 
		\eqref{hatp2form}, we can write 
		\begin{equation}\label{f32form}
			({\bf f}_{3}^{2})^{(1)}=\sum_{i=1}^{4}S_{i2}(x_1)x_{2}^{i},\quad({\bf f}
			_{3}^{2})^{(2)}=\sum_{i=0}^{5}G_{i2}(x_1)x_{2}^{i}\quad\text{in}
			~\Omega_{2R},
		\end{equation}
		and for $k\ge0$, $0\le i\le 4$, and $0\le j\le5$,
		\begin{equation}\label{hsg3esti}	
			|S_{i2}^{(k)}|\le C \delta(x_1)^{-i-\frac{k}{2}},\quad
			|G_{j2}^{(k)}|\le C \delta(x_1)^{\frac12-j-\frac{k}{2}}.
		\end{equation}
		Therefore, $|{\bf f}_{3}^{2}|\le C$ in $\Omega_{2R}$.
		
		{\bf Step III. higher-order derivatives estimates.}
		Similar to Step II of the proof of Proposition \ref{prop2.1}, we can 
		inductively prove that for $j\ge 2$, ${\bf f}_{3}^{j}$ can be expressed as a 
		polynomial in $x_2$:
		\begin{equation}\label{gn1?}
			({\bf f}_{3}^{l})^{(1)}=\sum_{i=0}^{2l}S_{il}(x_1)x_2^{i},\quad({\bf f}
			_{3}^{l})^{(2)}=\sum_{i=0}^{2l+1}G_{il}(x_1)x_2^{i}	\quad\text{in}
			~\Omega_{2R},
		\end{equation}
		and for any $k\ge0$,
		\begin{equation}\label{gn2?}
			|S_{il}^{(k)}|\le C\delta(x_1)^{l-i-\frac{k}{2}-2},\quad |G_{il}^{(k)}|
			\le C \delta(x_1)^{l-i-\frac{k}{2}-\frac32}.
		\end{equation}
		
		Indeed, by \eqref{f32form} and \eqref{hsg3esti},  \eqref{gn1?} and 
		\eqref{gn2?} hold for $j=2$. Assuming that \eqref{gn1?} and \eqref{gn2?} hold for 
		$j=l-1$ with $l\ge 2$, that is, 
		\begin{equation*}
			({\bf f}_{3}^{l-1})^{(1)}=\sum_{i=0}^{2l-2}S_{i(l-1)}(x_1)x_2^{i},\quad({\bf 
				f}_{3}^{l-1})^{(2)}=\sum_{i=0}^{2l-1}G_{i(l-1)}(x_1)x_2^{i},
		\end{equation*}
		and for any $k\ge0$,
		\begin{equation*}
			|S_{i(l-1)}^{(k)}|\le C\delta(x_1)^{l-i-\frac{k}{2}-3},\quad |G_{i(l-1)}
			^{(k)}|\le C \delta(x_1)^{l-i-\frac{k}{2}-\frac52}.
		\end{equation*}
		
		Note that the leading term in ${\bf f}_{3}^{l-1}$ is $({\bf f}_{3}^{l-1})^{(1)}$, which is of order $\delta(x_1)^{l-3}$. To cancel this leading term, we take 
		\begin{equation}\label{v3lform}
			\begin{split}
				({\bf v}_{3}^{l})^{(1)}=&\,\Big(\sum_{i=0}^{2l-2}F_{1l}^{i}(x_1)x_{2}
				^{i}+ \tilde{F}_{1l}(x_1)\Big)\Big(k(x)^{2}-\frac{1}{4}\Big),\\
				({\bf v}_{3}^{l})^{(2)}=&\,\Big(\sum_{i=0}^{2l-1}F_{2l}^{i}(x_1)x_{2}
				^{i}+\tilde{F}_{2l}(x)\Big)\Big(k(x)^{2}-\frac{1}{4}\Big).
			\end{split}	
		\end{equation}
		The purpose of constructing $F_{1l}^i$, $i=0,1,\dots,2l-2$, is to eliminate $({\bf f}_{3}^{l-1})^{(1)}$, while the objective of constructing $F_{2l}^i$, $i=0,1,\dots,2l-1$, $\tilde{F}_{1l}$, and $\tilde{F}_{2l}$ is to achieve $\nabla\cdot {\bf v}_3^l=0$. 
		
		By the same way as that in step 2, we have
		\begin{equation*}
			\begin{split}
				F_{1l}^{i}(x_1)=&\,-\frac{\delta(x_1)^2S_{i(l-1)}(x_1)}{\mu(i+1)(i+2)}+
				(h_1-h_2)(x_{1})F_{1l}^{i+1}(x_1)+\frac14(\varepsilon+2h_1)
				(\varepsilon+2h_2)F_{1l}
				^{i+2}(x_1),\\
				F_{2l}^{i}(x_1)=&\,-\frac{\delta(x_1)^2}{i+2}\Big(\delta(x_1)^{-2}(F_{1l}
				^{i-1}(x_1)-(h_1-h_2)(x_{1})F_{1l}^{i}(x_1)\\
				&\qquad \qquad \qquad -\frac14(\varepsilon+2h_1(x_{1}))
				(\varepsilon+2h_2(x_{1})) 
				F_{1l}^{i+1}(x_1))\Big)'\nonumber\\
				&\,+(h_1-h_2)(x_{1})F_{2l}^{i+1}(x_1)+
				\frac14(\varepsilon+2h_1(x_{1}))(\varepsilon+2h_2(x_{1}))F_{2l}^{i+2}(x_1),
			\end{split}
		\end{equation*}
		and
		\begin{equation*}
			\begin{split}
				&\tilde{F}_{1l}(x_1)=\frac{6\int_{0}^{x_1}\delta(y)R(y)\ dy}
				{\delta(x_1)},\\
				&\tilde{F}_{2l}(x)=-2\delta(x_1)R(x_1)k(x)-\delta(x_1)\partial_{x_1}k(x)
				\tilde{F}_{1l}(x_1):=P(x_1)x_2+Q(x_1),
			\end{split}	
		\end{equation*}
		where
		\begin{align*}
			R(y)=&\, -\Big(\frac{(\varepsilon+2h_1(y))(\varepsilon+2h_2(y))}{4\delta(y)^2}F_{1l}
			^{0}(y)\Big)'-\frac{(\varepsilon+2h_1(y))(\varepsilon+2h_2(y))}{4\delta(y)^2}F_{2l}
			^{1}(y)\\
			&\quad -\frac{(h_1-h_2)(y)}{\delta(y)^2}F_{2l}^{0}(y).
		\end{align*}
		Here we use the convention that $F_{1l}^{i}(x_1)=0$ if $i\notin\{0,\ldots,2l-2
		\}$ and $F_{2l}^{i}(x_1)=0$ if $i\notin\{0,\ldots,2l-1\}$.
		Moreover, we take 
		\begin{equation}\label{tilp3l}
			\tilde{p}_{l}=2\mu \int_{0}^{x_1}\frac{\tilde{F}_{1l}(y)}{\delta(y)^2}\ dy,
		\end{equation}
		such that 
		\begin{equation*}
			\partial_{x_1}
			\tilde{p}_{l}=\mu\partial_{x_2x_2}\Big(\tilde{F}_{1l}(x_1)\Big(k(x)^2-\frac14 \Big)\Big),\quad\partial_{x_2}\tilde{p}_{l}=0\quad \text{in}~\Omega_{2R}.
		\end{equation*}
		
		A direct calculation gives that in $\Omega_{2R}$,
		\begin{align*}
			\begin{split}
				&({\bf f}_{3}^{l-1})^{(1)}+\mu\Delta({\bf v}_{3}^{l})^{(1)}-\partial_{x_1}
				\tilde{p}_{l}=\sum_{i=0}^{2l-2}\hat{G}_{i(l-1)}(x_1)x_{2}^{i}, \\
				&({\bf f}_{3}^{l-1})^{(2)}+\mu\Delta({\bf v}_{3}^{l})^{(2)}-\partial_{x_2}
				\tilde{p}_{l}=\sum_{i=0}^{2l-1}\bar{S}_{i(l-1)}(x_1)x_2^{i}+
				\sum_{i=0}^{2l+1}\bar{G}_{i(l-1)}(x_1)x_{2}^{i}.
			\end{split}	
		\end{align*}
		We take 
		\begin{equation}\label{tilplSil}
			\quad\hat{p}_{l}=\sum_{i=0}^{2l-1}\frac{1}{i+1}\bar{S}_{i(l-1)}
			(x_1)x_{2}^{i}\quad\text{in}~\Omega_{2R}.
		\end{equation}
		Then, \eqref{gn1?} and \eqref{gn2?} hold for any $l\ge 2$. 
		
		By virtue of \eqref{gn2?}, \eqref{v3lform}, \eqref{tilp3l}, and 
		\eqref{tilplSil}, we have for any $k\ge0$ and $l\ge2$,
		\begin{equation}
		    \label{v3lesti}
			\begin{aligned}
			    &|\partial_{x_1}^{k}\partial_{x_2}^{s}({\bf v}_{3}^{l})^{(1)}|\le C 
				\delta(x_1)^{l-s-\frac{k+2}{2}}\quad\text{for}~s\le2l,\quad\partial_{x_2}^{s}
				({\bf v}_{3}^{l})^{(1)}=0\quad \text{for}~s\ge2l+1,\\
				&|\partial_{x_1}^{k}\partial_{x_2}^{s}({\bf v}_{3}^{l})^{(2)}|\le C 
				\delta(x_1)^{l-s-\frac{k+1}{2}}\quad\text{for}~s\le2l+1,
				\,\,\partial_{x_2}^{s}({\bf v}_{3}^{l})^{(2)}=0\quad\text{for}~s\ge2l+2,\\
				&|\partial_{x_1}^{k}\partial_{x_2}^{s}\hat{p}_{l}|\le C\delta(x_1)^{l-s-
					\frac{k+3}{2}}\, \text{for}~s\le2l,\, \partial_{x_2}^{s}\hat{p}_{l}
				=0\,\,\, \text{for}~s\ge2l+1,\,|\partial_{x_1}^{k}\tilde{p}_{l}|\le C\delta(x_1)^{l-
					\frac{k+5}{2}}.
			\end{aligned}
		\end{equation}
		Let $\bar{p}_{l}=\tilde{p}_{l}+\hat{p}_{l}.$ Then, for $m\ge 1$ 
		and $l\ge 2$, in $\Omega_{2R}$,
		\begin{align}\label{gradv3lest}
			\begin{split}
				&|\nabla^{m+1}{\bf v}_{3}^{l}|\le C\delta(x_1)^{-\frac{m+3}{2}}\quad \text{for}
				~l\le \frac{m}{2},\\& |\nabla^{m+1}{\bf v}_{3}^{l}|\le C\delta(x_1)^{l-m-2}
				\quad \text{for}~\frac{m+1}{2}\le l\le m+1,\\
				&|\nabla^{m}\bar{p}_{l}|\le C \delta(x_1)^{-\frac{m+3}{2}}\quad\,\, \quad\text{for}~ 
				l\le \frac{m-1}{2},\quad |\nabla\bar{p}_{l}|\le C \delta(x_1)^{l-3},\\& |\nabla^{m}\bar{p}_{l}|\le C\delta(x_1)^{l-m-\frac{3}{2}}\quad\quad \text{for}~\frac{m+1}{2}\le l\le m+1.
			\end{split}
		\end{align}
		By \eqref{gn1?} and \eqref{gn2?}, we obtain that for any $m\ge1$,
		\begin{equation*}
			|{\bf f}_{3}^{m+1}|\le C\delta(x_1)^{m-1},\quad
			|\nabla^{s}{\bf f}_{3}^{m+1}|\le C\delta(x_1)^{m-s-1},~1\le s\le m\quad\text{in}~\Omega_{2R}.
		\end{equation*}
		
		Denote
		\begin{equation*}
			{\bf v}^{m+1}(x)=\sum_{l=1}^{m+1}{\bf v}_{3}^{l}(x),\quad \bar{p}^{m+1}(x)=\sum_{l=1}^{m+1}\bar{p}_{l}(x).
		\end{equation*}
		It is easy to verify that ${\bf v}^{m+1}(x)={\bf u}_1^3(x)$ on $\Gamma_{2R}^{\pm}$, $\nabla\cdot{\bf v}^{m+1}(x)=0$ in $\Omega_{2R},$ and for $~1\le s\le m$,
		\begin{equation*}
			|{\bf f}^{m+1}(x)|=|\mu\Delta{\bf v}^{m+1}(x)-\nabla \bar{p}^{m+1}(x)|\leq C \delta(x_1)^{m-1},\quad |\nabla^s{\bf f}^{m+1}(x)|\leq C \delta(x_1)^{m-s-1}.
		\end{equation*}
		
		Thus, by Proposition \ref{prop3.3}, it holds that
		\begin{equation*}
			\|\nabla^{m+1}({\bf u}_{1}^{3} -{\bf v}^{m+1})\|_{L^{\infty}(\Omega_{\delta(x_1)/2}(x_1))} +\|\nabla^{m} (p_{1}^{3}-\bar{p}^{m+1})\|_{L^{\infty}(\Omega_{\delta(x_1)/2}(x_1))} \le C.
		\end{equation*}
		By \eqref{v31gradesti}, \eqref{v3lesti}, and \eqref{gradv3lest}, we have
		\begin{equation}\label{zgj5}
			|\nabla {\bf v}^{m+1}|\leq C\delta(x_1)^{-1}\quad \text{in}~\Omega_{R},
		\end{equation}
		and for $m\ge1$,
		\begin{equation}\label{zgj6}
			|\nabla^{m+1} {\bf v}^{m+1}|+|\nabla^m\bar{p}^{m+1}|\leq C\delta(x_1)^{-\frac{m+3}{2}}\quad \text{in}~\Omega_{R},
		\end{equation}
		Thus, by \eqref{zgj5} and \eqref{zgj6}, we have 
		$$|\nabla {\bf u}_1^3|\leq C\delta(x_1)^{-1}\quad \text{in}~\Omega_{R},$$
		and for $m\ge1$,
		\begin{equation*}
			|\nabla^{m+1} {\bf u}_{1}^{3}|+|\nabla^{m} {p}_{1}^{3}| \le C\delta(x_1)^{-\frac{m+3}{2}}
			\quad\text{in}~\Omega_{R}.
		\end{equation*}   
		To estimate of $p_{1}^{3}$, we fix a point $(z_1,0)$ with $|z_1|=R/2$. Then by the
		triangle inequality and the mean value theorem, for $x\in\Omega_{R}$,
		\begin{align*}
			|p_1^3(x)-p(z_1,0)|\leq&\, |(p_{1}^{3}-(\bar{p}_1+\bar{p}_2))(x)-(p_{1}^{3}-(\bar{p}_1+\bar{p}_2))(z_1,0)|+|(\bar{p}_1+\bar{p}_2)(x)|+C\nonumber\\
			\leq&\, \|\nabla (p_1^3-(\bar{p}_1+\bar{p}_2))\|_{L^\infty}+|(\bar{p}_1+\bar{p}_2)(x)|+C\le 
			C \delta(x_1)^{-3/2}.
		\end{align*} 
		This finishes the proof of Proposition \ref{prop2.3}.
	\end{proof}
	
	
	\section{Proofs of the main Theorems}\label{sec3}
	By using Proposition \ref{prop2.1}, \ref{prop2.1g}, \ref{prop2.2}, and \ref{prop2.3}, we now complete 
	the proofs of our main theorems.
	
	\subsection{Upper bounds for the higher derivatives}
	Recall that the rigid displacement space in dimension two is given by
	$$\Psi=\mathrm{span}\Bigg\{{\boldsymbol\psi}_{1}=\begin{pmatrix}
		1 \\
		0
	\end{pmatrix},
	{\boldsymbol\psi}_{2}=\begin{pmatrix}
		0\\
		1
	\end{pmatrix},
	{\boldsymbol\psi}_{3}=\begin{pmatrix}
		x_{2}\\
		-x_{1}
	\end{pmatrix}
	\Bigg\}.$$
	It follows from $e({\bf u})=0$ that
	$${\bf u}=\sum_{\alpha=1}^{3}C_{i}^{\alpha}{\boldsymbol\psi}_{\alpha}
	\quad\mbox{in}~D_i,\quad i=1,2,$$
	where $C_{i}^{\alpha}$ are constants to be determined by $({\bf u},p)$ 
	later. 
	
	We decompose the solution of \eqref{maineqs}  as follows:
	\begin{align}
		{\bf u}(x)=&\,\sum_{i=1}^{2}\sum_{\alpha=1}^{3}C_i^{\alpha}{\bf u}_{i}
		^{\alpha}(x)+{\bf u}_{0}(x),\quad 
		p(x)=\sum_{i=1}^{2}\sum_{\alpha=1}^{3}C_i^{\alpha}p_{i}^{\alpha}(x)+p_{0}
		(x),\quad x\in\,\Omega,\label{ud}
	\end{align}
	where ${\bf u}_{i}^{\alpha},{\bf u}_{0}\in{C}^{2}(\Omega;\mathbb R^2),~p_{i}
	^{\alpha}, p_0\in{C}^{1}(\Omega)$, respectively, satisfying
	\begin{equation}\label{equ_v12D}
		\begin{cases}
			\mu\Delta{\bf u}_{i}^\alpha=\nabla p_{i}^{\alpha},\quad\nabla\cdot 
			{\bf u}_{i}^{\alpha}=0&\mathrm{in}~\Omega,\\
			{\bf u}_{i}^{\alpha}={\boldsymbol\psi}_{\alpha}&\mathrm{on}~\partial{D}
			_{i},\\
			{\bf u}_{i}^{\alpha}=0&\mathrm{on}~\partial{D_{j}}\cup\partial{D},~j\neq i,
		\end{cases}\quad i=1,2,
	\end{equation}
	and
	\begin{equation}\label{equ_v32D}
		\begin{cases}
			\mu \Delta{\bf u}_{0}=\nabla p_0,\quad\nabla\cdot {\bf u}_{0}
			=0&\mathrm{in}~\Omega,\\
			{\bf u}_{0}=0&\mathrm{on}~\partial{D}_{1}\cup\partial{D_{2}},\\
			{\bf u}_{0}={\boldsymbol\varphi}&\mathrm{on}~\partial{D}.
		\end{cases}
	\end{equation}
	We  rewrite \eqref{ud} as
	\begin{equation*}
		{{\bf u}}=\sum_{\alpha=1}^{3}\left(C_{1}^{\alpha}-C_{2}^{\alpha}
		\right){\bf u}_{1}^{\alpha}
		+ {\bf u}_{b}~\mbox{and}~ 
		p=\sum_{\alpha=1}^{3}\left(C_{1}^{\alpha}-C_{2}^{\alpha}\right)p_{1}
		^{\alpha}+p_{b}\quad\mbox{in}~\Omega,
	\end{equation*}
	where
	\begin{equation*}
		{\bf u}_{b}:=\sum_{\alpha=1}^{3}C_{2}^{\alpha}({\bf u}_{1}^{\alpha}+{\bf u}
		_{2}^{\alpha})+{\bf u}_{0},\quad p_{b}:=\sum_{\alpha=1}^{3}C_{2}^{\alpha}(p_{1}
		^{\alpha}+p_{2}^{\alpha})+p_{0}.
	\end{equation*}
	Since ${\bf u}_{1}^\alpha+{\bf u}_{2}^\alpha-{\boldsymbol\psi}_\alpha=0$ on $
	\partial D_1\cup\partial D_2$, $\alpha=1,2,3$, it was proved in \cite[Proposition 2.4]{LX1} that, $\nabla\big({\bf u}_{1}^\alpha+{\bf u}_{2}^\alpha-
	{\boldsymbol\psi}_\alpha\big)$ has no singularity in the narrow region. In fact, 
	for $m\ge1$, $\nabla^m\big({\bf u}_{1}^\alpha+{\bf u}_{2}^\alpha\big)$ and $\nabla^{m}\big(p_{1}^{\alpha}+p_{2}^{\alpha})$ are also bounded in the narrow region.
	\begin{prop}\label{propu0}
		Let ${\bf u}_0,{\bf u}_{i}^{\alpha}\in{C}^{2,\gamma}(\Omega;\mathbb 
		R^2),~p_0,p_{i}^{\alpha}\in{C}^{1,\gamma}(\Omega)$ be the solution to 
		\eqref{equ_v12D} and \eqref{equ_v32D}. Then, for any $m\ge0$, we have
		\begin{equation*}
			\|\nabla^{m+1}({\bf u}_{1}^\alpha+{\bf u}_{2}^\alpha)\|_{L^{\infty}
				(\Omega_R)}+\|\nabla^{m}\big(p_{1}^{\alpha}+p_{2}^{\alpha}-(p_{1}^{\alpha}+p_{2}^{\alpha})(z_1,0)\big)\|_{L^{\infty}(\Omega_R)}\leq C
		\end{equation*}		
		and
		\begin{equation*}
			\|\nabla^{m+1}{\bf u}_{0}\|_{L^{\infty}(\Omega_R)}+\|\nabla^{m} \big(p_0-p_0(z_1,0)\big)\|
			_{L^{\infty}(\Omega_R)}\leq C
		\end{equation*}		
		for some point $z=(z_1, 0)\in\Omega_{R}$ with $|z_1|=R/2$.
	\end{prop}
	\begin{proof}
		By \eqref{equ_v12D} and \eqref{equ_v32D}, it is easy to see that $\big({\bf u}_{1}^\alpha+{\bf u}_{2}^\alpha-{\boldsymbol\psi}_\alpha, p_{1}^{\alpha}+p_{2}^{\alpha}\big)$, $\alpha=1,2,3$, and $({\bf u}_{0}, p_0)$ satisfy the following boundary value problem for the homogeneous Stokes equations in the narrow region
		\begin{equation*}
			\begin{cases}
				-\mu\Delta{\bf u}+\nabla p=0\quad&\mathrm{in}\ \Omega_{2R},\\
				\nabla\cdot {\bf u}=0\quad&\mathrm{in}\ \Omega_{2R},\\
				{\bf u}=0\quad&\mathrm{on}\ \Gamma^+_{2R}\cup\Gamma^-_{2R}.
			\end{cases}
		\end{equation*}
		By applying Proposition \ref{prop3.3} with ${\bf f}=0$, we have, for any $x=(x_1,x_2) \in \Omega_{R}$,
		\begin{equation*}
			\|\nabla{\bf u}\|_{L^{\infty}(\Omega_{\delta(x_1)/2}(x_1))}\leq C,
		\end{equation*}
		and for $m\ge 1$,
		\begin{equation*}
			\|\nabla^{m+1}{\bf u}\|_{L^{\infty}(\Omega_{\delta(x_1)/2}(x_1))} +\|
			\nabla^{m}p\|
			_{L^{\infty}(\Omega_{\delta(x_1)/2}(x_1))}\le C.
		\end{equation*}
		Thus, for any $m\ge 1$, 
		\begin{equation}\label{yyd1}
			\|\nabla{\bf u}\|_{L^{\infty}(\Omega_R)}\le C,\quad \|\nabla^{m+1}{\bf u}\|_{L^{\infty}(\Omega_R)}+\|\nabla^{m} p\|
			_{L^{\infty}(\Omega_R)}\leq C.
		\end{equation}		
		
		For the estimates of $p$, for a fixed point $z=(z_1,0)\in\Omega_{R}$ with $|
		z_1|=R$, by the mean value theorem and \eqref{yyd1}, we have
		\begin{equation}\label{yyd12}
			|p(x)-p(z_1,0)|\leq C\|\nabla p\|_{L^\infty(\Omega_{R})}\leq C,\quad x\in\Omega_{R}.
		\end{equation}
		It follows from \eqref{yyd1} and \eqref{yyd12} that for any $m\ge 0$,
		\begin{equation*}
			\|\nabla^{m+1}{\bf u}\|_{L^{\infty}(\Omega_R)}+\|\nabla^{m} \big(p-p(z_1,0)\big)\|
			_{L^{\infty}(\Omega_R)}\leq C.
		\end{equation*}
		This completes the proof.
	\end{proof}

	For the coefficients $C_i^\alpha$, $i=1,2$, $\alpha=1,2,3$, which depend only on the gradient estimates, we have the following result from \cite[Proposition 2.7]{LX1}.
	
	\begin{prop}\label{lemCialpha}
		Let $C_{i}^{\alpha}$ be defined in \eqref{ud}. Then
		$$|C_i^{\alpha}|\leq\,C,\quad\,i=1,2,~\alpha=1,2,3,$$
		and
		\begin{equation*}
			|C_1^1-C_2^1|\leq C\sqrt{\varepsilon},\quad |C_1^2-C_2^2|\leq 
			C\varepsilon^{3/2},\quad|C_1^3-C_2^3|\leq C\sqrt{\varepsilon}.
		\end{equation*}
	\end{prop}
	
	We now prove the main result, Theorem \ref{main thm1n}. 
	\begin{proof}[Proof of Theorem \ref{main thm1n}]
		By virtue of Propositions \ref{prop2.1}, \ref{prop2.2}, \ref{prop2.3}, \ref{propu0}, and \ref{lemCialpha}, we have, for $x\in\Omega_{R}$,
		\begin{equation*}
			|\nabla{\bf u}(x)|\le \sum_{\alpha=1}^{3}|C_{1}^{\alpha}-C_{2}^{\alpha}||
			\nabla{\bf u}_{1}^{\alpha}(x)|+C\le\frac{C\sqrt{\varepsilon}}{\delta(x_1)}+\frac{C\varepsilon^{3/2}}{\delta(x_1)^{\frac{3}{2}}}+\frac{C\sqrt{\varepsilon}}{\delta(x_1)}+C\le 
			\frac{C\sqrt{\varepsilon}}{\delta(x_1)}+C, 
		\end{equation*}
		and
		\begin{align}\label{p-pz_1esti}
			|p(x)-p(z_1,0)|\leq&\, \sum_{\alpha=1}^{3}|(C_{1}^{\alpha}-C_{2}^{\alpha})
			(p_{1}^{\alpha}(x)-p_{1}^{\alpha}(z_1,0))|+C \nonumber\\
			\leq&\,\frac{C\varepsilon^{3/2}}{\delta(x_1)^2}+\frac{C\sqrt{\varepsilon}}
			{\delta(x_1)^{\frac{3}{2}}}+C \le \frac{C\sqrt{\varepsilon}}
			{\delta(x_1)^{\frac{3}{2}}}+C,
		\end{align}
		where $(z_1,0)\in\Omega_{R}$ with $|z_{1}|=R$ is a fix point.
		
		For higher-order derivative estimates, when $m\ge1$,
		\begin{align*}
			&|\nabla^{m+1}{\bf u}(x)|\leq\, \sum_{\alpha=1}^{3}|C_1^\alpha-C_2^\alpha||
			\nabla^{m+1}{\bf u}_{1}^\alpha(x)|+C\nonumber\\
			\leq& \, C\sqrt{\varepsilon}\delta(x_1)^{-\frac{m+3}{2}}+C\varepsilon^{\frac{3}
				{2}}\delta(x_1)^{-\frac{m+4}{2}}+C\sqrt{\varepsilon}\delta(x_1)^{-\frac{m+3}{2}}
			+C\leq C\sqrt{\varepsilon}\delta(x_1)^{-\frac{m+3}{2}}+C,
		\end{align*}
		and
		\begin{align*}
			|\nabla^{m}p(x)|\leq&\, \sum_{\alpha=1}^{3}|C_{1}^{\alpha}-C_{2}^{\alpha}||
			\nabla^{m}p_{1}^{\alpha}|+C\nonumber\\\leq&\, C\sqrt{\varepsilon}\delta(x_1)^{-
				\frac{m+3}{2}}+C\varepsilon^{3/2}\delta(x_1)^{-\frac{m+4}{2}}+
			C\sqrt{\varepsilon}\delta(x_1)^{-\frac{m+3}{2}}+C\leq C\sqrt{\varepsilon}\delta(x_1)^{-\frac{m+3}{2}}+C.
		\end{align*}
		This completes the proof of Theorem \ref{main thm1n}.
	\end{proof}
	
	\begin{proof}[Proof of Theorem \ref{main thm1}]
		The proof for the estimates of ${\bf u}$ and $p$ in Theorem  \ref{main thm1} is similar to that of Theorem \ref{main thm1n}. We only give a proof for the Cauchy stress $\sigma[{\bf u},p-p(z_1,0)]$. By Proposition \ref{prop2.1g} and \eqref{cxiangdeng}, \eqref{p-pz_1esti} becomes
		\begin{align*}
			|p(x)-p(z_1,0)|\leq&\, \sum_{\alpha=1}^{3}|(C_{1}^{\alpha}-C_{2}^{\alpha})
			(p_{1}^{\alpha}(x)-p_{1}^{\alpha}(z_1,0))|+C \nonumber\\
			\leq&\,\frac{C\sqrt{\varepsilon}}{\sqrt{\delta(x_1)}}+
			\frac{C\varepsilon^{3/2}}{\delta(x_1)^2}+C \le \frac{C\sqrt{\varepsilon}}{\delta(x_1)}+C. 
		\end{align*}
		Hence, by Propositions \ref{prop2.1g} and \ref{prop2.2}, we have for $m\ge 0$,
		\begin{align*}
			&|\nabla^{m}\sigma[{\bf u},p-p(z_1,0)]| \\
			\leq&\, \sum_{\alpha=1}^{2}|C_1^\alpha-
			C_2^\alpha| |\nabla^{m+1}{\bf u}_{1}^\alpha(x)|+\sum_{\alpha=1}^{2} |(C_{1}
			^{\alpha}-C_{2}^{\alpha})\nabla^{m}(p_{1}^{\alpha}-p_{1}^{\alpha}(z_1,0))|
			+C\nonumber\\\leq& \, C\sqrt{\varepsilon}\delta(x_1)^{-\frac{m+2}{2}}
			+C\varepsilon^{3/2}\delta(x_1)^{-\frac{m+4}{2}}+C\leq \frac{C\sqrt{\varepsilon}}{\delta(x_1)^{\frac{m+2}{2}}}+C.
		\end{align*}
		The proof of Theorem \ref{main thm1} is finished.
	\end{proof}
	
	\subsection{Lower bounds for the higher derivatives}
	
	Under the assumptions of Theorem \ref{thmlowerbound}, it has been shown in \cite[Propositions 5.1 and 5.4]{LX1} that 
	\begin{equation}\label{cxiangdeng}
		C_1^3=C_2^3.
	\end{equation}
	\begin{proof}[Proof of Theorem \ref{thmlowerbound}]
		By \cite[(5.37)]{LX1}, we have
		\begin{equation}\label{C1-C2bound}
			|C_{1}^{1}-C_{2}^{1}|\ge C\tilde{b}_{1}^{*1}[{\boldsymbol{\varphi}}]
			\sqrt{\varepsilon}.
		\end{equation}
		By Proposition \ref{lemCialpha}, \eqref{xjyy}, and \eqref{v2lest}, we have 
		\begin{equation}\label{C1-C2partu1211}
			|(C_{1}^{2}-C_{2}^{2})\partial_{x_{1}}^{m}\partial_{x_2}({\bf u}_{1}^{2})^{(1)}|
			\le C\varepsilon^{\frac32}\sum_{i=1}^{m+1}|\partial_{x_{1}}^{m}\partial_{x_2}({\bf v}_{2}^{i})^{(1)}|\leq C\varepsilon^{-\frac{m}{2}}.
		\end{equation}
		It follows from \eqref{v111defg} that for some small $r>0$, which may depend on $m$,
		\begin{equation*}
			|\partial_{x_1}^{m}\partial_{x_{2}}({\bf v}_{1}^{1})^{(1)}(r\sqrt{\varepsilon},0 )|=|\partial_{x_1}^{m}(\delta(x_1)^{-1})(r\sqrt{\varepsilon},0 )|\ge C\varepsilon^{-\frac{m+2}{2}},
		\end{equation*}
		and by \eqref{highvpfesti},
		\begin{equation*}
			\sum_{i=2}^{m+1}|\partial_{x_1}^{m}\partial_{x_{2}}({\bf v}_{1}^{i})^{(1)}|\le C\varepsilon^{-\frac{m}{2} }.
		\end{equation*}
		Hence,
		\begin{equation}\label{lower2}
			|\partial_{x_1}^{m}\partial_{x_2}({\bf u}_{1}^{1})^{(1)}|(r\sqrt{\varepsilon},0)\ge\,C\varepsilon^{-\frac{m+2}{2}}.
		\end{equation}
		Combining  \eqref{C1-C2bound}--\eqref{lower2}, we have 
		\begin{align*}
			|\partial_{x_1}^{m}\partial_{x_2}{\bf u}^{(1)}|(r\sqrt{\varepsilon},0)=&\,|
			\sum_{\alpha=1}^{2}(C_{1}^{\alpha}-C_{2}^{\alpha})\partial_{x_1}^{m}
			\partial_{x_2}({\bf u}_{1}^{\alpha})^{(1)}+\partial_{x_1}^{m}\partial_{x_2}{\bf 
				u}_{b}|(r\sqrt{\varepsilon},0)\nonumber\\
			\geq&\,\frac{1}{2}|(C_{1}^{1}-C_{2}^{1})\partial_{x_1}^{m}\partial_{x_2}({\bf u}_{1}^{1})^{(1)}|-C\varepsilon^{-\frac{m}{2}}\ge C|\tilde{b}_{1}^{*1}[\boldsymbol{\varphi}]|
			\varepsilon^{-\frac{m+1}{2}}.
		\end{align*}
		
		For the lower bound of the Cauchy stress $\sigma[{\bf u},p-p(z_1,0)]$, by using \eqref{C1-C2bound}--\eqref{lower2}, we have
		\begin{align*}
			&|\partial_{x_{1}}^m\sigma[{\bf u},p-p(z_1,0)]|(r\sqrt{\varepsilon},0)=\Big|\partial_{x_1}^m\Big(2\mu e( {\bf u} ) - (p-p(z_1,0)) \mathbb{I}\Big)\Big|
			(r\sqrt{\varepsilon},0)\\ 
			\geq&\,\Big|\sum_{\alpha=1}^{2}(C_{1}^{\alpha}-
			C_{2}^{\alpha})\partial_{x_1}^{m}\Big(2\mu e({\bf u}_{1}^{\alpha})-(p_{1}
			^{\alpha}-p_{1}^{\alpha}(z_1,0))\mathbb{I}\Big) \Big|(r\sqrt{\varepsilon},0)-C\\
			\geq&\, C|(C_{1}^{1}-C_{2}^{1})\partial_{x_{1}}^me_{12}({\bf u}_{1}^{1})|(r\sqrt{\varepsilon},0)-C\varepsilon^{-\frac m 2}-
			C\ge C|\tilde{b}_{1}^{*1}[\boldsymbol{\varphi}]|\varepsilon^{-\frac{m+1}{2}}.
		\end{align*} 
		This completes the proof of Theorem \ref{thmlowerbound}.      
	\end{proof}

	\section*{Declarations}
	
	\noindent{\bf Data availability.} Data sharing not applicable to this article as no data sets were generated or analyzed during the current study.
	
	\noindent{\bf Financial interests.} The authors have no relevant financial or non-financial interests to disclose.
	
	\noindent{\bf Conflict of interest.} The authors declare that they have no conflict of interest.

\end{document}